\documentclass[12pt,a4paper]{amsart}
\usepackage{amssymb}
\usepackage{amsfonts}
\usepackage{amsmath}
\usepackage[mathscr]{eucal}

\usepackage{color}

\newtheorem{thm}{Theorem}[section]
\newtheorem{prop}[thm]{Proposition}
\newtheorem{lem}[thm]{Lemma}
\newtheorem{cor}[thm]{Corollary}
\numberwithin{equation}{section}

\def\N{{\Bbb N}}
\def\Z{{\Bbb Z}}
\def\Q{{\Bbb Q}}
\def\R{{\Bbb R}}
\def\C{{\Bbb C}}
\def\A{{\Bbb A}}
\def\bT{{\Bbb T}}

\def\emp{\varnothing}

\def\fa{{\frak a}}
\def\fb{{\frak b}}
\def\fd{{\frak d}}

\def\ff{{\frak f}}
\def\fg{{\frak g}}

\def\fm{{\frak m}}
\def\fn{{\frak n}}

\def\fp{{\frak p}}

\def\fU{{\frak U}}
\def\fS{{\frak S}}

\def\fI{{\frak I}}

\def\fK{{\frak K}}

\def\fc{{\frak c}}

\def\fX{{\frak X}}

\def\cO{\frak o}

\def\cH{{\mathscr H}}

\def\cQ{{\mathcal Q}}

\def\cG{{\mathcal G}}

\def\GL{{\operatorname {GL}}}

\def\SL{{\operatorname{SL}}}
\def\SO{{\operatorname{SO}}}

\def\PSL{{\operatorname {PSL}}}
\def\PGL{{\operatorname{PGL}}}

\def\Re{{\operatorname {Re}}}
\def\Im{{\operatorname {Im}}}
\def\tr{{\operatorname{tr}}}
\def\nr{{\operatorname{N}}}

\def\sgn{{\operatorname {sgn}}}
 
\def\vol{{\operatorname{vol}}}

\def\ch{{\cosh\,}}
\def\sh{{\sinh\,}}

\def\leq{\leqslant}
\def\geq{\geqslant}
\def\bsl{\backslash}
\def\le{\leq}
\def\ge{\geq}

\def\d {{{d}}}

\def\JJ{{\Bbb J}}
\def\bs{{\bold s}}
\def\bK{{\bold K}}
\def\bC{{\bold C}}
\def\1{{\bold 1}}
\def\ccA{{\mathscr A}}

\def\prec{\ll}

\renewcommand{\a}{\alpha}
\renewcommand{\b}{\beta}

\newcommand{\e}{\epsilon}

\renewcommand{\l}{\lambda}

\newcommand{\s}{\sigma}

\newcommand{\ga}{{\mathfrak{a}}}

\newcommand{\gf}{{\mathfrak{f}}}

\newcommand{\gn}{{\mathfrak{n}}}
\newcommand{\go}{{\mathfrak{o}}}
\newcommand{\gp}{{\mathfrak{p}}}

\newcommand{\gS}{{\mathfrak{S}}}
\newcommand{\gU}{{\mathfrak{U}}}

\newcommand{\Bcal}{{\mathcal B}}

\newcommand{\Gcal}{{\mathcal G}}

\newcommand{\Ocal}{{\mathcal O}}

\renewcommand{\AA}{\mathbb{A}}

\newcommand{\CC}{\mathbb{C}}

\newcommand{\II}{\mathbb{I}}
\newcommand{\LL}{\mathbb{L}}
\newcommand{\NN}{\mathbb{N}}
\newcommand{\PP}{\mathbb{P}}
\newcommand{\QQ}{\mathbb{Q}}
\newcommand{\RR}{\mathbb{R}}
\newcommand{\TT}{\mathbb{T}}
\newcommand{\ZZ}{\mathbb{Z}}

\newcommand{\bfc}{{\mathbf c}}

\newcommand{\bfs}{{\mathbf s}}

\newcommand{\bfy}{{\mathbf y}}

\newcommand{\bfK}{{\mathbf K}}

\newcommand{\bfT}{{\mathbf T}}

\newcommand{\ord}{\operatorname{ord}}

\newcommand{\fin}{{\rm fin}}
\renewcommand{\Re}{\operatorname{Re}}
\newcommand{\Lie}{{\operatorname{Lie}}}
\newcommand{\Res}{\operatorname{Res}}

\def\Sd{P}
\def\ul{{\underline l}}

\title{Relative trace formulas and subconvexity estimates for $L$-functions of Hilbert modular forms}

\author{Shingo Sugiyama}
\author{Masao Tsuzuki}

\subjclass[2010]{Primary 11F67; Secondary 11F72.}

\begin{document}

\maketitle

\begin{abstract}
 We elaborate an explicit version of the relative trace formula on $\PGL(2)$ over a totally real number field for the toral periods of Hilbert cusp forms along the diagonal split torus. As an application, we prove (i) a spectral equidistribution result in the level aspect for Satake parameters of holomorphic Hilbert cusp forms weighted by central $L$-values, and (ii) a bound of quadratic base change $L$-functions for Hilbert cusp forms with a subconvex exponent in the weight aspect.\\

Keyword: relative trace formulas, central $L$-values, subconvexity estimates.
\end{abstract}


\section{Introduction}
In this paper, by the method developed in \cite{Tsud} and \cite{Sugiyama2}, we explicitly compute Jacquet's relative trace formula for the toral periods along the diagonal split torus in $\PGL(2)$, which encodes the central $L$-values for the quadratic base change of holomorphic Hilbert cusp forms in its spectral side. By introducing new techniques in a broader setting, we elaborate an explicit relative trace formula partly obtained by Ramakrishnan and Rogawski (\cite{Ramakrishnan-Rogawski}) for the elliptic modular case. Let $F$ be a totally real number field and $\AA$ its adele ring. Let $\Sigma_\infty$ denote the set of archimedean places of $F$ and $\Sigma_\fin$ the set of finite places of $F$. We consider a family of positive even integers $l=(l_v)_{v\in \Sigma_\infty}$, calling it a {\it weight}. Given a weight $l$ and an integral ideal $\fn$ of $F$, let $\Pi_{\rm{cus}}(l,\fn)$ be the set of all the irreducible cuspidal representations $\pi$ of the adele group $\PGL(2,\A)$ such that its $v$-th local component $\pi_v$ is isomorphic to the discrete series representation of $\PGL(2,\R)$ of weight $l_v$ if $v\in \Sigma_\infty$ and possesses a non-zero vector invariant by the local Hecke congruence subgroup 
$$
\bK_0(\fn\cO_v)=\left\{\left[\begin{smallmatrix} a & b \\ c & d \end{smallmatrix}\right]
\in \GL(2,\cO_v)|\,c\in \fn\cO_v\,\right\},
$$
if $v\in \Sigma_\fin$, where $\cO$ is the maximal order of $F$ and $\cO_v$ its completion at $v$. 
The standard $L$-function $L(s,\pi)$ of $\pi \in \Pi_{\rm{cus}}(l,\fn)$ is defined to be the Euler product of local factors $L(s,\pi_v)$ over all places $v$ if $\Re(s)$ is sufficiently large. Recall that for a finite place $v$ not dividing the ideal $\fn$,  
$$
L(s,\pi_v)=(1-q_{v}^{\nu(\pi_v)/2-s})^{-1}(1-q_{v}^{-\nu(\pi_v)/2-s})^{-1},
$$
where $q_v$ is the order of the residue field at $v$ and $q_v^{\pm \nu(\pi_{v})/2}$ is the Satake parameter of $\pi$ at $v$.
We remark that the $L$-function in our sense is the completed one by the archimedean local factors $L(s,\pi_v)=\Gamma_\C(s+(l_v-1)/2)$ for $v\in \Sigma_\infty$, and the Euler product with all the gamma factors removed from $L(s,\pi)$ is denoted by $L_\fin(s,\pi)$. It is well known that $L(s,\pi)$, originally defined on a half-plane $\Re(s)\gg 0$, has a holomorphic continuation to the whole complex plane with the self-dual functional equation $L(s,\pi)=\epsilon(s,\pi)\,L(1-s,\pi)$. The central value $L_\fin(1/2,\pi)$ and its twist $L_\fin(1/2,\pi\otimes\eta)$ by a real valued idele class character $\eta$ of $F^\times$ have been studied extensively from several different points of view. For example, when the base field $F$ is $\Q$, Iwaniec and Sarnak \cite{IwaniecSarnak} announced a number of asymptotic formulas of the 1st and the 2nd moments of the central $L$-values $L_\fin(1/2,\pi\otimes\eta)$ for $\pi \in \Pi_{\rm{cus}}(l,\gn)$ twisted by the Hecke operators and by suitably designed mollifiers. Combining such asymptotic formulas, they proved that among $L$-functions whose functional equation has even sign, 50 percent vanish at the central point in a quantitative sense as the weight $l$ (or the square-free level $\fn$) grows; moreover, they claimed that the quantitative nonvanishing of more than 50 percent of them eliminates the possibility of Landau-Siegel zero of the quadratic Dirichlet $L$-function $L(s,\eta)$. Among several twisted means of $L$-values considered in \cite{IwaniecSarnak}, one of the most basic means is
\begin{align}
\sum_{\pi \in \Pi_{\rm{cus}}(l,\fn)} \,\frac{L(1/2,\pi)\,L(1/2,\pi\otimes \eta)}{L(1,\pi,{\rm{Ad}})}\alpha(\nu_S(\pi)),
 \label{2ndMoment}
\end{align}
where $S$ is a finite set of finite places coprime to both $\fn$ and the conductor $\ff$ of $\eta$, $\nu_S(\pi)=\{{\nu(\pi_v)}\}_{v\in S}$ is the collection of the exponent in the Satake parameters of $\pi$ over $S$, and $\alpha(\{\nu_v\}_{v\in S})$ is a polynomial of the functions $q_v^{-\nu_v/2}+q_v^{\nu_v/2}$ in the variable $\nu_v$. When $F=\Q$, Ramakrishnan and Rogawski (\cite{Ramakrishnan-Rogawski}) studied the asymptotic behavior of the twisted 2nd moment \eqref{2ndMoment} for an odd Dirichlet character $\eta$ as the level $\fn$, to be kept prime and coprime to $\ff$, grows to infinity when the weight $l\geq 4$ is fixed. In the same setting, Michel and Ramakrishnan (\cite{MichelRamakrishinan}) obtained an explicit closed formula for the average \eqref{2ndMoment} and observed that the formula gets simplified significantly in a certain range of the parameters $(\fn, \ff, {\rm{deg}}(\alpha))$ called the stable range. Later, Feigon and Whitehouse (\cite{FeigonWhitehouse}) extended the result of \cite{Ramakrishnan-Rogawski} and \cite{MichelRamakrishinan} to the Hilbert modular forms in a more general setting as \cite{Jacquet-Chen} but still keeping the square-free condition on the level $\fn$ and the oddness conditions on $\eta$ at all archimedean places. In this article, we consider the twisted 2nd moment \eqref{2ndMoment} in our general Hilbert modular setting without assuming those conditions on $\fn$ and $\eta$, and obtain its formula in a computable form (Theorems~\ref{RELATIVETRACEFORMULA}, \ref{HYPERBOLIC(fine)} and \ref{UNIPOTENT(fine)}); thus, we generalize some results of \cite{FeigonWhitehouse} and \cite{MichelRamakrishinan} in several directions. 

As a first application of our formula, we prove an equidistribution theorem of the Satake parameters weighted by the central $L$-values $L(1/2,\pi)\,L(1/2,\pi\otimes \eta)$ as in \cite{FeigonWhitehouse}; we work with a more general sign condition on $\eta$ than \cite{FeigonWhitehouse} at archimedean places, allowing the level $\fn$ to be a general ideal not necessarily square-free. We remark that a similar asymptotic result for spectral averages of $L$-values of non-holomorphic modular forms was proved first by \cite{Tsud} for square-free levels and later by \cite{Sugiyama2} for arbitrary levels. In what follows, $\nr(\fn)$ denotes the absolute norm of an ideal $\fn\subset \cO$. 

\begin{thm}\label{ED-THM} Let $l=(l_v)_{v\in \Sigma_\infty}$ be a weight such that $l_v\geq 6$ for all $v\in \Sigma_\infty$. Let $\fn$ be an ideal of $\cO$ and $\eta=\otimes_v\eta_v$ a quadratic idele class character of $F^\times$ with conductor $\ff$ prime to $\gn$, and $S$ a finite  set of finite places relatively prime to $\fn\ff$. Assume that $\eta_{v}$ is non trivial for any prime divisor $v$ of $\gn$, and that $(-1)^{\e(\eta)}\tilde\eta(\gn)=1$ where $\e(\eta)$ is the number of $v\in \Sigma_\infty$ such that $\eta_v$ is non trivial and $\tilde \eta$ is the character of the group of ideals prime to $\ff$ induced by $\eta$. Then, for any even holomorphic function $\alpha(\bfs)$ on the complex manifold $\fX_S=\prod_{v\in S}(\C/\frac{4\pi i }{\log q_v}\Z)$, we have the asymptotic formula
{\allowdisplaybreaks 
\begin{align*}
&
\ (2 \pi)^{[F: \QQ]}\{ \prod_{v \in \Sigma_{\infty}}\frac{(l_{v}-2)!} {\{(l_{v}/2-1)!\}^{2}} \}\times
\frac{1}{{\rm N}(\gn)} \sum_{\pi \in \Pi_{\rm cus}^{*}(l,\gn)}
\frac{L(1/2, \pi) L(1/2, \pi \otimes \eta)}
{L^{S_{\pi}}(1, \pi; {\rm Ad})} \a(\nu_{S}(\pi)) \\
= & 
4 D_{F}^{3/2} \nu(\gn) L_{\fin}(1, \eta) \int_{\fX_S^{0} }\a(\bfs) d\mu_{S}^{\eta}(\bfs) + \Ocal_{\e, l, \eta, \a}({\rm N}(\gn)^{-\delta}),
\end{align*}
}for some $\delta>0$; if $\gn$ is restricted to square-free ideals, then the asymptotic formula is true with a smaller error term $\Ocal_{\e, l, \eta, \a}(\nr(\gn)^{-\inf_{v\in \Sigma_\infty}l_v/2+1+\e})$ for any $\e>0$.  Here on the right-hand side of the formula,$$
\nu(\gn)=\prod_{\substack{v\in \Sigma_\fin \\ \ord_{v}(\gn) \ge 3}}(1-q_v^{-2})\prod_{\substack{v\in \Sigma_\fin \\ \ord_v(\gn)=2}}(1-(q_v^2-q_v)^{-1}),$$
$\fX_S^0$ denotes the purely imaginary locus of $\fX_S$ and $\d\mu_S^{\eta}(i\bfy)=\prod_{v\in S}\d \mu_v^{\eta_v}(iy_v)$ with
{\allowdisplaybreaks 
\begin{align*}
\d \mu_v^{\eta_v}(iy_v)
=
\begin{cases}
\dfrac{q_{v}-1}{(q_{v}^{1/2}+q_{v}^{-1/2}-x_v)^2}\,\d\mu^{\rm{ST}}(x_v), \quad &(\eta_v(\varpi_v)=+1), \\
\dfrac{q_{v}+1}{(q_{v}^{1/2}+q_{v}^{-1/2})^2-x_v^2}\,\d\mu^{\rm{ST}}(x_v), \quad &(\eta_v(\varpi_v)=-1),
\end{cases}
\end{align*}
}where $x_v=q_v^{iy_v/2}+q_v^{-iy_v/2}$, $\d\mu^{\rm{ST}}(x)=\frac{\sqrt{4-x^2}}{2\pi}\,\d x$ is the Sato-Tate measure
and $\varpi_{v}$ is a prime element of $\go_{v}$. On the left-hand side of the formula, $\Pi_{\rm{cus}}^*(l,\fn)$ denotes the set of $\pi\in \Pi_{\rm{cus}}(l,\fn)$ whose conductor $\ff_\pi$ is $\fn$, and $L^{S_\pi}(s,\pi;{\rm{Ad}})$ is the adjoint square $L$-function of $\pi$, whose local $v$-factors are removed for all $v$ belonging to $S_\pi=\{v\in \Sigma_\fin \ | \ {\rm{ord}}_v(\ff_\pi) \geq 2\,\}$.  
\end{thm}
We remark that our relative trace formula yields an exact formula for the $L$-value average \eqref{2ndMoment}, which reduces to a finite expression for some $(\fn,\eta,\a)$ (see Corollary \ref{stability}). As a corollary to this theorem, we have the following result (cf.\ \cite[Corollary B]{Ramakrishnan-Rogawski}, \cite[Theorem 3]{Sugiyama2}).

\begin{cor}
Let $l=(l_v)_{v\in \Sigma_\infty}$ be a weight such that $l_v\geq 6$ for all $v\in \Sigma_\infty$. Let $\eta$ be a quadratic idele class character of $F^\times $ with conductor $\ff$. Let $S$ be a finite set of finte places relatively prime to $\ff$ and $\{J_v\}_{v\in S}$ a collection of subintervals of $[-2,2]$. Given a sequence of $\cO$-ideals $\{\fn_k\}_{k\in \N}$ relatively prime to $\ff$ and $S$ such that $(-1)^{\e(\eta)}\tilde\eta(\fn_k)=+1$, $\eta_{v}(\varpi_{v})= -1$ for all prime divisors $v$ of $\gn_{k}$ and $\lim_{k\rightarrow \infty} \nr(\fn_k)=+\infty$, there exists $k_0$ with the following property: For any $k\geq k_0$, there exists $\pi \in \Pi_{\rm{cus}}^*(l,\fn_k)$ such that $L_{\fin}(1/2,\pi)L_{\fin}(1/2,\pi\otimes\eta)\not=0$ and $q_v^{\nu(\pi_v)/2}+q_v^{-\nu(\pi_v)/2} \in J_v$ for all $v\in S$. 
\end{cor}

The so called convexity bound of $L_\fin(1/2,\pi)$ is
$$|L_\fin(1/2,\pi)| \ll_\epsilon \{\nr(\fn)\prod_{v\in \Sigma_\infty}l_v^2\}^{1/4+\epsilon}, \qquad \pi \in \Pi_{\rm{cus}}(l,\fn)
$$
 for any $\epsilon>0$. When $F=\Q$ (so the weight $l$ is a number) and $\gn =\ZZ$ , the bound $|L_\fin(1/2,\pi)| \ll_{\e} l^{1/3+\epsilon}$, which breaks the convexity bound in the weight aspect, has long been known (\cite{Peng}, \cite{JutilaMotohashi}). Thanks to a recent result by Michel and Venkatesh \cite{MichelVenkatesh}, existence of a subconvexity bound for $L_\fin(1/2,\pi)$ in any aspect in the general setting is now known; however, beyond its existence, an explicit form of the subconvex exponent is not obvious in their work. As a second application of our formula, we deduce a bound with an explicit subconvex exponent in the {\it weight aspect} for the $L$-function $L_\fin(1/2,\pi)\,L_\fin(1/2,\pi\otimes \eta)$ with $\eta$ an idele class character of $F^\times$ which is odd at all archimedean places, where $F$ is a general totally real number field.  

\begin{thm}\label{T4}
Let $l=(l_{v})_{v \in \Sigma_{\infty}}$ be a weight such that $l_{v} \ge 6$ for all $v \in \Sigma_{\infty}$.
Let $\fn$ be an arbitrary ideal of $\cO$ and $\eta$ a real valued idele class character of $F^\times$ such that $\eta_v(-1)=-1$ for all $v\in \Sigma_\infty$.
Suppose that the conductor $\gf$ of $\eta$ is relatively prime to $\gn$.
Then, for any $\e>0$, 
\begin{align*}
|L_{\fin}(1/2,\pi)\,L_\fin(1/2,\pi \otimes \eta)| \ll_{\epsilon} \nr(\ff)^{3/4
+\epsilon} \nr(\fn)^{1+\epsilon}\,\{\prod_{v\in \Sigma_\infty}l_v\}^{7/8+\epsilon}
\end{align*}
with the implied constant independent of $l=(l_v)_{v\in \Sigma_\infty}$, $\fn$, $\eta$ and $\pi\in \Pi_{\rm{cus}}(l,\fn)$. 
\end{thm}

Theorems~\ref{ED-THM} and \ref{T4} are obtained by the relative trace formula stated in \S 9 together with the explicit formulas of local terms given in \S 10 and \S 11. Technically speaking, there are substantial differences between \cite{Ramakrishnan-Rogawski}, \cite{FeigonWhitehouse} and ours in the way explained below. In \cite{FeigonWhitehouse}, by using the Jacquet-Langlands transfer and some refinements of Waldspurger's period formula, when the character $\eta$ is odd at all archimedean places, the equidistribution theorem is deduced from the relative trace formula developed by \cite{Jacquet-Chen} and explicated by the authors of \cite{FeigonWhitehouse} for an anisotropic inner form of $\GL(2)$ which is certainly an easier place to do analysis than $\GL(2)$. Contrary to this, like in \cite{Ramakrishnan-Rogawski}, we perform an explicit computation of the relative trace formula on $\GL(2)$ for the period along the split torus, which is slightly harder analytically due to the non-compactness of the spaces but much easier algebraically because we only have to consider the Hecke's zeta integral in the spectral side. The analytical difficulty can be resolved by the technique developed in \cite{Tsud} (see 6.3 and 6.5. For a different approach, we refer to \cite{Ramakrishnan-Rogawski}). For the algebraic aspect, we have an advantage from \cite{Sugiyama1} which completed the computation of local Hecke's zeta integrals for local old forms. Due to the direct nature of the method, we can rather easily drop several local and global constraints on automorphic representations and the character $\eta$ which is essential to move to an anisotropic group by the Jacquet-Langlands transfer. For example, the character $\eta$ is allowed to be trivial in our work. Moreover, the usage of the ``Shintani functions'' (see \S3.2) in place of the matrix coefficients of discrete series at archimedean places simplifies some computation of the archimedean orbital integrals compared with \cite{Ramakrishnan-Rogawski}. Similarly, the usage of the ``Green functions'' (see \S 4) at finite places makes it possible for us to compute the non-archimedean orbital integrals directly; our result is thorough in the sense that it covers not only the unit element of the spherical Hecke algebra but also all of its elements.

This article is organized as follows. After a preliminary section, in \S 3, we recall the definition of the Shintani functions for the symmetric pair $(\GL(2,\R),T)$ studied by \cite{Hirano} with $T$ being the diagonal torus, and prove several properties of them necessary later. In \S 4, we briefly review about the Green function on $\GL(2)$ over a non-archimedean local field, which was introduced in \cite[\S 5]{Tsud}. Combining these, in \S 5, we define a left $H_\A$-equivariant smooth function on the adele group $\GL(2, \A)$ with the ``reproducing property" (Lemma~\ref{Green and period}), calling it the adelic Green function. Here $H$ denotes the diagonal split torus of $\GL(2)$. In \S 6, after reviewing the explicit formulas of the toral period integrals of $\GL(2)$ cusp forms with arbitrary level given by \cite{Sugiyama1}, we compute the spectral expansion of the automorphic renormalized kernel \eqref{centralObj}, which is constructed by forming the sum of the adelic Green function translated by $H_F\bsl \GL(2, F)$ after a regularization to compensate ${\rm{vol}}(Z_{\AA}H_F\bsl H_\A)=\infty$, where $Z$ is the center of $\GL(2)$.
Although such a regularization is not needed in the spectral side since \eqref{centralObj} is cuspidal,
the regularization plays a role in the geometric side in \S8.
In \S 7 and \S 8, closely following \cite[\S 12]{Tsud}, we compute the geometric expression of the period integral of the automorphic renormalized kernel. Up to \S 8, most of the estimates and computations are obtained from the corresponding ones in \cite{Tsud} and \cite{Sugiyama2} by modification at archimedean places; we make the proofs as brief as possible by leaving detailed arguments to our previous works. In the final formula (Theorem~\ref{RELATIVETRACEFORMULA}), two linear functionals $\tilde \JJ^\eta_{{\rm{u}}}(l, \fn|\alpha)$ and $\JJ^\eta_{\rm{hyp}}(l, \fn|\alpha)$ in the test function $\alpha$ arise. We can deduce Theorem~\ref{ED-THM} easily from Theorem~\ref{RELATIVETRACEFORMULA} as explained in \S 9; the point is to show that the term $\JJ^\eta_{{\rm{hyp}}}(l, \fn|\alpha)$ amounts at most to $\nr(\fn)^{-\delta}$ giving an error term. The new and essential contribution of this paper to the relative trace formula is probably \S 10 and \S 11, which are devoted to computing the functionals $\tilde \JJ^\eta_{{\rm{u}}}(l, \fn|\alpha)$ and $\JJ^\eta_{\rm{hyp}}(l, \fn|\alpha)$ explicitly for particular but sufficiently general test functions $\alpha$. For the result, we refer to Theorems~\ref{HYPERBOLIC(fine)} and \ref{UNIPOTENT(fine)}. In the final section \S 12, we prove Theorem~\ref{T4} by applying the relative trace formula (Theorem~\ref{RELATIVETRACEFORMULA}) to a specially chosen test function (see 12.1) originally due to Iwaniec. In the proof, our explicit formula of orbital integrals for arbitrary Hecke functions plays an essential role. We would like to mention our work \cite{SugiyamaTsuzuki} where we obtain an analogue of results of \cite{Royer} for the central (derivative of) $L$-values of Hilbert modular forms; in \cite{SugiyamaTsuzuki}, the explicit relative trace formula to be developed in this article is also indispensable.

\medskip
\noindent
{\bf Basic notation and convention} : Let $\NN$ be the set of all positive integers and we write $\NN_{0}$ for $\NN \cup \{0\}$.
For any condition $P$, we put $\delta(P)=1$ if $P$ is true, and
$\delta(P)=0$ if $P$ is false, respectively. For any $z\in \C^\times$ and $\alpha\in \C$, we define $\log z$ and $z^\alpha$ by the formula 
$$
\log z=\log r+i\theta, \qquad z^{\alpha}=\exp(\alpha \log z)
$$
with $z=re^{i\theta} \,(r>0,\theta\in (-\pi,\pi])$. For a complex function $f(z)$ in $z\in \C$ and for $\sigma\in \R$, the contour integral $\int_{\sigma-i\infty}^{\sigma+i\infty} f(z)\d z$ along the vertical line $\Re(z)=\sigma$ is sometimes denoted by $\int_{L_\sigma}f(z)\d z$. We set $\Gamma_\R(s)=\pi^{-s/2}\Gamma(s/2)$ and $\Gamma_{\C}(s)=2(2\pi)^{-s}\Gamma(s)$. Set $1_2=\left[\begin{smallmatrix} 1 & 0 \\ 0 & 1 \end{smallmatrix}\right]$, the identity matrix. All the fractional ideals appearing in this paper are supposed to be non-zero. 

\section{Preliminaries}
We introduce basic objects and notation, which are used throughout this article. 

\subsection{}
Let $F$, $\AA$, $\cO$, $\Sigma_{\infty}$, $\Sigma_{\fin}$ and $\go_{v}$ for $v \in \Sigma_{\fin}$ be as in the introduction.
Set $\Sigma_F=\Sigma_\infty\cup \Sigma_\fin$
and $d_F=[F:\Q]$.
The finite adele ring of $F$ is denoted by $\AA_{\fin}$. 
For $v\in \Sigma_\fin$,
$F_{v}$ denotes the completion of $F$ at $v$ and
we fix a prime element $\varpi_v$ of $F_v$ and set $\fp_v=\varpi_v\cO_{v}$; the modulus of $F_v^\times$ is denoted by $|\,|_v$ and the associated order function is defined by $\ord_v=-\log_{q_v}|\,|_v$. Let $d_v$ be the local differential exponent of $F_v$ over $\Q_p$, where $p$ is the characteristic of $\cO_v/\gp_v$. 
The discriminant $D_F$ of $F/\Q$ is defined to be the absolute norm $\nr(\fd_{F/\Q})$, where $\fd_{F/\Q}$ is the global different of $F/\Q$.
The completed Dedekind zeta function of $F$ is denoted by $\zeta_{F}(s)$.
For any ideal $\ga$ of $\go$, let $S(\ga)$ denote the set of all $v \in \Sigma_{\fin}$ such that $\ord_{v}(\ga)\ge 1$.

\subsection{} 
Let $G$ be the $F$-algebraic group $\GL(2)$. For any $F$-subgroup $M$ of $G$, we set $M_\A=M(\A)$, $M_{F}=M(F)$, $M_\fin=M(\A_\fin)$, $M_\infty=M(F\otimes_\Q \R)$ and $M_v=M(F_v)$ for any $v\in \Sigma_F$. The points of finite adeles $G_\fin$ of $G$ is realized as a restricted direct product of the local groups $G_v$ with respect to the maximal compact subgroups $\bK_{v}=\GL(2,\cO_{v})$ over all $v\in \Sigma_\fin$.
For an ideal $\gn$ of $\go$,
let $\bfK_{0}(\gn\go_{v})$ be as in the introduction and
we put
$\bfK_{0}(\gn)=\prod_{v \in \Sigma_{\fin}}\bfK_{0}(\gn\go_{v})$,
which is an open compact subgroup of $\bfK_{\fin}=\prod_{v\in \Sigma_\fin}\bK_v$.
The Lie group $G_\infty$ is isomorphic to $\prod_{v\in \Sigma_\infty}G_v$. For each $v\in \Sigma_\infty$, let $\bK_v$ be the image of ${\rm{O}}(2,\R)$ by the isomorphism $\GL(2,\R) \cong G_v$. Note that $\bK_v^{0}$ is isomorphic to the rotation group $\SO(2, \RR)$. Set $\bK_{\infty}=\prod_{v\in \Sigma_\infty}\bK_v$ and $\bK=\bK_\fin\,\bK_\infty$. Let $Z$ be the center of $G$, $H$ the $F$-split torus of $G$ consisting of all the diagonal matrices and $N$ the $F$-subgroup of $G$ consisting of all the upper triangular unipotent matrices. Set $B=HN$.

\subsection{Haar measures}
\label{Haar measures}
For $v\in \Sigma_F$, let $\d x_v$ be the additive Haar measure of $F_v$ such that $\vol(\cO_v)=q_v^{-d_v/2}$ if $v\in \Sigma_\fin$ and $\vol(\{x\in F_v | \ |x|_v<1\})=2$ if $v\in \Sigma_\infty$. Fix a multiplicative Haar measure $\d^\times x_v$ on $F^\times_v$ by $\d^\times x_v=c_v\,\d x_v/|x_v|_v$, where $c_v=1$ if $v\in\Sigma_\infty$ and $c_v=(1-q^{-1}_v)^{-1}$ if $v\in \Sigma_\fin$. We fix a Haar measure of the idele group $\A^\times$ by $\d^\times x=\prod_{v}\d^\times x_v$. For $y>0$, let $\underline{y}\in \A^\times$ be the idele such that $\underline {y}_\iota=y^{1/d_F}$ for all $\iota\in \Sigma_\infty$ and $\underline{y}_v=1$ for all $v\in \Sigma_\fin$. Then, $y\mapsto {\underline{y}}$ is a section of the idele norm $|\,|_\A:\A^\times \rightarrow \R_+^\times$, which allows us to identify $\A^\times$ with the direct product of $\{\underline y|\,y>0\,\}$ and the norm one subgroup $\A^1=\{x\in \A^\times|\,|x|_\A=1\,\}$. We fix a Haar measure $\d^{1} u$ on $\A^1$ so that $\d^\times x=\d^{1} u\,\d^{\times} y$ when $x=u{\underline y}$ with $x\in \A^\times$, $u\in \A^1$ and $y>0$. 

We fix Haar measures $\d h_v$, $\d n_v$ and $\d k_v$ on groups $H_v$, $N_v$, $\bK_v$ respectively by setting $\d h_v=\d^\times t_{1,v}\,\d^\times t_{2,v}$ if $h_v=\left[\begin{smallmatrix} t_{1,v} & 0 \\ 0 & t_{2,v} \end{smallmatrix}\right]$, $\d n_v=\d x_v$ if $n_v=\left[\begin{smallmatrix} 1 & x_v \\ 0 & 1 \end{smallmatrix}\right]$ and by requiring $\vol(\bK_v, d k_v)=1$. Then we normalize the Haar meausre $\d g_v=\d h_v\,\d n_v\,\d k_v$ on $G_v$ by using the Iwasawa decomposition $G_v=H_vN_v\bK_v$.
We note that $\vol(\bfK_v, dg_v)=q_v^{-3d_v/2}$.
By taking the tensor product of measures $\d g_v$ on $G_v$, we fix a Haar measure $\d g$ on $G_\A$.

Let $\varphi$ be a smooth function on $G_\A$. The right translation of $\varphi$ by $g\in G_\A$ is denoted by $R(g)\varphi$, i.e., $[R(g)\varphi](h)=\varphi(hg)$. The derived action of the universal enveloping algebra of the complexified
Lie algebra $\mathfrak g_{\infty}=\Lie(G_{\infty})_\C$
on smooth functions on $G_\A$ is also denoted by $R$. Let $W$ and $\bar W$ be the element $\tfrac{1}{2} \left(\begin{smallmatrix} 1 & -i \\ -i & -1 \end{smallmatrix}\right)$ of ${\frak {sl}}_2(\C)$ and its complex conjugate, respectively. For any $v\in \Sigma_\infty$, the elements of ${\rm{Lie}}(G_v)_{\C}$ corresponding to $W$ and ${\overline {W}}$ are denoted by $W_v$ and ${\overline{W}}_v$, respectively. 
For any compactly supported smooth function $f$ on the direct product $G_S$ of $\{G_v\}_{v\in S}$ with a finite subset $S\subset \Sigma_F$, the right translation of $\varphi$ by $f$ is defined by the convolution
$R(f)\varphi(x)=\int_{G_S} \varphi(x g_S)\,f(g_S)\,\d g_S$ for $x\in G_\AA$ with respect to the product measure $d g_S=\otimes_{v\in S} \d g_v$.

\subsection{}
\label{add char and gauss sum}
Given a real valued idele class character $\eta$ of $F^\times$ with conductor $\ff$, we set $f(\eta_v)={\rm{ord}}_v(\ff)$ for $v\in \Sigma_\fin$. For any $v\in \Sigma_\infty$, there exists $\epsilon_v\in \{0,1\}$ such that $\eta_v(x)=(x/|x|_{v})^{\epsilon_v}$; we call $\epsilon_v$ the sign of $\eta$ at $v$, and set $\epsilon(\eta)=\sum_{v\in \Sigma_\infty} \epsilon_v$. Let $I(\ff)$ be the group of fractional ideals relatively prime to $\ff$; then we define a character $\tilde \eta:I(\ff) \rightarrow \{\pm 1\}$ by setting $\tilde \eta(\fp_v \cap \go)=\eta_v(\varpi_v)$ for any $v\in \Sigma_\fin-S(\ff)$. The Gauss sum $\cG(\eta)$ for $\eta$ is defined to be the product of 
$$
\cG(\eta_v)=\int_{\cO_v^\times} \eta_v(u\varpi_v^{-d_v-f(\eta_v)})\,\psi_{F,v}(u\varpi_v^{-d_v-f(\eta_v)})\,\d^\times u, 
$$
over all $v\in \Sigma_\fin$, where $\psi_{F}=\psi_{\Q}\circ \tr_{F/\Q}$ with $\psi_\Q$ being the character of $\Q\bsl \A$ such that $\psi_\Q(x)=\exp(2\pi i x)$ for $x\in \R$.

\subsection{}
Fix a relatively compact subset $\omega_B$ of $B_\A^1=\{\left[\begin{smallmatrix} a & b \\ 0 & d\end{smallmatrix}\right]|\,a,\,d\in\A^{1},\,b\in \A\,\}$ such that $B_\A^{1}=B_F\,\omega_B$. Let $\fS^1= \omega_B\,\{\left[\begin{smallmatrix} \underline{t} & 0 \\ 0 & \underline{t}^{-1}\end{smallmatrix}\right]|\, t>0, \, t^2>c\,\}\bK$ with some $c>0$ be a Siegel domain such that $G_\A=Z_\A G_F\,\fS^1$. Define $y:G_\A\rightarrow \R_{+}^{\times}$ by setting $y\left(\left[\begin{smallmatrix} a & b \\ 0 & d\end{smallmatrix}\right]k\right)=|{a}/{d}|_\A$ for any $\left[\begin{smallmatrix} a & b \\ 0 & d\end{smallmatrix}\right]\in B_\A$ and $k\in \bK$.

\section{Holomorphic Shintani functions on $\GL(2,\R)$}

Consider the following one parameter subgroups in $\GL(2,\R)$: 
\begin{align*}
k_\theta&=\left[\begin{smallmatrix} \cos \theta & -\sin \theta \\ \sin \theta & \cos \theta \end{smallmatrix} \right], \qquad a_r=\left[\begin{smallmatrix} \ch r & \sh r \\ \sh r & \ch r \end{smallmatrix} \right],
\end{align*}
where $\theta,\,r\in \R$. We have $\SO(2, \RR)=\{k_\theta|\,\theta\in \R\,\}$. 

\subsection{Discrete series of $\PGL(2,\R)$}
For $n\in \Z$, let $\tau_n$ be the character of $\SO(2, \RR)$ defined by 
\begin{align*}
\tau_n\left(k_\theta \right)=e^{i n \theta}, \qquad \theta \in \R.
\end{align*}
Let $l\geq 2$ be an even integer. Recall that there correspond discrete series representations $D_{l}^{+}$ and $D_{l}^{-}$ of $\SL_2(\R)$ such that $D_l^\pm |\SO(2, \RR)$ is a direct sum of characters $\tau_n$ for all $n\in \pm(l+2\N_0)$. We have a unitary representation $D_l$ of $\GL_2(\R)$ such that (a) $D_l$ has the trivial central character and (b) $D_{l}|\SL_2(\R)=D_l^{+}\oplus D_l^{-}$. We call $D_l$ the discrete series representation of $\PGL_2(\R)$ of minimal $\SO(2, \RR)$-type $l$.

\subsection{Shintani functions}
Let $f(\tau)$ be a cuspform on the upper half plane satisfying the modularity condition $f((a\tau+b)/(c\tau+d))=(c\tau+d)^{l}\,f(\tau)$ for any matrix $\left[\begin{smallmatrix} a & b \\ c & d \end{smallmatrix}\right]$ in a fixed congruence subgroup $\Gamma$ of $\PSL_2(\Z)$. Then it is lifted to a left $\Gamma$-invariant function $\tilde f$ on the group $\GL(2,\R)$ by setting 
$$
 \tilde f(g)=(\det g)^{l/2}(ci+d)^{-l}\,f\left(\tfrac{ai+b}{ci+d}\right) \times \delta(\det g>0), \qquad g=\left[\begin{smallmatrix} a & b \\ c & d \end{smallmatrix}\right]\in \GL(2,\R).
$$
Let $\tilde f_c$ be the complex conjugate of $\tilde f$. Then, $\tilde f_c$ satisfies the conditions 
\begin{align*}
\tilde f_c(gk_\theta)&=\tau_l(k_\theta)\,\tilde f_{c}(g), \hspace{2mm} (\forall k_\theta\in \SO(2, \RR)) , \hspace{7mm}
[R(\overline{W})\tilde f_{c}](g)=0.
\end{align*}
Since ${\rm{Ad}}(k_\theta)\overline{W}=e^{-2i\theta}\,\overline{W}$ in any $({\frak {gl}}_2(\RR), {\rm O}(2, \RR))$-module $(\pi,V)$, we have $\pi(\overline{W})V[\tau_{l}]\subset V[\tau_{l-2}]$, where
$$
 V[\tau_l]=\{v\in V|\,\pi(k_\theta)v=e^{il\theta}\,v\,(\forall k_\theta \in \SO(2, \RR))\,\}.
$$ 
Let $V$ be the $({\frak g}{\frak l}_2(\R),{\rm O}(2, \RR))$-submodule of the regular representation $L^{2}(\Gamma\bsl \GL(2,\R))$ generated by $\tilde f_c$. Then the condition above, or equivalently $\tilde f_{c} \in V[\tau_{l}]$ and $R(\overline{W})\tilde f_{c}=0$, tells us that inside the module $V$ (which is a finite sum of discrete series $D_{l}$) the vector $\tilde f_c$ is extremal.
For $z \in \CC$,
let $\chi_z$ be the quasi-character of the diagonal split torus $T$ defined by 
$
\chi_z\left(\left[\begin{smallmatrix} t_1 & 0 \\ 0 & t_2 \end{smallmatrix}\right]\right)=|t_1/t_2|^{z}.
$
The integral 
$$
\phi(g)=\int_{\Gamma \cap T \bsl T} \tilde f_c(h g)\,\chi_{-z}(h)\,\d h, \qquad g\in \GL(2,\R),
$$
often called the $(T,\chi_z)$-period integral of $\tilde f_c$, satisfies the following two conditions: 
\begin{itemize}
\item $
\phi\left(\left[\begin{smallmatrix} t_1 & 0 \\ 0 & t_2\end{smallmatrix} \right]\,g\,k_\theta\right)=|t_1/t_2|^{z}\,\tau_{l}(k_\theta)\,\phi(g) \quad
{\text{for all $\left[\begin{smallmatrix} t_1 & 0 \\ 0 & t_2\end{smallmatrix} \right] \in T$ and $\theta\in \R$}}, 
$
\item $R(\overline{W})\phi=0$.
\end{itemize}
A function having these properties is called a holomorphic Shintani function of weight $l$. The next proposition tells that these conditions determine the function $\phi(g)$ uniquely up to a constant multiple. 

\begin{prop} \label{Sh-ftn} $($\cite[Proposition 5.3]{Hirano}$)$ Let $z\in \C$. For each even integer $l\geq 2$, there exists a unique $\C$-valued $C^\infty$-function $\Psi^{(z)}(l;-)$ on $\GL(2,\R)$ with the properties:
\begin{itemize}
\item[(S-i)] It satisfies the equivariance condition 
$$
\Psi^{(z)}\left(l;\left[\begin{smallmatrix} t_1 & 0 \\ 0 & t_2\end{smallmatrix} \right]\,g\,k_\theta\right)=|t_1/t_2|^{z}\,\tau_{l}(k_\theta)\,\Psi^{(z)}(l;g) \quad{\text{for all $\left[\begin{smallmatrix} t_1 & 0 \\ 0 & t_2\end{smallmatrix} \right] \in T$ and $\theta\in \R$}}. 
$$
\item[(S-ii)] It satisfies the differential equation 
$$R(\overline{W})\,\Psi^{(z)}(l;-)=0.$$
\item[(S-iii)] $\Psi^{(z)}(l;1_2)=1$. 
\end{itemize}
We have the explicit formula
$$
 \Psi^{(z)}(l;a_r)=2^{-l/2}\,(-y)^{(2z-l)/4}\,(1-y)^{l/2}\qquad {\text{with $y=\left(\frac{e^{2r}-i}{e^{2r}+i}\right)^2$}}.$$
\end{prop}
We remark that the function $\Psi^{(z)}(l;-)$ is characterized by its restriction to the torus $A=\{a_r|\,r\in \R\,\}$ due to the property (S-i) and the decomposition $\GL(2, \RR)=T\,A\,\SO(2, \RR)$
(cf.\ \cite[Lemma 3.1]{Hirano}).

\begin{lem}\label{Shintani-val-unipotent}
Let $\Psi^{(z)}(l;-)$ be as in Proposition~\ref{Sh-ftn}. Then, 
$$ \Psi^{(z)}\left(l;\left[\begin{smallmatrix} 1 & x \\ 0 & 1\end{smallmatrix} \right]\right)=(1+ix)^{z-l/2}, \qquad x \in \R.
$$
\end{lem}

\begin{proof}
By a direct computation, $\left[\begin{smallmatrix} 1 & x \\ 0 & 1\end{smallmatrix} \right]
=\left[\begin{smallmatrix} t & 0 \\ 0 & t^{-1} \end{smallmatrix} \right]
\, a_r\,k_\theta$ with 
{\allowdisplaybreaks 
\begin{align*}
t&=(1+x^2)^{1/4}, \qquad \ch 2r=(1+x^{2})^{1/2}, \qquad \sh 2r=x,\\
e^{i\theta}&=\frac{(\sqrt{1+x^{2}}+1)^{1/2}}{\sqrt{2}\,(1+x^2)^{1/4}}\,\left(1-\frac{ix}{\sqrt{1+x^{2}}+1}\right),
\end{align*}
}and $y=\frac{x-i}{x+i}$, $1-y=\frac{2i}{x+i}$. Using these, we have the desired formula by a direct computation. 
\end{proof}

\begin{lem}\label{estimate of Shintani}
We have the estimate
$$
|\Psi^{(z)}\left(l;\left[\begin{smallmatrix} t_{1} & 0\\ 0 & t_{2} \end{smallmatrix}\right] a_{r}k\right)| \leq 2^{-l/2}\,|t_{1}/t_{2}|^{\Re(z)} e^{\pi |\Im(z)|/2} ( \ch 2r )^{-l/2}$$
for any $t_{1}, t_{2} \in \RR^{\times}$, $r \in \RR$ and $k \in \SO(2, \RR)$.
\end{lem}
\begin{proof}
Set $y=\left(\frac{e^{2r}-i}{e^{2r}+i}\right)^{2}$. Then,  
$$y = \left(\tanh 2r - \frac{i}{\cosh 2r}\right)^{2} = 1 - \frac{2}{\cosh^{2}2r} - \frac{2 i \tanh 2r}{\ch 2r}.$$
Hence, by a direct computation, we have 
$|1-y| = (\cosh 2r)^{-1}.$
Furthermore, by $|y|=1$, we have
$|(-y)^{(2z-l)/4}| \leq e^{\pi |\Im(z)|/2}$.
This completes the proof.
\end{proof}

\subsection{An inner-product formula of Shintani functions}

For an even integer $l\geq 2$ and $z\in \C$, let us consider the integral
$$
C_l(z)= \int_{1}^{\infty} \left\{\left(-\left(\frac{u-i}{u+i}\right)^2\right)^{z}+\left(-\left(\frac{u+i}{u-i}\right)^2\right)^{z}\right\}\,(1+u^2)^{1-l}\,u^{l-2}\,\d u. 
$$
\begin{lem} \label{Clproperty}
\label{C_l(z)}
The integral $C_l(z)$ converges absolutely. It has the following properties. 
\begin{itemize}
\item[(i)] The function $z\mapsto C_l(z)$ is entire and satisfies the functional equation 
$$ C_l(-z)=C_l(z). $$
\item[(ii)] The value at $z=0$ is given by 
$$C_l(0)=2^{-1}{\Gamma((l-1)/2)^2}{\Gamma(l-1)}^{-1}= 2^{3-2l}{\pi\Gamma(l-1)}
{\Gamma(l/2)^{-2}}.$$
\item[(iii)] We have 
$$
 |C_l(z)|\leq C_l(0)\,\exp(\pi|\Im(z)|), \qquad z\in \C.
$$
\end{itemize}
\end{lem}
\begin{proof} By the variable change $v^{-1}=1+u^2$, we have
\begin{align*}
C_l(0)&=2\int_{1}^{\infty}(1+u^2)^{1-l}u^{l-2}\,\d u
={2}^{-1}\int_{0}^{1}(1-v)^{(l-3)/2}v^{(l-3)/2}\,\d v 
=2^{-1}{\Gamma((l-1)/2)^2}{\Gamma(l-1)}^{-1}
\end{align*}
as desired in (ii). Remark that the second equality in (ii) is obtained by the duplication formula.
Since $w=-((u-i)/(u+i))^2$ satisfies $|w|=1$, by definition, we have $w^z=\exp(i\theta z)$ with $\theta\in (-\pi,\pi]$. Thus, $|w^{z}|=\exp(-\Im(z)\,\theta)\leq \exp(\pi|\Im z|)$, by which (iii) is immediate. From definition, we have the relation $w^{-z}=(w^{-1})^{z}$, which shows the functional equation in (i). 
\end{proof}

The inner-product of Shintani functions $\Psi^{(z)}(l;-)$ and $\Psi^{(-\bar z)}(l;-)$ is given as follows.

\begin{prop}\label{lem:inner prod of Shintani}
We have
\begin{align*}
\int_{T \bsl \GL(2,\R)} \Psi^{(z)}(l;g)\,\overline{\Psi^{(-\bar z)}(l;g)}\,\d g = 2^{l-1}\,C_l(z).
\end{align*}
\end{prop}
\begin{proof} Set $f(g)=\Psi^{(z)}(l;g)\,\overline{\Psi^{(-\bar z)}(l;g)}$.  We have 
\begin{align*}
\int_{T\bsl \GL(2,\R)} f(g)\,\d g=2\,\int_\R f(a_r)\,\ch 2r\, \d r 
\end{align*}
by the formula \cite[(3.3)]{Tsud}, which is checked by computing the Jacobian of the coordinate transform from the Iwasawa decomposition to the decomposition $G=T A\,\SO(2)$. From Proposition~\ref{Sh-ftn}, 
\begin{align*}
f(a_r)&=2^{-l}\,(-y)^{-l/2+z}\,(1-y)^{l} \quad \text{with $y=\left(\frac{e^{2r}-i}{e^{2r}+i}\right)^{2}$}.  
\end{align*}
By this, we compute 
{\allowdisplaybreaks
\begin{align*}
2 \int_{0}^{+\infty}f(a_r)\,\ch 2r\,\d r
&=2^{1-l}\int_{0}^{\infty} (-y)^{-l/2+z}\,(1-y)^{l}\,\ch 2r\,\d r \\
&=2^{l-1} \int_1^{\infty}\left\{-\left(\frac{u-i}{u+i}\right)^2\right\}^{z} (1+u^2)^{1-l}\,u^{l-2}\,\d u, 
\end{align*}}
setting $u=e^{2r}$. In the same way, we have
{\allowdisplaybreaks
\begin{align*}
2 \int_{-\infty}^{0} f(a_r)\,\ch 2r\,\d r
=2^{l-1} \int_1^{\infty}\left\{-\left(\frac{u+i}{u-i}\right)^2\right\}^{z} (1+u^2)^{1-l}\,u^{l-2}\,\d u. 
\end{align*}}
\end{proof}

\subsection{Orbital integrals of Shintani functions}
Set $w_{0}= k_{\pi/2}=\left[\begin{smallmatrix}0 & -1\\1 & 0 \end{smallmatrix}\right]$.

\begin{lem}\label{lem:orbital integral of Shintani}
If $0< \Re(z) <l/2$, then, for $\e,\,\e' \in \{0, 1\}$, we have
{\allowdisplaybreaks
\begin{align*}
\int_{\RR^{\times}}\Psi^{(0)}\left(l;\left[\begin{smallmatrix}1 & x\\ 0 &1\end{smallmatrix}\right]w_{0}^{\e'}\right)|x|^{z}\sgn^{\e}(x)d^{\times}x &=2 i^{l\e'}
{\Gamma(z)\Gamma(l/2-z)}{\Gamma(l/2)}^{-1}
i^{\e}\cos\left(\tfrac{\pi}{2}(z+\e)\right),
\\
\int_{\RR^{\times}}\Psi^{(0)}\left(l;\left[\begin{smallmatrix}1 & 0\\ x &1\end{smallmatrix}\right]w_{0}^{\e'}\right)|x|^{z}\sgn^{\e}(x)d^{\times}x &=2i^{l\e'}{\Gamma(z)\Gamma(l/2-z)}{\Gamma(l/2)}^{-1}(-i)^{\e}\cos\left(\tfrac{\pi}{2}(z+\e)\right).
\end{align*}
}
\end{lem}
\begin{proof} Let $J_{l,\epsilon}(z)$ denote the first integral with $\e'=0$. From Lemma~\ref{Shintani-val-unipotent}, we have $J_{l,\epsilon}(z)=J^+_l(z)+(-1)^{\epsilon}J^{-}_l(z)$ with 
\begin{align*}
J^{\pm}_l(z)&=\int_{0}^{\infty}(1\pm ix)^{-l/2}\,x^{z}\,\d^\times x.
\end{align*}
By the formula \cite[3.194.3]{Gradshteyn}, we have 
\begin{align*}
J^{\pm}_l(z)=(\pm i)^{-z}B(z,l/2-z)=(\pm i)^{-z}{\Gamma(z)\Gamma(l/2-z)}{\Gamma(l/2)}^{-1} \quad (l/2>\Re(z)>0).
\end{align*}
Hence, 
\begin{align*}
J_{l,\epsilon}(z)={\Gamma(z)\Gamma(l/2-z)}{\Gamma(l/2)}^{-1}\{i^{-z}+(-1)^{\epsilon} (-i)^{-z}\}.
\end{align*}
Since $i^{-z}+(-1)^{\epsilon} (-i)^{-z}=2 i^{\e} \cos(\pi (z+\epsilon)/2)$, we are done. We have the Iwasawa decomposition  
\begin{align*}
\left[\begin{smallmatrix} 1 & 0 \\ x & 1\end{smallmatrix} \right]
&=\left[\begin{smallmatrix} \frac{1}{\sqrt{1+x^{2}}} & 0 \\ 0 & \sqrt{1+x^{2}} \end{smallmatrix} \right]
\left[\begin{smallmatrix} 1 & x \\ 0 & 1 \end{smallmatrix} \right]
\,k_\theta \qquad {\text{with $e^{i\theta}=\frac{1+ix}{\sqrt{1+x^{2}}}$}}.
\end{align*}
Hence, by Lemma \ref{Shintani-val-unipotent}, we obtain
\begin{align*}
\Psi^{(z)}\left(l;\left[\begin{smallmatrix}1&0\\x&1\end{smallmatrix}\right]\right)
= & \left(\frac{1}{1+x^{2}}\right)^{z}\times \left(\frac{1+ix}{\sqrt{1+x^{2}}}\right)^{l}\times (1+ix)^{z-l/2}=(1-ix)^{-z-l/2}
\end{align*}
Using this formula, in the same way as above, we can prove the second formula with $\e'=0$. The remaining two formulas follow immediately from the proved ones by the relation $\Psi^{(0)}(l;gw_0)=i^{l}\,\Psi^{(0)}(l;g)$.  
\end{proof}

\section{Green's functions on $\GL(2)$ over non-archimedean local fields}
This section is a review of results in \cite[\S 5]{Tsud}. We fix a place $v\in \Sigma_\fin$. For $z\in \C$, there exists a unique function $\Phi_{0,v}^{(z)}:G_v\rightarrow \C$ such that
\begin{align}
\Phi_{0,v}^{(z)}(\left[\begin{smallmatrix} t_1 & 0 \\ 0 & t_2 \end{smallmatrix}\right]\left[\begin{smallmatrix} 1& x \\ 0 & 1\end{smallmatrix}\right] k)=|t_1/t_2|_v^{z}\,\delta(x\in \cO_v), \qquad \left[\begin{smallmatrix} t_1 & 0 \\ 0 & t_2 \end{smallmatrix}\right]\in H_v,\,\left[\begin{smallmatrix} 1& x \\ 0 & 1\end{smallmatrix}\right]\in N_v ,\,k\in \bK_v. 
 \label{vPhi_0}
\end{align}

 Given $z\in \C$ and $s\in \C/4 \pi i (\log q_v)^{-1}\Z$, we consider the following inhomogeneous equation
\begin{align}
R\left({\bT}_v- (q_v^{(1-s)/2}+q_v^{(1+s)/2})\,1_{\bK_v}\right) \Psi=\Phi_{0,v}^{(z)}
 \label{v-Eigeneq1}
\end{align}
with the unknown function $\Psi:G_v\rightarrow \C$ possessing the $(H_v,\bK_v)$-equivariance: 
\begin{align}
\Psi\left(
\left[\begin{smallmatrix} t_1 & 0 \\ 0 & t_2 \end{smallmatrix}\right]
gk\right)=|t_1/t_2|_v^{z}\,\Psi(g), \qquad \left[\begin{smallmatrix} t_1 & 0 \\ 0 & t_2 \end{smallmatrix}\right]\in H_v,\,k\in \bK_v.
 \label{v-Eigeneq2}
\end{align}
Here ${\bT}_v$ and $1_{\bK_v}$ are elements of the spherical Hecke algebra $\cH(G_v,\bK_v)$ defined by 
\begin{align*}
\bT_v=\vol(\bfK_v, dg_v)^{-1}\,{\rm ch}_{\bK_v\left[\begin{smallmatrix} \varpi_v & 0 \\ 0 & 1 \end{smallmatrix}\right]\bK_v}, \qquad 1_{\bK_v}=\vol(\bfK_v, dg_v)^{-1}\,{\rm ch}_{\bK_v}.
\end{align*}
We note that $\vol(\bfK_v, dg_v)=q_v^{-3d_v/2}$ (see \S \ref{Haar measures}).

\begin{lem} $($\cite[Lemma 5.2]{Tsud}$)$ \label{v-Greenftn}
Suppose $\Re(s)>|2\Re(z)-1|$. Then, there exists a unique bounded function $\Psi_v^{(z)}(s;-):G_v\rightarrow \C$ satisfying \eqref{v-Eigeneq1} and \eqref{v-Eigeneq2}, whose values on $N_v$ are given by
\begin{align}
\Psi_v^{(z)}\left(s;\left[\begin{smallmatrix} 1 & x \\ 0 & 1 \end{smallmatrix}\right]\right)=-q^{-(s+1)/2}_v(1-q_v^{-(s-2z+1)/2})^{-1}(1-q_v^{-(s+2z+1)/2})^{-1}\sup(1,|x|_v)^{-(s-2z+1)/2}, \qquad x\in F_v.
 \label{v-Greenftn0}
\end{align}
\end{lem}
\begin{proof} We review the proof from \cite[Lemmas 5.1 and 5.2]{Tsud}. By the decomposition $G_{v}= \coprod_{m\ge 0} H_{v}{\rm n}_{m}\bfK_{v}$ with ${\rm n}_{m}= \left[\begin{smallmatrix} 1 & \varpi_{v}^{-m} \\ 0 & 1 \end{smallmatrix}\right]$, the condition (\ref{v-Eigeneq2}) implies that a function $\Psi$ satisfying (\ref{v-Eigeneq1}) and (\ref{v-Eigeneq2}) is determined by the system of numbers $a(m) = \Psi({\rm n}_{m})$, $m \ge 0$. The equation (\ref{v-Eigeneq1}) yields a recurrence relation among $a(m-1)$, $a(m)$ and $a(m+1)$. By solving it, we are done. 
\end{proof}

The following lemma is necessary in the proof of Proposition~\ref{GreenEquation}. 

\begin{lem} $($\cite[Lemma 5.4]{Tsud}$)$\label{NONARCHGREEN-EQ}
Let $f:G_v\rightarrow \C$ be a smooth function such that $f\left(\left[\begin{smallmatrix} t_1 & 0\\ 0 & t_2\end{smallmatrix}\right]gk\right)=|t_1/t_2|_v^{-z}\,f(g)$ for any $t_1,t_2\in F_v^\times$ and for any $k\in \bK_v$. Then, the equality
\begin{align}
\int_{H_v\bsl G_v} \Psi_v^{(z)}(s;g)\,[R(\bT_v-(q_v^{(1+s)/2}+q_v^{(1-s)/2})\,1_{\bK_v})f](g)\,\d g=\vol(H_v\bsl H_v\bK_v)\,f(1_{2})
 \label{NONARCHGREEN-EQ0}
\end{align}
holds as long as the integral on the left-hand side converges absolutely. 
\end{lem}
\begin{proof}
 We review the proof from \cite[Lemma 5.4]{Tsud}. In the left-hand side of (\ref{NONARCHGREEN-EQ0}), we move the operator $R(\bT_v-(q_v^{(1+s)/2}+q_v^{(1-s)/2})\,1_{\bK_v})$ applied for $f$ to the front of $\Psi_v^{(z)}$ by a simple variable change; then due to \eqref{v-Eigeneq1}, we have the equality 
$$\int_{H_{v}\bsl G_{v}}
[R(\bT_v-(q_v^{(1+s)/2}+q_v^{(1-s)/2})\,1_{\bK_v})\Psi_{v}^{(z)}(s)](g) f(g)dg = \int_{H_{v}\bsl G_{v}}\Phi_{0, v}^{(z)}(g)f(g)dg,
$$
whose right-hand side equals $\vol(H_v\bsl H_v\bK_v)\,f(1_{2})$ by \eqref{vPhi_0}.
\end{proof}

\section{Automorphic Green functions}
Let $S\subset \Sigma_\fin$ be a finite subset. Put 
$$ \fX_S=\prod_{v\in S} \left(\C/{4\pi i }{(\log q_v)}^{-1}\Z\right),
$$
which we regard as a complex manifold in the obvious way. Note that for any $\bfc\in \R^S$, the slice ${\mathbb L}_S(\bfc)=\{\bs\in \fX_S|\Re(\bfs)=\bfc\}$ is a compact set homeomorphic to the torus $({\bf S}^1)^{S}$. 

Given $\bfs \in \mathfrak{X}_{S}$, $z \in \CC$, an ideal $\gn \subset \go$ such that $S(\fn)\cap S=\emp$, and a family $l = (l_{v})_{v \in \Sigma_{\infty}} \in (2\ZZ_{\ge 2})^{\Sigma_{\infty}}$, {\it the adelic Green function} $\Psi_l^{(z)}(\fn|\bs,-)$ is defined by
$$\Psi^{(z)}_{l} (\gn|{\bf s}; g) := \prod_{v \in \Sigma_{\infty}} \Psi_{v}^{(z)}(l_{v}; g_{v}) \prod_{v \in S} \Psi_{v}^{(z)}(s_{v}; g_{v})
\prod_{v \in S(\gn)} \Phi_{\gn, v}^{(z)}(g_{v}) \prod_{v \in \Sigma_{\fin}-(S\cup S(\gn))} \Phi_{0, v}^{(z)}(g_{v})$$
for any $g = (g_{v})_{v\in \Sigma_{F}} \in G_{\AA}$, where $\Psi_{v}^{(z)}(l_{v};-)$ for $v \in \Sigma_{\infty}$ is the holomorphic Shintani function on $G_v\cong \GL(2,\R)$ defined in Proposition~\ref{Sh-ftn}, $\Psi_v^{(z)}(s;-)$ for $v\in S$ is the Green function recalled in \S 4, and for any $v \in \Sigma_{\fin}$, we set
$$\Phi_{\gn, v}^{(z)} \left(\left[\begin{smallmatrix}t_{1} & 0 \\ 0 & t_{2} \end{smallmatrix}\right]\left[\begin{smallmatrix} 1 & x \\ 0 & 1 \end{smallmatrix}\right] k
\right) = |{t_{1}}/{t_{2}}|_{v}^{z}\delta(x \in \go_{v}) \delta(k \in \bfK_{0}(\gn \go_{v})), \quad t_{1}, t_{2} \in F_{v}^{\times}, \, x \in F_{v},\,k \in \bfK_{v}.$$
We remark that $\Phi_{\fn,v}^{(z)}=\Phi_{0,v}^{(z)}$ if $v\in \Sigma_\fin-S(\fn)$. The adelic Green function $\Psi_l^{(z)}(\fn|{\bf s},-)$ is a smooth function on $G_\A$ having the equivariance 
$$
\Psi_l^{(z)}(\fn|{\bf s}; hgk_\infty k_\fin)=\{\prod_{v\in \Sigma_\infty} \tau_{l_v}(k_v)\} \chi_{z}(h)\Psi_l^{(z)}(\fn|{\bf s},g), \quad g\in G_\A
$$
for any $h\in H_\A$, $k_\infty=(k_v)_{v\in \Sigma_\infty}\in \bK_\infty^0$ and $k_\fin \in \bK_0(\fn)$, where $\chi_{z}:H_{F}\backslash H_{\AA} \rightarrow \CC^{\times}$ is the quasi-character defined by
$$\chi_{z}\left(\left[\begin{smallmatrix}t_{1} & 0 \\ 0 & t_{2}\end{smallmatrix}\right] \right) = |t_{1}/t_{2}|_{\AA}^{z}, \quad t_{1}, t_{2} \in \AA^{\times}.
$$
To state the most important property of the adelic Green functions, we introduce the $(H,\chi_z)$-period integral of $\varphi \in C_{c}^{\infty}(Z_\A G_{F} \backslash G_{\AA})$ by setting
$$
 \varphi^{H, (z)}(g)= \int_{Z_\A H_{F} \backslash H_{\AA}} \varphi(hg)\chi_{z}(h)dh.
$$
The integral $\varphi^{H, (z)}(g)$ converges absolutely and satisfies $\varphi^{H, (z)}(hg) = \chi_{z}(h)^{-1}\varphi^{H, (z)}(g)$ for any $h \in H_{\AA}$ (cf.\ \S3.2). Let $C_{c}^{\infty}(Z_\A G_{F} \backslash G_{\AA})[\tau_{l}]$ be the space of $\varphi\in C_{c}^{\infty}(Z_\A G_{F} \backslash G_{\AA})$ such that 
$$\varphi(g k_{\infty}) = \{ \prod_{v \in \Sigma_{\infty}} \tau_{l_{v}}(k_{v})\}\, \varphi(g) \quad {\text{ for all $k_{\infty} = (k_{v})_{v \in \Sigma_{\infty}} \in \bK_\infty^{0}$ and $g\in G_\A$}}. $$

\begin{lem}\label{lem:factorization of phi^{H,(z)}}
Suppose $\varphi \in C_{c}^{\infty}(Z_\A G_{F} \backslash G_{\AA})[\tau_{l}]$ and $R(\overline{W}_v)\,\varphi=0$ for all $v \in \Sigma_{\infty}$. Then we have
$$\overline{\varphi}^{H, (z)}(g_\fin g_\infty) = \{ \prod_{v \in \Sigma_{\infty}}\overline{\Psi_{v}^{(-{\bar z})}(l_{v}, g_{v})} \} \overline{\varphi}^{H, (z)}(g_{\fin})$$
for $g_\infty=(g_v)_{v\in \Sigma_\infty}\in G_\infty$ and $g_\fin\in  G_\fin$. 
\end{lem}
\begin{proof} Let $g_\fin \in G_\fin$. For any $v \in \Sigma_{\infty}$,
we can easily verify
$$\overline{\varphi}^{H, (z)}\left(g_\fin \left[\begin{smallmatrix}t_{1} & 0 \\ 0 & t_{2} \end{smallmatrix}\right]g_\infty k\right) = |t_{1}/t_{2}|_{v}^{-z}\tau_{l_{v}}(k)^{-1}\overline{\varphi}^{H, (z)}(g_\fin g_\infty), \hspace{3mm} t_{1}, t_{2} \in F_{v}^{\times}, k \in \bK_v^0,\, g_{\infty} \in G_{\infty}.$$
Moreover we have $R(W_v)\,(\overline{\varphi}^{H, (z)})=0$
by the equality $R(W_v)\,(\overline{\varphi}^{H, (z)}) = (\overline{R(\bar W_v)
\,\varphi})^{H, (z)}$. Thus the uniqueness of Shintani functions (Proposition~\ref{Sh-ftn}) yields a constant $C$ such that
$$\overline{\varphi}^{H, (z)}(g_\fin g_\infty) = C \prod_{v \in \Sigma_{\infty}} \overline{\Psi_{v}^{(-{\bar z})}(l_{v}; g_{v})}\quad {\text{for all $g_\infty \in G_\infty$}}.$$
By setting $g_\infty = 1_{2}$, we have $C = \overline{\varphi}^{H, (z)}(g_{\fin}) \{ \prod_{v \in \Sigma_{\infty}}\overline{\Psi_{v}^{(-{\bar z})}(l_{v}; 1_{2})} \}^{-1} = \overline{\varphi}^{H, (z)}(g_{\fin})$. This completes the proof.
\end{proof}

For $\bfs \in \mathfrak{X}_{S}$, we consider the element
$${\bfT}_{S}(\bfs) = \bigotimes_{v \in S}\{\TT_{v} - (q_{v}^{(1-s_{v})/2} + q_{v}^{(1+s_{v})/2})\,1_{\bfK_{v}}\}$$
of the Hecke algebra $\bigotimes_{v \in S} \cH(G_{v}, \bfK_{v})$. We also set
$$
q(\bs)=\inf\{(\Re(s_v)+1)/4\,|\,v\in S\,\}.
$$

\begin{prop} \label{GreenEquation}
Suppose $q(\bfs) > 2 |\Re(z)|+1$. For $\varphi \in C_{c}^{\infty}(Z_\A G_{F} \backslash G_{\AA})[\tau_{l}]^{\bfK_{0}(\gn)}$ such that $R(\overline{W}_v)\,\varphi = 0$ for all $v \in \Sigma_{\infty}$, the function $g \mapsto \Psi_{l}^{(z)}(\gn | \bfs; g)\overline{\varphi}^{H, (z)}(g)$ is integrable on $H_{\AA} \backslash G_{\AA}$. Moreover, we have
$$
\int_{H_{\AA} \backslash G_{\AA}} \Psi_{l}^{(z)}(\gn|{\bf s}; g) [R(\bfT_{S}(\bfs)){\overline{\varphi}}^{H, (z)}](g)dg =
\{ \prod_{v \in \Sigma_{\infty}}2^{l_{v}-1}C_{l_{v}}(z) \}\vol(H_{\fin} \backslash H_{\fin}\bfK_{0}(\gn)) {\overline{\varphi}}^{H, (z)}(1_{2}).
$$
\end{prop}
\begin{proof} We follow the argument in the proof of \cite[Lemma 6.3]{Tsud}. By Lemma \ref{lem:factorization of phi^{H,(z)}}, the integral in the left-hand side is the product of 
\begin{align*}
& \prod_{v\in \Sigma_\infty} \int_{H_{v} \backslash G_{v}} \Psi_{v}^{(z)}(l_{v}; g_{v})\overline{\Psi_{v}^{(-\overline{z})}(l_{v}; g_{v})}\,\d g_v, 
\end{align*}
which is evaluated by Proposition \ref{lem:inner prod of Shintani}
and 
\begin{align*}
&\int_{H_\fin\bsl G_\fin} \{\prod_{v \in S} \Psi_{v}^{(z)}(s_{v}; g_{v}) \prod_{v\in S(\gn)} \Phi_{\gn, v}^{(z)}(g_{v})\prod_{v \in \Sigma_{\fin}-(S \cup S(\gn))}\Phi_{0, v}^{(z)}(g_{v})\} \,[R(\bfT_{S}(\bfs)){\overline{\varphi}}^{H, (z)}](g_{\fin})\,\d g_\fin,
\end{align*}
which yields the factor $\vol(H_{\fin} \backslash H_{\fin}\bfK_{0}(\gn)) {\overline{\varphi}}^{H, (z)}(1_{2})$ by Lemma \ref{NONARCHGREEN-EQ}.
\end{proof}

\subsection{Regularization of periods and automorphic smoothed kernels}
For a weight $l=(l_v)_{v\in \Sigma_\infty}\in (2\N)^{\Sigma_\infty}$, set
$\ul=\inf_{v\in \Sigma_\infty}l_v$. In this subsection, we introduce the automorphic renormalized smoothed kernel function $\hat{\bf \Psi}_{\beta,\lambda}^{l}(\fn|\alpha;g)$ depending on a complex parameter $\lambda$. We show that $\hat{\bf \Psi}_{\beta,\lambda}^{l}(\fn|\alpha;g)$, originally defined by the Poincar\'{e} series \eqref{centralObj} convergent for $\Re(\l)>0$, becomes square integrable (even cuspidal) when $\ul \ge 4$ and $1/2<\Re(\lambda)<\ul/2-1$. 

\subsubsection{}
Let $\Bcal$ denote the space of all the entire functions $\b(z)$ on $\C$ such that $\b(z) = \b(-z)$ satisfying the following condition: For any interval $[a,b] \subset \RR$, there exist $A>0$ and $B \in \RR$ such that the estimate
$$|\b(\s+it)| \ll e^{-A(|t|+B)^{2}}, \hspace{5mm} \s \in [a, b],\ t \in \RR$$
holds. We impose a stronger condition than  \cite[(6.1)]{Tsud} to have the inclusion $C_l\Bcal \subset \Bcal$, which is seen from Lemma~\ref{Clproperty} (iii). For $\b\in \Bcal$ and $(\bfs,\l)\in \fX_S\times \C$ such that $q(\bfs)>1$, $\Re(\l)>1-q(\bfs)$, we define the renormalized Green function by
$$\Psi_{\b, \l}^{l}(\gn|\bfs; g) = \frac{1}{2\pi i}\int_{L_{\s}}\frac{\b(z)}{z+\l}\{\Psi_{l}^{(z)}(\gn |\bfs; g ) + \Psi_{l}^{(-z)}(\gn |\bfs; g) \}dz,
$$
where the contour is taken so that $-\inf(q(\bfs)-1,\Re(\l))< \s < q(\bfs)-1$. The defining integral is absolutely convergent and is independent of the choice of a contour; the function $\l\mapsto \Psi_{\b, \l}^{l}(\gn|\bfs; g)$ is holomorphic on the region $\Re(\l)>1-q(\bfs)$ which contains $\l=0$. 

\subsubsection{}
Our main interest is the central $L$-values $L(1/2,\pi)$ which are essentially the $(H,\1)$-period of cusp forms belonging to $\pi$ by Hecke's zeta integral. Proposition~\ref{GreenEquation} strongly suggests that the automorphic object
\begin{align}
\sum_{\gamma\in H_F\bsl G_F}\Psi_l^{(0)}(\gn|\bfs; \gamma g),
 \label{fakeKernel}
\end{align}
if well-defined, might have the spectral resolution describable by the $(H,\1)$-period integral of cusp forms $\varphi$ through the following formal computation 
{\small\begin{align}
&\int_{Z_\A G_F\bsl G_\A} \bigl\{\sum_{\gamma\in H_F\bsl G_F}\Psi_l^{(0)}(\gn|\bfs; \gamma g)\bigr\}\,[R(\bfT_S(\bs))\varphi](g)\,\d g
 \label{formalcomputation}
\\
&=\int_{Z_\A H_F \bsl G_\A} \Psi_l^{(0)}(\gn|\bfs; g)\,[R(\bfT_S(\bs))\varphi](g)\,\d g
 \notag
\\
&=\int_{H_\A\bsl G_\A} \Psi_l^{(0)}(\gn|\bfs; g)\,[R(\bfT_S(\bs))\varphi^{H,(0)}](g)
\d g 
={\rm {const.}}\, \varphi^{H,(0)}(1_2).
 \notag
\end{align}
}Unfortunately, this is not attained for free due to the divergence of the series $\sum_{\gamma}|\Psi_l^{(0)}(\gn|\bfs; \gamma g)|$ for almost all $g$ caused by the fact that $H_FZ_\A\bsl H_\A\cong F^\times \bsl \A^\times$ is of infinite volume. Since ${\rm{CT}}_{\l=0}\Psi_{\b, \l}^{l}(\gn|\bfs; g)=\Psi_l^{(0)}(\gn|\bfs; g)\,\b(0)$ by \cite[Lemma 6.5 or 6.9]{Tsud} and since $\Psi_{\b, \l}^{l}(\gn|\bfs; g)$ with large $\Re(\l)$ behaves on $G_\A$ well enough to ensure the absolute convergence of the Poincar\'{e} series 
\begin{align}
{\bf \Psi}_{\b, \l}^{l}(\gn|\bfs; g) = \sum_{\gamma \in H_{F} \bsl G_{F}} {\Psi}_{\b, \l}^{l}(\gn|\bfs; \gamma g), \hspace{3mm} g \in G_{\AA},
 \label{autGreenKernel}
\end{align} 
we expect some substitute for \eqref{fakeKernel} could be gained as the constant term at $\l=0$ of the analytic continuation in $\l$ of the series \eqref{autGreenKernel}. This circle of ideas motivates our study of the series \eqref{autGreenKernel}. 

\begin{lem}\label{estimate of automorphic Green}
Suppose $\ul \ge 4$.
\begin{enumerate}
\item The series ${\bf \Psi}_{\b, \l}^{l}(\gn|\bfs; g)$ converges absolutely and locally uniformly in $(\l, \bfs, g) \in \{\Re(\l)>0\} \times \{q(\bfs)>1\} \times G_{\AA}$. For a fixed $(\l, \bfs)$ in this region, ${\bf \Psi}_{\b, \l}^{l}(\gn|\bfs; g)$ is a continuous function in $g\in G_{\AA}$, which is left $Z_{\AA}G_{F}$-invariant and right $\bfK_{0}(\gn)$-invariant, and satisfies
$${\bf \Psi}_{\b, \l}^{l}(\gn|\bfs; g k_{v}) = \tau_{l_{v}}(k_{v}) {\bf \Psi}_{\b, \l}^{l}(\gn|\bfs; g)$$
for all $v \in \Sigma_{\infty}$ and $k_{v} \in \bfK_{v}^{0}$.

\item Let $(\l, \bfs)$ be an element of $\CC \times \mathfrak{X}_{S}$ such that $2\Re(\l)>1$,
$q(\bfs)>2\Re(\l)+1$ and $\ul/2 > \Re(\l)+1$.
Then, for any $\s \in (1/2 ,\Re(\l))$, we have the estimate
$$|{\bf \Psi}_{\b, \l}^{l}(\gn|\bfs; g)| \prec y(g)^{1-\s}, \ \ g \in \gS^1.$$

\end{enumerate}
\end{lem}
\begin{proof}
The same proof of \cite[Proposition 8.1]{Tsud} goes through with a minor modification. The outline is as follows.
For $p>0$ and $q>1$, set
\begin{align*}
\Xi_{l, p, q, S}([\begin{smallmatrix}
t_{1} & 0 \\ 0& t_{2}
\end{smallmatrix}](a_{r_{v}})_{v \in \Sigma_{\infty}}
([\begin{smallmatrix}
1 & x_{v} \\ 0& 1 \end{smallmatrix}])_{v \in \Sigma_{\fin}} k)
= & \inf\{ |t_{1}/t_{2}|_{\AA}^{p}, |t_{1}/t_{2}|_{\AA}^{-p} \}
\prod_{v \in \Sigma_{\infty}}(\ch 2 r_{v})^{-\ul/2} \\
& \times \prod_{v \in S}\sup(1, |x_{v}|_{v})^{-q}\prod_{v \in \Sigma_{\fin}-S} \delta(x_{v} \in \go_{v})
\end{align*}
for $t_{1}, t_{2} \in \AA^{\times}$, $(r_{v})_{v \in \Sigma_{\infty}} \in \RR^{\Sigma_{\infty}}$ and $\ (x_{v})_{v \in \Sigma_{\fin}} \in \AA_{\fin}$,
and set
$${\bf\Xi}_{l, p, q, S}(g)=\sum_{\gamma \in H_{F}\bsl G_{F}}\Xi_{l, p, q, S}(\gamma g), \qquad g \in G_{\AA}.$$
By Lemma \ref{estimate of Shintani}, $\Xi_{l, \s, q(\bfs), S}$ with $0<\s< \inf(\Re(\l), q(\bfs)-1)$ gives a majorant of $\Psi_{\b, \l}^{l}(\gn | \bfs)$ in the same way as \cite[Lemma 6.7]{Tsud}. Thus to prove the convergence and the estimation for ${\mathbf\Psi}_{\b, \l}^{l}(\gn | \bfs)$, it is enough to establish that ${\bf\Xi}_{l, p, q, S}$ is locally uniformly convergent in $G_{\AA}$, and that$${\bf\Xi}_{l, p, q, S}(g) \ll y(g)^{1-p}, \quad g \in \gS^{1}$$
if $1+2p<q$ and $1+p<\ul/2$. To have these, we modify the proof of \cite[Lemma 3.5]{Tsud} by replacing $q$ in the archimedean factors of $\Xi_{p, q, S}$ used there with $\ul/2>1$. We also note that the condition $1+p<\ul/2$ is necessary to guarantee $\int_{\RR}\ch(2 r_{v})^{p-\ul/2+1}dr_{v} <\infty$. 
\end{proof}

For a fixed $(\l, \bfs)$ such that $2\Re(\l)>1$, $q(\bfs)>2\Re(\l)+1$ and $\ul
/2 > \Re(\l)+1$, the function
${\bf \Psi}_{\b, \l}^{l}(\gn|\bfs)$ defines
a distribution on $Z_\A G_{F} \bsl G_{\AA}$ by
$$\langle {\bf \Psi}_{\b, \l}^{l}(\gn|\bfs), \varphi \rangle = \int_{Z_\A G_{F} \bsl G_{\AA}}
 {\bf \Psi}_{\b, \l}^{l}(\gn|\bfs; g)\varphi(g) dg, \hspace{5mm}
\varphi \in C_{c}^{\infty}(Z_\A G_{F} \bsl G_{\AA})^{\bK_0(\fn)}.$$
We remark that the absolute convergence of the integral is valid for any rapidly decreasing function $\varphi$ by Lemma \ref{estimate of automorphic Green} (2).
\subsubsection{} 
We need to make the computation \eqref{formalcomputation} regorous using ${\mathbf\Psi}_{\b, \l}^{l}(\gn|\bfs; g)$. For this, the notion of periods should be modified properly. Let us recall the regularization of period integrals along $H$, which was introduced in \cite[\S 7]{Tsud}. For $t>0$, set
$$\hat{\b}_{\l}(t) = \frac{1}{2\pi i}\int_{L_{\s}}\frac{\b(z)}{z+\l}t^{z}dz \hspace{5mm}(\s>-\Re(\l)).$$
We have the estimate $|\hat\beta_{\l}(t)|\ll \inf\{t^{\s}, t^{-\Re(\l)}\}, \ t>0$ (\cite[Lemma 7.1]{Tsud}). Given a real valued idele class character $\eta$ of $F^\times$, let $x_{\eta} = (x_{\eta,v})_{v \in \Sigma_{F}} \in \AA$ be the adele such that $x_{\eta,v} = 0$ for $v \in \Sigma_{\infty}$ and $x_{\eta,v}=\varpi_{v}^{-f(\eta_{v})}$ for $v \in \Sigma_{\rm fin}$, and let $x_{\eta}^*$ be the idele
such that the finite component of $x_{\eta}^{*}$ coincides with the projection of $x_{\eta}$ to $\AA_{\fin}^\times$ and all archimedean components of $x_{\eta}^{*}$ are equal to $1$.
A continuous function $\varphi$ on $Z_\A G_F\bsl G_\A$ is said to have the regularized $(H,\eta)$-period $P_{\rm{reg}}^\eta(\varphi)\in \C$ if the following condition is satisfied: For any $\beta \in \Bcal$, the integral
\begin{align}
P_{\beta,\lambda}^{\eta}(\varphi)=\int_{F^\times \bsl \A^\times} \varphi\left(\left[\begin{smallmatrix} t & 0 \\ 0 & 1 \end{smallmatrix}\right] \left[\begin{smallmatrix} 1 & x_\eta \\ 0 & 1 \end{smallmatrix}\right]\right)\,\eta(tx_\eta^*)\,\{\hat\beta_\lambda(|t|_\A)+\hat\beta_\lambda(|t|_\A^{-1}) \}\,\d^\times t
 \label{PBLE}
\end{align} 
converges absolutely when $\Re(\l)\gg 0$ and is continued meromorphically in a neighborhood of $\lambda=0$ with the constant term
${\rm CT}_{\l=0}P_{\b, \l}^{\eta}(\varphi) = P_{\rm{reg}}^{\eta}(\varphi) \b(0)$ in its Laurent expansion at $0$. We note that if $\varphi\in C^\infty(Z_\A G_F\bsl G_\A)$ is rapidly decreasing on $\fS^1$, then by \cite[Lemma 7.3]{Tsud}, the regularized period $P_{\rm reg}^{\1}(\varphi)$ coincides with the $(H,\1)$-period. 

The following lemma is shown along the same line of the formal compuation \eqref{formalcomputation}.   
\begin{lem}
\label{Green and period}
Assume $\ul \ge 4$. Let $(\l, \bfs)$ be an element of $\CC \times \mathfrak{X}_{S}$ such that $2\Re(\l)>1$, $q(\bfs)>2\Re(\l)+1$ and $\ul/2 > \Re(\l)+1$. Then, for any rapidly decreasing function $\varphi \in C^{\infty}(Z_\A G_{F} \bsl G_{\AA})[\tau_{l}]^{\bfK_{0}(\gn)}$ such that $R(\overline{W}_v)\varphi=0$ for all $v \in \Sigma_{\infty}$, we have
$$
\langle {\bf \Psi}_{\b, \l}^{l}(\gn|\bfs), R(\bfT_{S}(\bfs))\overline{\varphi} \rangle
= \{ \prod_{v \in \Sigma_{\infty}}2^{l_{v}-1} \} \vol(H_{\fin} \bsl H_{\fin}\bfK_{0}(\gn)) P_{\b C_{l}, \l}^{\bf 1}(\overline{\varphi}),$$
where
$C_{l}(z)=\prod_{v \in \Sigma_{\infty}}C_{l_{v}}(z)$.
\end{lem}
\begin{proof}
The proof is given in the same way as \cite[Lemma 8.2]{Tsud} with the aid of Lemma \ref{estimate of Shintani} and Proposition \ref{GreenEquation}. We note that $P_{\b C_{l},\lambda}^{{\bf 1}}(\overline{\varphi})$ is well-defined because $\b C_{l}$ belongs to $\Bcal$.
\end{proof}

\subsubsection{} 
Assume $\ul\geq 4$. Given a holomorphic function $\alpha(\bs)$ on $\fX_S$ such that $\alpha(\varepsilon \bs)=\alpha(\bs)$ for all $\varepsilon\in \{\pm 1\}^{S}$, we define the renormalized smoothed kernel 
$$\hat{\Psi}_{\b, \l}^{l}(\gn|\a ; g)= \left( \frac{1}{2\pi i} \right)^{\#S}\int_{\LL_{S}(\bf{c})}{\Psi}_{\b, \l}^{l}(\gn|\bfs;g)\,\a(\bfs) d\mu_{S}(\bfs)$$
for $\Re(\l)>0$ with ${\bf c}\in \R^{S}$ such that $q({\bf c})>\sup(\Re(\l)+1,2)$, where $\int_{\LL_S(\bfc)}f(\bfs)\,\d\mu_{S}(\bfs)$ means the multidimensional contour integral along the slice
$\LL_S(\bfc)=\{\bfs\in \fX_S|\Re(\bfs)=\bfc\}$ oriented naturally with respect to the form $\d \mu_{S}(\bs)=\prod_{v\in S} \d \mu_v(s_v)$ with
\begin{align*}
\d \mu_{v}(s_v)=2^{-1}\log q_v\,(q_v^{(1+s_v)/2}-q_v^{(1-s_v)/2})\,\d s_v.
\end{align*}
For $\Re(\l)>0$, let us consider the Poincar\'{e} series
\begin{align}
\hat{\bf \Psi}_{\b, \l}^{l}(\gn|\a; g) = \sum_{\gamma \in H_{F} \bsl G_{F}} \hat{\Psi}_{\b, \l}^{l}(\gn|\a; \gamma g), \hspace{3mm} g \in G_{\AA},
\label{centralObj}
\end{align} 
which is a central object in this paper. We introduce \eqref{centralObj} because it has a much nicer spectral expansion than ${\bf \Psi}_{\b, \l}^{l}(\gn|\bfs)$ (see Lemma~\ref{regsmoothedGreenSPECT}). In the same way as in \cite{Tsud}, we analyze this series and obtain the following.

\begin{lem}\label{estimate of smooth kernel}
\begin{enumerate}
\item The series $\hat{\bf \Psi}_{\b, \l}^{l}(\gn|\a; g)$ converges absolutely and locally uniformly in $(\l, g) \in \{\Re(\l)>0\} \times G_{\AA}$. The function $g \mapsto \hat{\bf \Psi}_{\b, \l}^{l}(\gn|\a; g)$ is continuous on $G_{\AA}$,  left $Z_\A G_{F}$-invariant, and right $\bfK_{0}(\gn)$-invariant; moreover it satisfies
\begin{align}
\hat{\bf \Psi}_{\b, \l}^{l}(\gn|\a; g k_{v}) = \tau_{l_{v}}(k_{v}) \hat{\bf \Psi}_{\b, \l}^{l}(\gn|\a; g) \quad \text{for all $k_v\in \bfK_{v}^0$}
 \label{estimate of smooth kernel1}
\end{align}
for $v \in \Sigma_{\infty}$.
\item For $0<\Re(\l)<\ul/2-1$, the function $\hat{\bf \Psi}_{\b, \l}^{l}(\gn|\a; g)$ belongs to $L^{m}(Z_\A G_{F} \bsl G_{\AA})$ for any $m>0$ such that $m(1-\Re(\l))<1$.
\end{enumerate}
\end{lem}
\begin{proof}
The argument in the proof of \cite[Proposition 9.1]{Tsud} works
with a minor modification;
We use $\Xi_{l, p, q, S}$ and ${\bf\Xi}_{l, p, q, S}$ given in the proof
of Lemma \ref{estimate of automorphic Green}.
\end{proof}

\begin{prop}\label{cuspidal}
For $1/2< \Re(\l) < \ul/2-1$, the function $\hat{\bf \Psi}_{\b, \l}^{l}(\gn|\a; g)$ is cuspidal.
\end{prop}
\begin{proof} From Proposition~\ref{Sh-ftn} and Lemma \ref{estimate of smooth kernel} (1), the function $g \mapsto \hat{\bf \Psi}_{\b, \l}^{l}(\gn|\a; g)$ on $G_{\AA}$ is smooth and satisfies the equation
\begin{align}
&R(\overline{W}_v)\,\hat{\bf \Psi}_{\b, \l}^{l}(\gn|\a; g)=0, \quad g\in G_\A
 \label{cuspidal1}
\end{align}
for all $v \in \Sigma_{\infty}$. From this equation together with the $\bK_v^0$-equivariance \eqref{estimate of smooth kernel1} of $\hat{\bf \Psi}_{\b, \l}^{l}(\gn|\a; g)$, the Casimir element of $G_v$ for each $v\in \Sigma_\infty$ acts on $\hat{\bf \Psi}_{\b, \l}^{l}(\gn|\a; g)$ by a scalar. Hence there exists a compactly supported smooth function $f$ on $G_\A$ such that $\hat{\bf \Psi}_{\b, \l}^{l}(\gn|\a)*f=\hat{\bf \Psi}_{\b, \l}^{l}(\gn|\a)$ by \cite[Theorem 2.14]{Borel}. From Lemma \ref{estimate of smooth kernel} (2), $\hat{\bf \Psi}_{\b, \l}^{l}(\gn|\a)$ belongs to $L^{2}(Z_\A G_{F} \backslash G_{\AA})^{\bK_0(\fn)}$. Thus, for any $X\in \fg_\infty$, the derivative $R(X)\hat{\bf \Psi}_{\b, \l}^{l}(\gn|\a)=\hat{\bf \Psi}_{\b, \l}^{l}(\gn|\a)*R(-X)f$ also belongs to the same $L^2$-space. Let $V$ be the $(\mathfrak{g}_{\infty}, \bfK_{\infty})$-submodule of $L^2(Z_\A G_{F} \backslash G_{\AA})^{\bK_0(\fn)}$ generated by $\hat{\bf \Psi}_{\b, \l}^{l}(\gn|\a; g)$ ; from \eqref{estimate of smooth kernel1} and \eqref{cuspidal1}, $V$ is decomposed into a finite sum of the discrete series representation $\boxtimes_{v \in \Sigma_{\infty}}D_{l_{v}}$ of $\PGL(2, F \otimes_{\QQ} \RR)$ of weight $(l_{v})_{v\in \Sigma_{\infty}}$. By Wallach's criterion \cite[Theorem 4.3]{Wallach}, the space $V$ is contained in the cuspidal part of $L^{2}(Z_\A G_{F} \backslash G_{\AA})$.
\end{proof}

By Proposition \ref{cuspidal}, for $1/2<\Re(\l)<\ul/2-1$, the function $\hat{\bf \Psi}_{\b, \l}^{l}(\gn|\a; g)$ has the spectral expansion
\begin{align}
\hat{\bf \Psi}_{\b, \l}^{l}(\gn|\a; g) = \sum_{\pi \in \Pi_{\rm cus}(l, \gn)}\sum_{\varphi \in \Bcal(\pi; l, \gn)}\langle \hat{\bf \Psi}_{\b, \l}^{l}(\gn|\a) | \varphi \rangle_{L^{2}} \varphi(g)
 \label{L^2SPECT}
\end{align}
for almost all $g\in G_\A$. Here $\langle \cdot | \cdot \rangle_{L^{2}}$ is the $L^{2}$-inner product on $L^{2}(Z_{\AA}G_{F} \backslash G_{\AA})$, and
$\Bcal(\pi; l, \gn)$ is an orthonormal basis of $\{\varphi\in L^2(Z_\A G_F\bsl G_\A)[\tau_l]^{\bK_0(\fn)} \ | \ R(\overline{W}_v)\varphi=0\,(\forall v\in \Sigma_\infty)\}$,
which consists of smooth functions.
From the finite dimensionality of the space above, the sum in \eqref{L^2SPECT} is finite and the equality holds pointwisely for all $g$.

\section{Spectral side}

From this section until \S 11, we fix an even weight $l=(l_v)_{v\in \Sigma_\infty}$, an ideal $\fn \subset \cO$, an idele class character $\eta$ of $F^\times $ such that $\eta^2={\bf 1}$ whose conductor $\ff$ is relatively prime to $\fn$, and a finite subset $S\subset \Sigma_\fin-S(\gn\gf)$. Using the spectral expansion \eqref{L^2SPECT},
we show that $\hat{\bf \Psi}_{\beta,\lambda}^{l}(\fn|\alpha;g)$ has an entire extension to the whole $\lambda$-plane. As the value at $\lambda=0$ of the entire extension, we define the regularized kernel $\hat {\bf \Psi}_{\rm{reg}}^{l}(\fn|\alpha;g)$, which is our desired substitute for the divergent series \eqref{fakeKernel}, and obtain its spectral expression. The upshot of this section is Proposition~\ref{Reg-per}, which gives the period integral of the regularized kernel.

\subsection{Extremal Whittaker vectors of discrete series}
For $v \in \Sigma_{\infty}$, let $\pi_{v}$ be the discrete series representation of $\PGL(2, \RR)$ of minimal $\bfK_{v}^0$-type $l_{v}$. Let $V_{\pi_{v}}$ denote the Whittaker model of $\pi_{v}$ with respect to the character $\psi_{F,v}$ (see \S \ref{add char and gauss sum}). It is known that $V_{\pi_{v}}[\tau_{l_v}]$ contains a unique vector $\phi_{0,v}$ characterized by  
\begin{align}
\phi_{0, v}\left(\left[\begin{smallmatrix}y & 0\\ 0 & 1\end{smallmatrix}\right]\right) = 2|y|_{v}^{l_{v}/2}e^{2\pi y}\delta(y<0), \quad y \in \RR^{\times}.
 \label{ExplicitformWhittaker}
\end{align}
We remark that $\phi_{0,v}$ is extremal, i.e., $\pi_{v}(\overline{W})\phi_{0, v} = 0$, and $V_{\pi_{v}}[\tau_{l_{v}}]=\CC \phi_{0, v}$. We should also note that the local epsilon factor of $\pi_v$ is given as
$\e(s, \pi_{v} \otimes \sgn^{m}, \psi_{F,v})=i^{l_{v}}$
 for $m \in \{ 0, 1\}$. 

\subsection{Construction of basis} 
Let $(\pi,V_\pi)$ be an irreducible cuspidal automorphic representation of $G_{\AA}$ with trivial central character such that $V_\pi\subset L^{2}(Z_\A G_F\bsl G_\A)$. We fix a family $\{ (\pi_{v},V_{\pi_v}) \}_{v \in \Sigma_{F}}$ of unitarizable irreducible admissible representations of $G_v$ with $V_{\pi_v}$ being contained in the $\psi_{F,v}$-Whittaker functions on $G_v$ such that $\pi \cong \bigotimes_{v\in \Sigma_{F}}\pi_{v}$. The conductor of ${\pi}$ is defined to be the ideal ${\gf}_{\pi}$ determined by the condition ${\gf}_{\pi}{\go}_{v} = {\gp}_{v}^{c(\pi_{v})}$ for all $v \in \Sigma_{\fin}$, where $c(\pi_v)$ is the minimal non-negative integer among those $c\in \N_0$ such that $V_{\pi_v}^{\bK_0(\fp_v^{c})}\not=\{0\}$. Let $\Pi_{\rm{cus}}(l,\fn)$ denote the set of all those cuspidal representations $\pi$ such that $\pi_{v}\cong D_{l_v}$ for all $v \in \Sigma_{\infty}$ and $\fn\subset \ff_\pi$. 

For $\pi \in \Pi_{\rm{cus}}(l,\fn)$, let $\Lambda_\pi(\gn)$ be the set of all maps $\rho:\Sigma_\fin\rightarrow \N_0$ such that $\rho(v)\in \{0,\cdots,{\rm{ord}}_v(\fn\ff_\pi^{-1})\}$ for all $v\in \Sigma_\fin$. Corresponding to each $\rho\in \Lambda_\pi(\fn)$, we have a cusp form $\varphi_{\pi, \rho} \in V_{\pi}[\tau_{l}]^{{\bf K}_{0}(\gn)}$ as the image of the decomposable tensor
$$\displaystyle \bigotimes_{v\in\Sigma_{\infty}} \phi_{0,v}\otimes \bigotimes_{v\in S({\gn}{\gf}_{\pi}^{-1})}\phi_{\rho(v),v} \otimes \bigotimes_{v\in\Sigma_{\fin}-S({\gn}{\gf}_{\pi}^{-1})}\phi_{0,v}
$$
by the isomorphism $V_{\pi}\cong \bigotimes_{v\in \Sigma_{F}} V_{\pi_{v}}$,
 where for each $v \in \Sigma_{\fin}$, the system $\{\phi_{k, v}|\,0\leq k\leq {\rm{ord}}_v(\fn\ff_\pi^{-1})\}$ is the basis of $V_{\pi_v}^{\bfK_{0}(\gn\go_{v})}$ constructed in \cite{Sugiyama1}. We remark that $\phi_{0,v}$ is the local new vector of $\pi_v$. In this way, we have an orthogonal basis $\{\varphi_{\pi,\rho}|\rho\in \Lambda_\pi(\fn)\}$ of the finite dimensional space $V_{\pi}[\tau_{l}]^{{\bf K}_{0}(\gn)}$ equipped with the $L^2$-inner-product on $Z_\A G_F\bsl G_\A$ (\cite[Proposition 17]{Sugiyama1}). The vector $\varphi_{\pi,\rho_0}$ with $\rho_0(v)=0$ for all $v\in \Sigma_\fin$ is denoted by $\varphi_{\pi}^{\rm{new}}$. 

\smallskip
\noindent
{\bf Remark}: Let $S_{k}(\fn\ff_\pi^{-1})$ be the set of $v\in S(\fn\ff_\pi^{-1})$ such that ${\rm{ord}}_{v}(\fn\ff_\pi^{-1})=k$ and $n$ the maximal non-negative integer $k$ such that $S_{k}(\fn\ff_{\pi}^{-1})\not=\emp$. For $\rho\in \Lambda_\pi(\fn)$, by writing $\rho_k=\rho|_{S_k(\fn\ff_\pi^{-1})}$ for each $0\leq k \leq n$, we can identify $\rho$ with the family of maps $(\rho_k)_{0\leq k\leq n}$ as done in \cite{Sugiyama1}.

\subsection{Regularized periods and standard $L$-values}(For details, see \cite[\S 7]{Tsud} and \cite{Sugiyama1}.)
In this paragraph, $\pi$ denotes an element of $\Pi_{\rm{cus}}(l,\fn)$. We note that for the cusp forms $\varphi\in V_\pi$, the regularized period $P_{\rm reg}^{\eta}(\varphi)$ defined by \eqref{PBLE} coincides with the global zeta integral $$Z^{*}(1/2,\eta,\varphi)=\int_{F^\times \bsl \A^\times}\varphi\left(\left[\begin{smallmatrix} t & 0 \\ 0 & 1 \end{smallmatrix}\right] \left[\begin{smallmatrix} 1 & x_\eta \\ 0 & 1 \end{smallmatrix}\right]\right)\,\eta(tx_\eta^*)\,\d^\times t,$$
 which is absolutely convergent. 

\begin{prop}\label{periods of cusp forms} 
For any $\rho \in \Lambda_{\pi}(\gn)$, $\varphi_{\pi, \rho}$ has the regularized $(H, \eta)$-period given by
$$P_{\rm reg}^{\eta}(\varphi_{\pi, \rho}) 
=Z^{*}(1/2,\eta,\varphi_{\pi, \rho})=(-1)^{\e(\eta)}\Gcal(\eta)\{ \prod_{v \in S(\gn\gf_{\pi}^{-1})}Q_{\rho(v), v}^{\pi_{v}}(\eta_{v}, 1)\} L(1/2, \pi \otimes \eta).$$
Here $Q_{k, v}^{\pi_{v}}(\eta_{v}, 1)$ with $v \in S(\gn \gf_{\pi}^{-1})$ and $k\in \{1,\dots,\ord_{v}(\gn\gf_\pi^{-1})\}$ is the constant appearing in \cite[Main Theorem A]{Sugiyama1}, and $\Gcal(\eta)$ is the Gauss sum defined in \S \ref{add char and gauss sum}.
\end{prop}
\begin{proof} The first identity is obtained by \cite[Lemma 7.3]{Tsud}. The second identity follows basically from \cite[Main Theorem A]{Sugiyama1}. Although the hypothesis $\eta_v(-1)=1$ for all $v\in \Sigma_\infty$ in \cite[Main Theorem A]{Sugiyama1} is not satisfied in our setting, it is easy to modify the proof at archimedean places by means of \eqref{ExplicitformWhittaker}. 
\end{proof}

Set
$$\PP^{\eta}(\pi ; l, \gn)= \sum_{\varphi \in \Bcal(\pi;l, \gn)} \overline{
P^{{\bf 1}}_{{\rm{reg}}}(\varphi)} P^\eta_{{\rm{reg}}}(\varphi),$$
where $\Bcal(\pi;l, \gn)$ is an orthonormal basis of
$V_{\pi}[\tau_{l}]^{\bfK_{0}(\gn)}$.

\begin{lem} \label{value of PP}
The sum $\PP^{\eta}(\pi ; l, \gn)$ is independent of the choice of $\Bcal(\pi;l,\gn)$. We have
$$\PP^{\eta}(\pi ; l, \gn) = D_{F}^{-1/2}(-1)^{\e(\eta)}\Gcal(\eta)w^{\eta}_{\gn}(\pi) \frac{L(1/2, \pi)L(1/2, \pi \otimes \eta)}{\|\varphi_{\pi}^{\rm new}\|^{2}},$$
and that the value $(-1)^{-\e(\eta)}\Gcal(\eta)^{-1}\PP^{\eta}(\pi ; l, \gn)$ is non-negative. Here $w_{\gn}^{\eta}(\pi)$ is the explicit non-negative constant given by
$$w_{\gn}^{\eta}(\pi)=\prod_{v\in S(\gn\gf_\pi^{-1})}r(\pi_v,\eta_v)$$
with $r(\pi_v,\eta_v)$ defined as follows. Set $k_v=\ord_v(\gn\gf_\pi^{-1})$. If $\eta_v(\varpi_v)=-1$,   
\begin{align*}
r(\pi_v,\eta_v)
=\dfrac{1+(-1)^{k_v}}{2}
\begin{cases}
(q_v+1)(q_v-1)^{-1}, \qquad &(c(\pi_v)=0), 
\\
1, \qquad &(c(\pi_v)\geq 1).
\end{cases}
\end{align*} If $\eta_v(\varpi_v)=1$, 
{\small\begin{align*}
r(\pi_v,\eta_v)
=
\begin{cases}
\dfrac{q_v+1}{(1+q_v^{1/2}\alpha_v)(1+q_v^{1/2}\alpha_v^{-1})} \left\{
2+\dfrac{k_v-1}{q_v-1}(1-\alpha_vq_v^{1/2})(1-\alpha_v^{-1}q_v^{1/2})
\right\}, \qquad &(c(\pi_v)=0), 
\\
1+k_v\dfrac{1-q_v^{-1}\chi_v(\varpi_v)}{1+q_v^{-1}\chi_v(\varpi_v)}
, \qquad &(c(\pi_v)=1), \\
k_v+1, \qquad &(c(\pi_v)\geq 2),
\end{cases}
\end{align*}
}where $(\alpha_v,\alpha_v^{-1})$ is the Satake parameter of $\pi_v$ if $c(\pi_v)=0$, and $\chi_v$ is the unramified character of $F_v^\times$ such that $\pi_v\cong \sigma(\chi_v|\,|_v^{1/2}, \chi_v|\,|_v^{-1/2})$ if $c(\pi_v)=1$.

If $\eta$ satisfies $\eta_v(\varpi_v)=-1$ for all $v\in S(\fn)$, then $w_\gn^\eta(\pi)=0$ unless $\gn\gf_\pi^{-1}$ is a square of integral ideal.
\end{lem}

\begin{proof}
With the aid of Proposition \ref{periods of cusp forms}, we obtain the assertion in the same way as \cite[Lemma 12]{Sugiyama2}. The non-negativity of $(-1)^{-\e(\eta)}\Gcal(\eta)^{-1}\PP^{\eta}(\pi ; l, \gn)$ follows from $w_{\gn}^{\eta}(\pi)\ge 0$ combined with the non-negativity of $L(1/2, \pi)L(1/2, \pi \otimes \eta)$ proved in \cite{Jacquet-Chen}.
\end{proof}

The sign of the functional equation of the $L$-function $L(s,\pi)L(s,\pi\otimes \eta)$ is given as follows. 

\begin{lem}\label{central=zero}
We have $\e(1/2, \pi)\e(1/2, \pi \otimes \eta)=(-1)^{\e(\eta)}\tilde{\eta}(\gf_{\pi}).$ In particular, $L(1/2, \pi)L(1/2, \pi \otimes \eta) = 0$ unless $(-1)^{\e(\eta)}\tilde{\eta}(\gf_{\pi})=1$.
\end{lem}
\begin{proof}
Since $l_{v}$ is even for all $v \in \Sigma_{\infty}$, by virtue of \cite[Lemma 13]{Sugiyama2}, we have
{\small $$\e(1/2, \pi)\e(1/2, \pi \otimes \eta)=\prod_{v\in \Sigma_{\infty}}i^{2l_{v}}\prod_{v \in S(\gf)}\eta_{v}(-1)\prod_{v \in S(\gf_{\pi})}\eta_{v}(\varpi_{v}^{c(\pi_{v})}) = \eta_{\fin}(-1)\tilde{\eta}(\gf_{\pi})
=(-1)^{\e(\eta)}\tilde{\eta}(\gf_{\pi}).$$
}By the functional equation, we are done. 
\end{proof}
\subsection{Adjoint $L$-functions}
Let $E(\nu, g)=\sum_{\gamma \in B_F\bsl G_F}y(\gamma g)^{(\nu+1)/2}\,(\Re(\nu)>1)$ be the $\bK$-spherical Eisenstein series on $G_\AA$.

\begin{lem}\label{adjoint L} 
For any $\pi \in \Pi_{\rm{cus}}(l,\fn)$, 
\begin{align}
& \int_{Z_{\AA}G_{F}\backslash G_{\AA}} \varphi_{\pi}^{\rm new}(g)\overline{\varphi_{\pi}^{\rm new}}(g)E(2s-1, g)dg 
\label{adjointL-0} \\
= & \{\prod_{v \in \Sigma_{\infty}}2^{1-l_{v}}\}
\frac{{\rm N}(\gf_{\pi})^{s}D_{F}^{s-3/2}}
{[\bfK_{\fin}:\bfK_{0}(\gf_{\pi})]}
\frac{\zeta_{F}(s)L(s, \pi; {\rm Ad})}{\zeta_{F}(2s)}
\prod_{v \in S_{\pi}}\frac{q_{v}^{d_{v}(3/2-s)}Z_{v}(s)}{q_{v}^{c(\pi_{v})(s-1)}L(s, \pi_{v}; {\rm Ad})}\frac{1+q_{v}^{-1}}{1+q_{v}^{-s}}
\notag
\end{align}
for $\Re(s)\gg0$ and
$\|\varphi_{\pi}^{\rm new}\|^{2}=2\{\prod_{v \in \Sigma_{\infty}}2^{1-l_{v}}\}{\rm N}(\gf_{\pi})[\bfK_{\fin}: \bfK_{0}(\gf_{\pi})]^{-1}L^{S_{\pi}}(1, \pi; {\rm Ad}).$
Here we set $S_{\pi} := \{ v \in \Sigma_{\fin} | \ord_{v}(\gf_{\pi})\ge 2 \}$
and
$Z_{v}(s):=\int_{\bfK_{v}}\int_{F_{v}^{\times}}\phi_{0, v}\left([\begin{smallmatrix}t&0 \\ 0&1\end{smallmatrix}]k\right) \overline{\phi_{0, v}\left([\begin{smallmatrix}t&0 \\ 0&1\end{smallmatrix}]k\right)}|t|_{v}^{s-1}d^{\times}tdk
$ for $ v \in \Sigma_{F}$.
\end{lem}
\begin{proof} By the standard procedure, we see that the left-hand side of \eqref{adjointL-0} is a product of the integrals
$Z_{v}(s)$ over all $ v \in \Sigma_{F}$. If $v \in \Sigma_{\infty}$, using \eqref{ExplicitformWhittaker}, we easily have $Z_{v}(s) =2^{1-l_{v}}{\Gamma_{\RR}(s)}{\Gamma_{\RR}(2s)}^{-1}L(s, \pi_{v}; {\rm Ad}).$ Together with the computations at finite places (cf.\ \cite[Lemma 2.14 and Corollary 2.15]{Tsud} and \cite[Lemma 14]{Sugiyama2}), this completes the proof. 
\end{proof}

\noindent
{\bf Remark} : Nelson, Pitale and Saha \cite{Saha} also considered the integrals $Z_v(s)$
and gave explicit formulas of $Z_{v}(s)$.
However, as already remarked in \cite[1.3]{Saha}, it seems difficult to give a simple formula of $Z_{v}(s)$ for $v \in S_{\pi}$.

\subsection{Spectral parameters} \label{Spectral parameters}
Let $\pi\in \Pi_{\rm{cus}}(l,\gn)$. For any $v\in \Sigma_\fin-S(\ff_\pi)$, the $v$-component $\pi_v$ of $\pi$ is isomorphic to the $\bK_v$-spherical principal series representation $ I_v(\nu_v)={\rm{Ind}}_{B_v}^{G_v}(|\,|_v^{\nu_v/2}\boxtimes|\,|_v^{-\nu_v/2})$ with $\nu_v\in \C$ such that $\nu_v$ or $\nu_v+\frac{2\pi i}{\log q_v}$ belongs to $(0,1)\cup i\R_{\geq 0}$. The point $\nu_S(\pi)=\{\nu_v\}_{v\in S}$ of $\fX_S$ is called the spectral parameter of $\pi$ at $S$. We have the Hecke eigenequation 
\begin{align}
R(\bT_v)\varphi=(q_v^{(1+\nu_v)/2}+q_v^{(1-\nu_v)/2})\varphi, \quad \varphi \in V_{\pi}^{\bK_v}, \,v\in \Sigma_\fin-S(\ff_\pi).
 \label{HeckeEEQ}
\end{align}
Since the Hecke operator $R(\bT_v)$ acting on the space $L^2(Z_\A G_F\bsl G_\A)$ is self-adjoint, the eigenvalue $q_v^{(1+\nu_v)/2}+q_v^{(1-\nu_v)/2}$ is a real number. 

\subsection{The spectral side}

By means of Lemma \ref{Green and period}, we can explicitly describe the coefficients of $\hat{\bf \Psi}_{\b, \l}^{l}(\gn|\a)$ in the $L^{2}$-expansion in terms of $(H, \bf1)$-period integrals and the spectral parameters of cupsidal representations. 

\begin{lem} \label{regsmoothedGreenSPECT}
Let $\pi \in \Pi_{\rm cus}(l, \gn)$ and $\nu_{S}(\pi)=(\nu(\pi_{v}))_{v \in S}$ the spectral parameter of $\pi$ at $S$. Then, for any $\varphi \in V_{\pi}[\tau_l]^{\bK_0(\fn)}$ and for $1/2<\Re(\l)<\ul/2-1$, we have 
$$\langle \hat{\bf \Psi}_{\b, \l}^{l}(\gn|\a)| \varphi \rangle_{L^{2}}
=(-1)^{\#S}\{\prod_{v\in \Sigma_{\infty}}2^{l_{v}-1}\}D_{F}^{-1/2}[\bfK_{\fin}:\bfK_{0}(\gn)]^{-1}\a(\nu_{ S}(\pi))P_{\b C_{l}, \l}^{\bf1}(\overline{\varphi}).$$
\end{lem}
\begin{proof}
In the same way as \cite[Lemma 9.2]{Tsud} with the aid of the majorant $\Xi_{l, \Re(\l)-\e, q(\bfc), S}$ for any sufficiently small $\e>0$, (Note: in the proof of \cite[Lemma 9.2]{Tsud}, the majorant of the integral (9.3) should be corrected to $\Xi_{\Re(\l)-\e, q(\bfc), S_{\fin}}$.) we have
$$\langle \hat{\bf \Psi}_{\b, \l}^{l}(\gn|\a) | \varphi \rangle_{L^{2}} = \left( \frac{1}{2\pi i} \right)^{\#S}\int_{\LL_{S}(\bf{c})}\langle {\bf \Psi}_{\b, \l}^{l}(\gn|\bfs), \bar{\varphi} \rangle \a(\bfs) d\mu_{S}(\bfs)$$
for any rapidly decreasing function $\varphi \in C^{\infty}(Z_\A G_{F} \bsl G_{\AA})[\tau_{l}]^{\bfK_{0}(\gn)}$, where $q(\bfc)$ is sufficiently large.
Contrary to \cite[Lemma 9.2]{Tsud}, the condition $\Re(\l)>1$ is not needed.
Indeed, in the proof of \cite[Lemma 9.2]{Tsud}, the estimate $|\varphi(g)|\ll\|g\|_{\AA}^{1+\e}$ is replaced with
$|\varphi(g)|\ll\|g\|_{\AA}^{-m}$ for any $m>0$, and
moreover, $\int_{1}^{\infty}y^{-\Re(\l)+1+2 \e}d^{\times}y$ is replaced with
$\int_{1}^{\infty}y^{-\Re(\l)-m+\e}d^{\times}y$.
(Note: in the proof of \cite[Lemma 9.2]{Tsud}, the first factor of the last integral should be corrected to $\int_{1}^{\infty}y^{-\Re(\l)+1+\e}d^{\times}y$.)
Thus, by Lemma \ref{Green and period} and \eqref{HeckeEEQ}, $\langle \hat{\bf \Psi}_{\b, \l}^{l}(\gn|\a)| \varphi \rangle_{L^{2}}$ is equal to $\{ \prod_{v \in \Sigma_{\infty}}2^{l_{v}-1} \} \vol(H_{\fin} \bsl H_{\fin}\bfK_{0}(\gn)) P_{\b C_{l}, \l}^{\bf 1}(\overline{\varphi})$ times the integral 
\begin{align*}
& \left( \frac{1}{2\pi i} \right)^{\#S}\int_{\LL_{S}(\bf{c})} 
\{\prod_{v \in S}
(q_{v}^{(1+\nu(\pi_{v}))/2}+q_{v}^{(1-\nu(\pi_{v}))/2}
-q_{v}^{(1+s_{v})/2}-q_{v}^{(1-s_{v})/2})\}^{-1} \a(\bfs) d\mu_{S}(\bfs).
\end{align*}
Here we note $q_{v}^{(1+\nu(\pi_{v}))/2}+q_{v}^{(1-\nu(\pi_{v}))/2}\in \R$ (see \S \ref{Spectral parameters}) and $\vol(H_{\fin}\bsl H_{\fin}\bfK_{0}(\gn))=D_{F}^{-1/2}[\bfK_{\fin} :\bfK_{0}(\gn)]$ from \cite[Lemma 8.3]{Tsud}. The integral is computed as $(-1)^{\# S}\a(\bfs)$. Indeed, we may assume that $\a$ is decomposable, i.e., $\a=\otimes_{v \in S}\, \a_v$, and invoke the formula
$$\frac{1}{2\pi i} \int_{c_v-2\pi i(\log q_v)^{-1}}^{c_v+2\pi i (\log q_v)^{-1}} \{q_{v}^{(1+\nu(\pi_{v}))/2}+q_{v}^{(1-\nu(\pi_{v}))/2}
-q_{v}^{(1+s_{v})/2}-q_{v}^{(1-s_{v})/2})\}^{-1} \a_v(s_v) d\mu_{v}(s_v)=-\a_v(s_v)$$
shown in \cite[Lemma 9.5]{Tsud}. This completes the proof.
\end{proof}

By this lemma and \eqref{L^2SPECT}, we have
{\small
\begin{align*}
\hat{\bf \Psi}_{\b, \l}^{l}(\gn|\a; g) = &
\frac{(-1)^{\#S}\{\prod_{v\in \Sigma_{\infty}}2^{l_{v}-1}\}
D_{F}^{-1/2}}{[\bfK_{\fin}:\bfK_{0}(\gn)]}\sum_{\pi \in \Pi_{\rm cus}(l, \gn)}\sum_{\varphi \in \Bcal(\pi; l, \gn)}
\a(\nu_{S}(\pi))P_{\b C_{l}, \l}^{\bf1}(\overline{\varphi})\varphi(g), \quad g\in G_\A. 
\end{align*}
}The integral $P_{\b C_{l}, \l}^{\bf1}(\overline{\varphi})$ is continued to an entire function in $\l$ for any cusp form $\varphi$ by \cite[Lemma 7.3]{Tsud}. As a finite linear combination of such, the function $\hat{\bf \Psi}_{\b, \l}^{l}(\gn|\a; g)$ has a holomorphic analytic continuation to the whole $\l$-plane. Since ${\rm{CT}}_{\lambda=0}P^{\eta}_{\beta C_l,\lambda}(\bar \varphi)=C_l(0)\,P^\eta_{\rm{reg}}(\bar\varphi)\,\b(0)$, we can define the regularized automorphic smoothed kernel $\hat{\bf \Psi}_{\rm reg}^{l}(\gn|\a; g)$ by the relation
$${\rm CT}_{\l=0}\hat{\bf \Psi}_{\b, \l}^{l}(\gn|\a; g)= \hat{\bf \Psi}_{\rm reg}^{l}(\gn|\a; g) \b(0)$$
for any $\b \in \Bcal$. Indeed, we have the expression
$$\hat{\bf \Psi}_{\rm reg}^{l}(\gn|\a; g) =
\frac{(-1)^{\#S}\{\prod_{v\in \Sigma_{\infty}}2^{l_{v}-1}\}C_{l}(0)
D_{F}^{-1/2}}{[\bfK_{\fin}:\bfK_{0}(\gn)]}
\sum_{\pi \in \Pi_{\rm cus}(l, \gn)}\sum_{\varphi \in \Bcal(\pi;l, \gn)}
\a(\nu_{S}(\pi)) \overline{P_{\rm reg}^{\bf 1}(\varphi)} \varphi(g),
$$
which is valid pointwisely with the summation being finite. By computing the period integral $P_{\rm{reg}}^{\eta}(\hat{\bf \Psi}_{\rm reg}^{l}(\gn|\a))$ in terms of this expansion, we obtain the one side of the relative trace formula, the {\it spectral side}.

\begin{prop} \label{Reg-per}
Suppose $\ul \ge 4$.
The function $\hat{\bf \Psi}_{\rm reg}^{l}(\gn|\a)$ has the regularized $(H, \eta)$-period given by
\begin{align*}
P_{\rm reg}^{\eta}(\hat{\bf \Psi}_{\rm reg}^{l}(\gn|\a)) =&
(-1)^{\#S}
\{ \prod_{v \in \Sigma_{\infty}}2\pi \frac{\Gamma(l_{v}-1)}{\Gamma(l_{v}/2)^{2}} \}
D_{F}^{-1}[\bfK_{\fin}:\bfK_{0}(\gn)]^{-1} \times (-1)^{\e(\eta)}\Gcal(\eta) \\
& \times \sum_{\pi \in \Pi_{\rm cus}(l,\gn)}
w_{\gn}^{\eta}(\pi) \frac{L(1/2, \pi) L(1/2, \pi \otimes \eta)}
{2{\rm N}(\gf_{\pi})[\bfK_{\fin}: \bfK_{0}(\gf_{\pi})]^{-1} L^{S_{\pi}}(1, \pi; {\rm Ad})} \a(\nu_{S}(\pi)).
\end{align*}
\end{prop}
\begin{proof} As was remarked in \S 6.3, for a cups form $\varphi$, the regularized period $P_{\rm{reg}}^\eta(\varphi)$ becomes the usual zeta integral $Z^{*}(1/2,\eta,\varphi)$ which is absolutely convergent. Thus, by term-wise integration, we have
{\small \begin{align*}
P_{\rm reg}^{\eta}(\hat{\bf \Psi}_{\rm reg}^{l}(\gn|\a)) = &
\frac{(-1)^{\#S}\{\prod_{v\in \Sigma_{\infty}}2^{l_{v}-1}\}C_{l}(0)
D_{F}^{-1/2}}{[\bfK_{\fin}:\bfK_{0}(\gn)]}\sum_{\pi \in \Pi_{\rm cus}(l,\gn)}\sum_{\varphi \in \Bcal(\pi; l, \gn)}
\overline{P_{\rm reg}^{\bf 1}(\varphi)} P_{\rm reg}^{\eta}(\varphi)\a(\nu_{S}(\pi)).
\end{align*}}
Then we obtain the assertion by Lemma \ref{C_l(z)} (ii), Proposition \ref{periods of cusp forms}, Lemmas \ref{value of PP} and \ref{adjoint L}.
\end{proof}

\section{Geometric expansions}

The reader might wonder why we take the regularized period in Proposition~\ref{Reg-per} which seems unnecessary because the function $\hat{\mathbf\Psi}_{\rm{reg}}^{l}(\fn|\a)$ is cuspidal and the usual period integral makes sense. The reason should become evident from our computation of the other side of the relative trace formula, the {\it geometric side}, to be performed in this section and the next. Suppose $\ul=\inf_{v\in \Sigma_\infty} l_{v} \ge 4$.  We compute the quantity $P^\eta_{\rm{reg}}(\hat{\bf \Psi}_{\rm reg}^{l}(\gn|\a))$ by using the series expression \eqref{centralObj}. The first step is to break the sum in \eqref{centralObj} over $H_F\bsl G_F$ to a sum of subseries according to double cosets $H_F\delta H_F$. For $\delta \in G_{F}$, we put ${\rm St(\delta)} := H_{F} \cap \delta^{-1}H_{F}\delta$. Then, the following elements of $G_{F}$ form a complete set of representatives for the double coset space $H_{F} \backslash G_{F}/H_{F}$:
\begin{center}
$e = [\begin{smallmatrix}1 & 0 \\ 0 & 1\end{smallmatrix}]$,
$w_{0} = [\begin{smallmatrix} 0 & -1 \\ 1 & 0\end{smallmatrix}]$,

$u = [\begin{smallmatrix}1 & 1 \\ 0 & 1\end{smallmatrix}]$,
$\overline{u} = [\begin{smallmatrix}1 & 0 \\ 1 & 1\end{smallmatrix}]$,
$uw_{0} = [\begin{smallmatrix}1 & -1 \\ 1 & 0\end{smallmatrix}]$,
$\overline{u}w_{0}= [\begin{smallmatrix}0 & -1 \\ 1 & -1\end{smallmatrix}]$,

$\delta_{b} = [\begin{smallmatrix} 1+b^{-1} & 1 \\ 1 & 1\end{smallmatrix}], \hspace{2mm}{b\in F^{\times} - \{-1\}}$.
\end{center}
Moreover, we have ${\rm St}(e) ={\rm St}(w_{0}) = H_{F}$ and ${\rm St}(\delta) = Z_{F}$ for any $\delta \in \{u, \overline{u},
uw_{0}, \overline{u}w_{0}\} \cup \{\delta_{b} | b \in F^{\times} - \{-1\} \}$. (See \cite[Lemma 1]{Ramakrishnan-Rogawski} and \cite[Lemma 73]{Tsud}). Thus we obtain the following expression for $\Re(\l) > 0$:
$$
\hat{{\bf \Psi}}_{\b, \l}^{l} \left(\gn | \a; [\begin{smallmatrix} t & 0 \\ 0 & 1\end{smallmatrix}]
[\begin{smallmatrix}1 & x_{\eta} \\ 0 & 1\end{smallmatrix}]\right)
= \sum_{\delta}J_{\delta}(\b, \l, \a; t),
$$
where $\delta$ runs through the double coset representatives listed above and, for each such $\delta$, $J_{\delta}(\b, \l, \a; t)$ is the sum of $\hat{\Psi}_{\b, \l}^{l} \left(\gn | \a; 
\delta \gamma
[\begin{smallmatrix} t & 0 \\ 0 & 1\end{smallmatrix}]
[\begin{smallmatrix}1 & x_{\eta} \\ 0 & 1\end{smallmatrix}]\right)
$ for $\gamma\in {\rm{St}}(\delta)\bsl H_F$.

\begin{lem} \label{JJ-id}
The function $\l \mapsto J_{e}(\b, \l, \a; t)$ and $\l \mapsto J_{w_{0}}(\b, \l, \a; t)$
are entire on $\CC$. Moreover their values at $\l = 0$ are $J_{\rm id}(\a; t)\b(0)$ and $i^{\tilde{l}}\delta(\gn = \go)J_{\rm id}(\a; t)\b(0)$, respectively, where
$$
J_{\rm id}(\a; t) = \delta(\gf = \go)\left( \frac{1}{2\pi i} \right)^{\#S}\int_{\LL_{S}(\bfc)}\Upsilon_{S}^{\bf1}(\bfs) \a(\bfs) d\mu_{S}(\bfs)$$
with $\tilde{l}= \sum_{v \in \Sigma_{\infty}}l_{v}$ and
$$
\Upsilon_{S}^{\bf1}(\bfs) = \prod_{v \in S}(1 - q_{v}^{-(s_{v}+1)/2})^{-1}(1 - q_{v}^{(s_{v}+1)/2})^{-1}.$$
\end{lem}
\begin{proof}
Since $\Psi_{v}^{(0)}(l_{v}; 1_{2} ) = 1$ for all $v \in \Sigma_{\infty}$, the assertion is proved in the same way as \cite[Lemma 11.2]{Tsud}. 
\end{proof}

We put
$$J_{{\rm u}}(\b, \l, \a; t) = J_{u}(\b, \l, \a, t) + J_{\overline{u}w_{0}}(\b, \l, \a, t)$$
and
$$J_{\bar{\rm u}}(\b, \l, \a; t) = J_{uw_{0}}(\b, \l, \a, t) + J_{\bar{u}}(\b, \l, \a, t).$$

\begin{lem} \label{JJ-unipotent}
For $* \in \{{\rm u}, \bar{\rm u} \}$, the function $\l \mapsto J_{*}(\b, \l, \a; t)$ on $\Re(\l)>0$ has a holomorphic continuation to $\C$ whose value at $\l=0$ is equal to
$J_{*}(\a, t)\b(0)$, where
{\small \begin{align*}
J_{\rm u}(\a; t) = & \left( \frac{1}{2\pi i} \right)^{\#S}\sum_{a \in F^{\times}}\int_{\LL_{S}(\bfc)}
\bigg\{ \Psi_{l}^{(0)}\left(\gn|\bfs; [\begin{smallmatrix}1 & at^{-1} \\ 0 & 1 \end{smallmatrix}][\begin{smallmatrix}1 & x_{\eta} \\ 0 & 1 \end{smallmatrix}]\right) +
\Psi_{l}^{(0)}\left(\gn|\bfs; [\begin{smallmatrix}1 & 0 \\ at^{-1} & 1 \end{smallmatrix}][\begin{smallmatrix}1 & 0 \\ -x_{\eta} & 1 \end{smallmatrix}]w_{0}\right)
\bigg\} \a(\bfs)d \mu_{S}(\bfs)
\end{align*}
}and
{\small \begin{align*}
J_{\bar{\rm u}}(\a; t) = & \left( \frac{1}{2\pi i} \right)^{\#S}\sum_{a \in F^{\times}}\int_{\LL_{S}(\bfc)}
\bigg\{ \Psi_{l}^{(0)}\left(\gn|\bfs; [\begin{smallmatrix}1 & 0 \\ at & 1 \end{smallmatrix}][\begin{smallmatrix}1 & x_{\eta} \\ 0 & 1 \end{smallmatrix}]\right) +\Psi_{l}^{(0)}\left(\gn|\bfs; [\begin{smallmatrix}1 & at \\ 0 & 1 \end{smallmatrix}][\begin{smallmatrix}1 & 0 \\ -x_{\eta} & 1 \end{smallmatrix}]w_{0}\right)
\bigg\} \a(\bfs)d \mu_{S}(\bfs).
\end{align*}}
\end{lem}
\begin{proof} We follow the proof of \cite[Lemma 11.3]{Tsud}. Take $\s >0$ such that $l_{v}/2 >\s+1$. Let us examine $J_{\rm u}(\b, \l, \a; t)$.
First we consider the sum of the functions
$\hat{\Psi}_{\b, \l}^{l}\left(\gn |\a; u[\begin{smallmatrix}t & 0 \\ 0 & 1 \end{smallmatrix}][\begin{smallmatrix}a^{-1} & 0 \\ 0 & 1 \end{smallmatrix}]
[\begin{smallmatrix}1 & x_{\eta} \\ 0 & 1 \end{smallmatrix}]\right)=$
{\small 
\begin{align*}
&\left(\tfrac{1}{2\pi i}\right)^{\#S}{\textstyle{\int_{\LL_{S}(\bfc)}}}
\bigg\{\tfrac{1}{2\pi i}{\textstyle{\int_{L_{\s}}}}\tfrac{\b(z)}{z+\l}\{|t|_{\AA}^{z}\Psi^{(z)}\left(\gn |\bfs ; [\begin{smallmatrix}1 & at^{-1} \\ 0 & 1 \end{smallmatrix}][\begin{smallmatrix}1 & x_{\eta} \\ 0 & 1 \end{smallmatrix}] \right) 
+|t|_{\AA}^{-z}\Psi^{(-z)}(\gn |\bfs; [\begin{smallmatrix}1 & at^{-1} \\ 0 & 1 \end{smallmatrix}] [\begin{smallmatrix}1 & x_{\eta} \\ 0 & 1 \end{smallmatrix}])\} dz \bigg\}\a(\bfs)d \mu_{S}(\bfs)
\end{align*}
}over all $a\in F^\times$. Here $\bfc$ is taken so that $q(\bfc)$ is sufficiently large. There exists an ideal $\ga$ of $F$
such that the estimate
$$\left|\Psi_{l}^{(\pm z)}\left(\gn|\bfs; [\begin{smallmatrix}1 & at^{-1} \\ 0 & 1 \end{smallmatrix}][\begin{smallmatrix}1 & x_{\eta} \\ 0 & 1 \end{smallmatrix}]\right)\right|
\prec f(a), \hspace{5mm} a \in F^{\times}, (\bfs, z) \in \LL_{S}(\bfc)\times L_{\s}$$
holds, where
$$f(a)= \prod_{v \in \Sigma_{\infty}}|1+ia_{v}t_{v}^{-1}|_{v}^{\s-l_{v}/2}
\prod_{v \in S}\sup(1, |a_{v}t_{v}^{-1}|_{v})^{-(2q(\bfc)-\s)}
\prod_{\Sigma_{\fin}-S}\delta(a_{v} \in \ga\go_{v}), \quad a\in \AA.$$
Thus to establish the absolute convergence of the sum of 
$\hat{\Psi}_{\b, \l}^{l}\left(\gn |\a; u[\begin{smallmatrix}t & 0 \\ 0 & 1 \end{smallmatrix}][\begin{smallmatrix}a^{-1} & 0 \\ 0 & 1 \end{smallmatrix}]
[\begin{smallmatrix}1 & x_{\eta} \\ 0 & 1 \end{smallmatrix}]\right)$ over $a\in F^\times$, it is enough to show $\sum_{a\in F^\times} f(a)<+\infty$. The convergence of the latter sum in turn follows from the convergence of the integral $\int_{\AA}f(a)da$, which is a product of the archimedean integrals for all $v \in \Sigma_{\infty}$ convergent when $l_{v}/2 - \s > 1$ and the non-archimedean ones convergent for sufficiently large $q(\bfc)$.

The sum of the functions $\hat{\Psi}_{\b, \l}^{l}\left(\gn |\a; \bar{u}w_{0}[\begin{smallmatrix}t & 0 \\ 0 & 1 \end{smallmatrix}][ \begin{smallmatrix}a^{-1} & 0 \\ 0 & 1 \end{smallmatrix}][\begin{smallmatrix}1 & x_{\eta} \\ 0 & 1 \end{smallmatrix}]\right)$ over $a\in F^\times$ is analyzed similarly. By the estimate 
$$\left|\Psi_{l}^{(\pm z)}\left(\gn|\bfs; [\begin{smallmatrix}1 & 0 \\ at^{-1} & 1 \end{smallmatrix}][\begin{smallmatrix}1 & 0 \\ -x_{\eta} & 1 \end{smallmatrix}]w_{0}\right)\right|
\prec f(a), \hspace{5mm} a \in F^{\times}, (\bfs, z) \in \LL_{S}(\bfc)\times L_{\s},$$
the problem is reduced to the convergence of the same series $\sum_{a\in F^\times} f(a)$ as above. Hence the assertion on $J_{\rm u}(\b, \l, \a; t)$ is obtained. The integral $J_{\bar{\rm u}}(\b, \l, \a; t)$ is examined in the same way. This completes the proof.
\end{proof}

\subsection{Hyperbolic terms}
We consider the convergence of
$$J_{\rm hyp}(\b, \l, \a; t)=\sum_{b \in F^{\times}-\{-1\}}J_{\delta_{b}}(\b, \l, \a; t).$$

Let $v \in \Sigma_{\infty}$. For $t\in F_{v}^{\times}$, $b\in F_{v}^\times-\{-1\}$ and $\sigma,\,\rho\in \R$, set
$$f^{(\s)}(l_{v}; t,b) = \{(b+1)^{2}t^{2} + b^{2}\}^{\s/2-l_{v}/4}(1+t^{-2})^{-\s/2-l_{v}/4}|t|_{v}^{-2\s}$$
and
$$M_{v}(\s, \rho,l_{v}; b) = |b+1|_{v}^{-(\s-\rho)_{-}}|b|_{v}^{l_{v}/4-\s/2} \times \int_{F_{v}^{\times}}f^{(\s)}(l_{v}; t,b)|t|_{v}^{\s+\rho}d^{\times}t,$$
where $q_{-}=\inf(0,q)$ for $q\in \R$. 

\begin{lem} \label{hyper(rev)0}
Let $v \in \Sigma_{\infty}$. Then, for any $\s \in \RR$ we have
$$\left|\Psi_{v}^{(z)}\left(l_{v}; [\begin{smallmatrix}1+b^{-1}&1 \\1&1 \end{smallmatrix}] [\begin{smallmatrix}t&0\\0&1 \end{smallmatrix}]\right) \right| \le 
|b|_{v}^{-\s}|t|_{v}^{\s}e^{\pi|\Im(z)|/2} f_{v}^{(\s)}(l_{v}; t, b), \hspace{5mm} t\in F_{v}^{\times}, \ b\in F_{v}^{\times}-\{-1 \}, \ z \in L_{\s}.$$
\end{lem}
\begin{proof} By writing the Iwasawa decomposition $[\begin{smallmatrix}1+b^{-1}&1 \\1&1 \end{smallmatrix}] [\begin{smallmatrix}t&0\\0&1 \end{smallmatrix}]=\left[\begin{smallmatrix} r & 0 \\ 0 & r^{-1} \end{smallmatrix}\right] [\begin{smallmatrix}1&x\\0&1 \end{smallmatrix}] k$ explicitly, we have $|1+ix|=(1+t^{-2})^{1/2}((b+1)^{2}t^{2}+b^{2})^{1/2}$. Then the assertion follows from Proposition \ref{Sh-ftn} and Lemma \ref{Shintani-val-unipotent}. 
\end{proof}

\begin{lem}\label{hyper(rev)1}
Let $v\in\Sigma_\infty$ and $l_{v} \in 2\Z_{\geq 2}$. Let $\s, \rho \in \RR$. Then the estimate
\begin{align}
M_{v}(\s, \rho,l_{v}; b) \prec |b+1|_{v}^{-l_{v}/4+\s/2-(\s-\rho)_{-}}, \hspace{5mm} b \in F_{v}^{\times}-\{-1\}
 \label{hyper(rev)1-0}
\end{align}
holds if $l_{v}/4>|\rho|-\s/2$ and $l_{v}/4>\s/2$. Moreover, for $\e>0$ and $c \in \RR$, the function
$|b(b+1)|_v^{\epsilon}\,
|b|_{v}^{-l_{v}/4+(c+1)/4}
\,M_{v}(\s,\rho,l_{v}; b)$ in $b\in F_{v}$ is locally bounded if
\begin{align}
\big| |\rho|-\sigma \big|+(\sigma-\rho)_{-}<\epsilon/3<1, \quad l_{v}/4>\s/2-(\sigma-\rho)_{-}+1,
\quad (c+1)/4>\s/2-(\sigma-\rho)_{-}.
 \label{hyper(rev)1-1}
\end{align}
\end{lem}
\begin{proof}
The assertion is proved in a similar way to \cite[Lemma 11.14]{Tsud}.
By $b^{2}+t^{2}(b+1)^{2} \ge 2|b|\ |b+1|\ |t|$ and $\s/2-l_{v}/4<0$,
we estimate
\begin{align*}
M_{v}(\s, \rho,l_{v}; b) \prec &
|b+1|^{-(\s-\rho)_{-}}|b|^{l_{v}/4-\s/2} \times \int_{0}^{\infty}
\{|b||b+1||t|\}^{\s/2-l_{v}/4}(1+t^{-2})^{-\s/2-l_{v}/4}
|t|^{-\s+\rho}d^{\times}t \\
= & |b+1|^{-l_{v}/4+\s/2-(\s-\rho)_{-}} \int_{0}^{\infty}
|t|^{\rho+l_{v}/4+\s/2}(1+t^{2})^{-\s/2-l_{v}/4}d^{\times}t.
\end{align*}
The integral converges absolutely if $l_{v}/4>|\rho|-\s/2$. In the same way as in the proof of \cite[Lemma 11.14]{Tsud}, we have
\begin{align*}
& |b(b+1)|^{\epsilon}\,|b|^{-l_v/4+(c+1)/4}\,
M_{v}(\s, \rho,l_{v}; b) \\
\ll & |b+1|^{\sigma-|\rho|-(\sigma-\rho)_{-}+\epsilon/3}|b|^{(c+1)/4+\s/2-|\rho|+\epsilon/3}\,|b(b+1)|^{\epsilon/3}{\frak m}(r;b(b+1)),\end{align*}
where $r=l_v+2\sigma-4|\rho|-4\epsilon/3$ and $\fm(r;b(b+1))=\int_{0}^{\infty} [(1+t^{-2})(b^2+t^2(b+1)^2\}]^{-r/4}\,\d^\times t$. By \cite[Lemma 15.5]{Tsud}, the function $|b(b+1)|^\epsilon \,{\frak m}(r;b(b+1))$ (with $r>0$) is locally bounded on $F_{v}$. From this, $|b(b+1)|^\epsilon |b|^{-l_v/4+(c+1)/4} M_{v}(\s, \rho,l_{v}; b)$ is also locally bounded on $F_{v}$ if 
$$\sigma-|\rho|-(\sigma-\rho)_{-}+\epsilon/3\geq 0,\quad r=l_v+2\sigma-4|\rho|-4\epsilon/3>0,\quad
(c+1)/4+\s/2-|\rho|+\e/3\ge 0.
$$
This condition is satisfied by \eqref{hyper(rev)1-1}. Thus, under \eqref{hyper(rev)1-1}, the estimate \eqref{hyper(rev)1-0} is extendable to $F_{v}$; from this, the last assertion is obvious. 
\end{proof}

Let $\bfc=(c_v)_{v\in S} \in\RR^{S}$, $l=(l_v)_{v\in \Sigma_\infty} \in (2\Z_{\geq 2})^{\Sigma_\infty}$, $t\in \A^\times$, $b\in F^\times-\{-1\}$ and $\sigma,\,\rho\in \R$. For $v\in S$, we put
\begin{align*}
f_v^{(\sigma)}(c_v;t_v,b)&=\inf(1,|t_v|_v^{-2})^{\sigma} 
\begin{cases}
 \sup(1,|t_v|_v^{-1}|b|_v)^{-(c_v+1)/2+\sigma}, \quad (|t_v|_v\leq 1), \\
 \sup(1,|t_v|_v|b+1|_v)^{-(c_v+1)/2+\sigma}, \quad (|t_v|_v>1),
\end{cases}
\\
M_v(\sigma,\rho,c;b)&=\sup(1,|b+1|_v)^{-(c+1)/4+\sigma/2+|\sigma-\rho|}, 
\end{align*}
and for $v\in \Sigma_\fin-S$, we put
\begin{align*}
f_v^{(\sigma)}(t_v,b)&=\inf(1,|t_v|^{-2})^{\s}\,\delta(b\in\fp_v^{-f(\eta_v)},q_v^{-2f(\eta_v)}|b|_v\leq |t_v|_v\leq |b+1|_v^{-1}).
\end{align*}
Then, define
{\small \begin{align*}
N(\gn|\s, l, \bfc; t,b)= & |t|_{\AA}^{\s} \prod_{v \in \Sigma_{\infty}} f_{v}^{(\s)}(l_{v}; t_{v}, b) \prod_{v \in S} f_{v}^{(\s)}(c_{v}; t_{v}, b)
\prod_{v \in S(\gn)}\delta(t_{v} \in \gn \go_{v})f_{v}^{(\s)}(t_{v}, b)\prod_{v \in \Sigma_{\fin}- (S \cup S(\gn))}f_{v}^{(\s)}(t_{v}, b),
\end{align*}}
{\small \begin{align*}
M(\gn|\s, \rho, l, \bfc; b) = & \prod_{v \in \Sigma_{\infty}} |b|_{v}^{-l_{v}/4+\s/2} M_{v}(\s, \rho, l_{v}; b)
\prod_{v \in S}|b|_{v}^{-(c_{v}+1)/4+\s/2}M_{v}(\s, \rho; c_{v}, b)\\
& \times
\prod_{v \in \Sigma_{\fin}}\sup(1, |b|_v^{\sigma+\rho})
\prod_{v \in \Sigma_{\fin}-S} \delta(b \in \gf^{-1}\gn\go_{v})
\end{align*}
}and
$M_{\e}(\gn|\s, \rho, l, \bfc; b) = \{\prod_{v\in \Sigma_\infty}|b(b+1)|_v^{\epsilon}\}\,M(\gn|\s, \rho, l, \bfc; b)$
for $\e \ge 0$.
By closely following \cite[11.4]{Tsud}, we have the following series of lemmas.
\begin{lem}\label{hyper(rev)2}
If $q(\bfc) > |\s|+1$, then we have
$$\left| \Psi_{l}^{(z)}\left(\gn |\bfs; \delta_{b}[\begin{smallmatrix}t&0\\ 0&1\end{smallmatrix}][\begin{smallmatrix}1&x_{\eta}\\ 0&1\end{smallmatrix}]\right) \right| \prec N(\gn|\s,l, \bfc; t,b)e^{d_{F}\pi|\Im(z)|/2}$$
uniformly in $(z, \bfs) \in L_{\s}\times \LL_{S}(\bfc),\ b \in F^{\times}-\{-1\},\ t \in \AA^{\times}.$
\end{lem}
\begin{proof}
 This follows from Lemma~\ref{hyper(rev)0} and \cite[Corollary 11.6, Lemma 11.10]{Tsud}.
\end{proof}

\begin{lem}\label{hyper(rev)3}
If $q(\bfc)>|\s|+|\rho|+1$, $\ul/4>\sup(\s/2 , |\rho|-\s/2)$ and $\sigma\not=\pm\rho$, then we have
$$\int_{\AA^{\times}}N(\gn|\s, l, \bfc; t,b)|t|_{\AA}^{\rho}d^{\times}t \prec_{\e} M_{\e}(\gn|\s, \rho, l, \bfc; b)\nr(\gn)^{\e}, \hspace{5mm}b \in F^{\times}-\{-1\}$$
for any $\e>0$, with the implied constant independent of the ideal $\fn$.
\end{lem}
\begin{proof}
We can apply the same argument in \cite[Lemma 11.16]{Tsud}
by using $l_{\iota}$ in place of $c_{\iota}+1$ for all $\iota \in \Sigma_{\infty}$.
\end{proof}

\begin{lem}\label{hyper(rev)4}
Let $U$ be a compact subset of $\AA^{\times}$.
If $q(\bfc)>|\s|+|\rho|+1$, $\ul/4>\sup(\s/2 , |\rho|-\s/2)$ and $\sigma\not=\pm\rho$, then we have
$$\sum_{t \in F^{\times}}N(\gn|\s, l, \bfc; t,b) \prec_\e M_{\e}(\gn|\s, \rho, l, \bfc; b)\,\nr(\gn)^{\e},
\hspace{5mm} b \in F^{\times}, t \in U$$
for any $\e>0$, with the implied constant independent of the ideal $\fn$.
\end{lem}
\begin{proof}
This follows from Lemma \ref{hyper(rev)3} and the argument in \cite[Corollary 11.17]{Tsud}.
\end{proof}

\begin{lem}\label{hyper(rev)6}
If $\s+\rho>-1$, $\s\not=\pm\rho$, $(c+1)/4 >5|\s|/2+2|\rho|+1$, 
$\ul/4>|\s|+|\rho|+1$ and
$\ul/2>(c+1)/4 +3|\s|/2+|\rho|+1$ hold,
then, we have the estimate
$$\sum_{b\in F^{\times}-\{-1\}}M_{\e}(\gn|\s, \rho, l, {\underline c} ; b) \prec
{\rm N}(\gn)^{-(c+1)/4+\s/2+|\s+\rho|}
$$
for any $\e>0$ such that $\big| |\rho|-\sigma \big|+(\sigma-\rho)_{-}<\epsilon/3<1$ and $\ul/2>(c+1)/4 +3|\s|/2+|\rho| +1+ 2\e$, with the implied constant independent of $\gn$. Here ${\underline c}=(c_v)_{v\in S}$ with $c_v=c\,(\forall v\in S)$. 
\end{lem}
\begin{proof}
	We give a proof in a similar way to \cite[Lemma 11.19]{Tsud}, replacing $c_{\iota}+1$ with $l_{\iota}$	for all $\iota \in \Sigma_{\infty}$.
Under the assumption on $l, \s, \rho, c$ in this lemma, the series
	\begin{align*}
	\sum_{b \in \go(S)-\{-1\}} & 
	\{ \prod_{v \in \Sigma_{\infty}}|b|_{v}^{-l_{v}/4+(c+1)/4} M_{v}(\s, \rho, l_{v}; b)\} 
	\{\prod_{v \in S}\sup(1, |b|_{v}^{\s+\rho}) M_{v}(\s, \rho, c_{v};b) \}\,|\nr(b(b+1))|^{\e},
	\end{align*}
	which is denoted by $A_{S}(\s,\rho, l, c)$, converges for any $\e>0$ such that $\big| |\rho|-\sigma \big|+(\sigma-\rho)_{-}<\epsilon/3<1$ and $\ul/2>(c+1)/4 +3|\s|/2+|\rho| +1+ 2\e$. Here $\go(S)$ denotes the $S$-integer ring of $F$. Indeed, this follows from Lemma \ref{hyper(rev)1} and \cite[Lemma 11.18]{Tsud}. By noting the Artin product formula $|b|_{\AA}=1$ for $b \in F^{\times}$,
we have
{\small \allowdisplaybreaks
\begin{align*}
	& \sum_{b\in F^{\times}-\{-1\}}M_{\e}(\gn|\s, \rho, l, {\underline c} ; b) \\
	= & \sum_{b \in \gf^{-1}\gn\go(S) -\{0, -1\}}\{\prod_{v\in \Sigma_{\fin}}\sup(1, |b|_{v}^{\s+\rho})\}\{\prod_{v \in S}|b|_{v}^{-(c+1)/4+\s/2}M_{v}(\s, \rho, c; b)\}\\
	& \times\{\prod_{v \in \Sigma_{\infty}}|b|^{-l_{v}/4+\s/2}M_{v}(\s, \rho, l_{v}; b )\}|\nr(b(b+1))|^{\e} \\
	=&\sum_{\substack{ b \in \gf^{-1}\gn\go(S)/\go(S)^{\times}\\ b\neq 0, -1}}\{\prod_{v\in \Sigma_{\fin}-S}\sup(1, |b|_{v}^{\s+\rho}) |b|_{v}^{(c+1)/4-\s/2}\} \\
	& \times\sum_{u \in \go(S)^{\times}}\{\prod_{v \in \Sigma_{\infty}}|ub|^{-l_{v}/4+(c+1)/4}M_{v}(\s, \rho, l_{v}; ub )\}
	\{\prod_{v \in S}\sup(1, |ub|_{v}^{\s+\rho})M_{v}(\s, \rho, c; ub)\} \\
	&\times |\nr(ub(ub+1))|^{\e} \\
	\ll & \sum_{\substack{b \in \gf^{-1}\gn\go(S)/\go(S)^{\times} \\ b\neq 0, -1}}\{\prod_{v\in \Sigma_{\fin}-S}\sup(1, |b|_{v}^{\s+\rho}) |b|_{v}^{(c+1)/4-\s/2}\} \times A_{S}(\s,\rho, l, c).
\end{align*}
}We note that the series in the last line is majorized by ${\rm N}(\gn)^{-(c+1)/4+\s/2+|\s+\rho|}$ as in the proof of \cite[Lemma 11.19]{Tsud}.
\end{proof}

\begin{lem} \label{JJ-hyperbolic}
Let $l=(l_{v})_{v \in \Sigma_{\infty}}\in (2\ZZ_{\ge2})^{\Sigma_{\infty}}$ and $c, \s \in \RR$.
Assume the following conditions:
\begin{center}
$\ul\ge 6$, \quad $\s>-1$,
\quad $(c+1)/4 > 9|\s|/2 +1,$
\quad $\ul/2 > (c+1)/4+5|\s|/2+1.$\end{center}
Then, for any compact subset $U$ of $\AA^{\times}$, the series
$$
\sum_{b \in F^{\times}-\{-1\}}\sum_{a \in F^{\times}}
\left| \Psi_{l}^{(z)}\left(\gn |\bfs; \delta_{b}[\begin{smallmatrix}at&0\\ 0&1\end{smallmatrix}][\begin{smallmatrix}1&x_{\eta}\\ 0&1\end{smallmatrix}]\right) \right|$$
converges uniformly in $(t, z, \bfs) \in U \times L_{\s} \times \LL_{S}(\underline{c})$, and there exists $\epsilon>0$ such that, for $\rho \in \RR$ satisfying $0<||\rho|-\sigma|<\epsilon$ and $\s+\rho>-1$, the integral
$$\sum_{b \in F^{\times}-\{-1\}}\int_{t \in \AA^{\times}}
\left| \Psi_{l}^{(z)}\left(\gn |\bfs; \delta_{b}[\begin{smallmatrix}t&0\\ 0&1\end{smallmatrix}][\begin{smallmatrix}1&x_{\eta}\\ 0&1\end{smallmatrix}]\right) \right||t|_{\AA}^{\rho}d^{\times}t$$
converges uniformly in $(z, \bfs) \in L_{\s} \times \LL_{S}(\underline{c})$.
\end{lem}
\begin{proof}
By assumption, we can take $\rho \in \RR$ such that $(c+1)/4 > 5|\s|/2+2|\rho|+1$, $\s+\rho>-1$, $\ul/4>|\s|+|\rho|+1$
and $\ul/2 > (c+1)/4 + 3|\s|/2 + |\rho| +1$
(we can take $\rho=0$ if $\s>-1$ and $\s\neq 0$ ).
Thus the assertion follows from Lemmas~\ref{hyper(rev)2}, \ref{hyper(rev)3}, \ref{hyper(rev)4}, and \ref{hyper(rev)6}.
We remark that the condition $\ul \ge 6$ is forced by the third and the fourth inequalities in Lemma \ref{JJ-hyperbolic}; indeed, they imply
$\ul/2 > 7|\s| +2$, and hence
$\ul> 4$.
\end{proof}

\begin{lem} \label{JJ-hyperbolic2}
Suppose $\ul \ge 6$. The function $J_{\rm hyp}(\b, \l, \a;t)$ on $\Re(\l)>1$ has a holomorphic continuation to $\CC$ whose value at $\l=0$ equals $J_{\rm hyp}(\a;t)\b(0)$, where
$$
J_{\rm hyp}(\a; t) = \sum_{b \in F^{\times}-\{-1\}}\sum_{a \in F^{\times}}
\hat{\Psi}_{l}^{(0)}\left(\gn |\a; \delta_{b}[\begin{smallmatrix}at & 0\\ 0&1\end{smallmatrix}][\begin{smallmatrix}1&x_{\eta}\\ 0&1\end{smallmatrix}]
\right).$$
The series converges absolutely and uniformly in $t \in \AA^{\times}$.
Here we set
$$\hat{\Psi}_{l}^{(0)}\left(\gn |\a; g\right)= \left(\frac{1}{2\pi i}\right)^{\#S}\int_{\LL_{S}(\bfc)}\Psi_{l}^{(0)}(\gn |\bfs; g) \a(\bfs)d\mu_{S}(\bfs)$$
with $\bfc$ being an element of $\RR^S$ such that $q(\bfc)>1$ $(${\rm cf}.\ \cite[\S 6.3]{Tsud}$)$.
\end{lem}
\begin{proof}
This follows from Lemma \ref{JJ-hyperbolic} in the same way as \cite[Lemma 11.21]{Tsud}.
\end{proof}

From Lemmas~\ref{JJ-id}, \ref{JJ-unipotent} and \ref{JJ-hyperbolic2}, we have
\begin{align}
\hat{{\bf \Psi}}_{\rm reg}^{l}\left(\gn | \a; [\begin{smallmatrix}t&0\\0&1\end{smallmatrix}][\begin{smallmatrix}1&x_{\eta}\\0&1\end{smallmatrix}]\right)
= (1+i^{\tilde{l}}\delta(\gn=\go))J_{\rm id}(\a;t) + J_{\rm u}(\a;t)+J_{\bar{\rm u}}(\a;t) +J_{\rm hyp}(\a;t)
 \label{GEOMETRIC IDENTITY}
\end{align} 
for any $t\in \A^\times$. Some terms on the right-hand side, viewed as functions on $H_F\bsl H_\A$ individually, have divergent $(H,\eta)$-period integrals; to proceed further, we need to regularize them.

\section{Geometric side}

Suppose $\ul=\inf_{v \in \Sigma_{\infty}} l_{v} \ge 6$.
We fix a holomorphic function $\alpha(\bfs)$ on $\fX_S$ such that $\alpha(\varepsilon\bfs)=\alpha(\bfs)$ for any $\varepsilon \in \{\pm 1\}^S$. Let $\beta\in \Bcal$ as before. For $\natural \in \{\rm id, u, \bar{u}, hyp\}$,
we set
$$\JJ_{\natural}^{\eta}(\b, \l; \a)= \int_{F^{\times}\backslash \AA^{\times}} J_{\natural}(\a;t)\{\hat{\b}_\lambda(|t|_{\AA})
+\hat{\b}_\lambda(|t|_{\AA}^{-1})\}\eta(t x_{\eta}^{*})d^{\times}t.$$
 In this section, we shall show that this integral converges absolutely when $\Re(\lambda)\gg 0$ and has a meromorphic continuation to a neighborhood of $\lambda=0$; at the same time, we determine the constant term in its Laurent expansion at $\lambda=0$. As a result, by the identity 
$$
P_{\beta,\lambda}^\eta(\hat{{\bf \Psi}}_{\rm reg}^{l}(\gn|\alpha))
=
\JJ_{\rm id}^{\eta}(\b, \l; \a)
+i^{\tilde l}\delta(\gn=\go)\JJ_{\rm id}^{\eta}(\b, \l; \a)
+\JJ_{\rm u}^{\eta}(\b, \l; \a)
+\JJ_{\rm \bar u}^{\eta}(\b, \l; \a)
+\JJ_{\rm hyp}^{\eta}(\b, \l; \a)
$$
obtained from \eqref{GEOMETRIC IDENTITY}, we have another expression of $P_{\rm{reg}}^\eta(\hat{{\bf \Psi}}_{\rm reg}^{l}(\gn|\alpha))$ already computed in Proposition~\ref{Reg-per} by means of the spectral expansion. 
   
\begin{lem}\label{IDEMPOTENT TERM}
For $\Re(\l)>0$, the integral $\JJ_{\rm id}^{\eta}(\b, \l; \a)$ converges absolutely
and we have
$$\JJ_{\rm id}^\eta(\b, \l; \a)= \delta_{\eta, \bf 1}\vol(F^{\times}\backslash \AA^{1})
\left(\frac{1}{2\pi i}\right)^{\#S}\int_{\LL_{S}(\bfc)}\Upsilon_{S}^{\bf 1}(\bfs)\a(\bfs)d\mu_{S}(\bfs)\frac{2\b(0)}{\l},$$
where $\delta_{\eta, \bf1} = \delta(\eta = \bf1)$. We have a meromorphic continuation of $\JJ_{\rm id}^{\eta}(\b, \l; \a)\,(\Re(\l)>0)$ to $\C$ with ${\rm CT}_{\l=0}\JJ_{\rm id}^\eta(\b, \l; \a) =0$. 
\end{lem}
\begin{proof} The first claim is shown in the same way as \cite[Lemma 12.2]{Tsud}. From the expression, the function $\JJ_{\rm id}^\eta(\b, \l; \a)$ $(\Re(\l)>0)$ which is just a constant multiple of $\l^{-1}$ obviously has a meromorphic continuation to $\CC$ with ${\rm CT}_{\l=0}\JJ_{\rm id}^\eta(\b, \l; \a) =0$. 
\end{proof}

Let us examine the terms $\JJ_{\rm u}^\eta(\b, \l; \a)$ and $ \JJ_{\rm \bar u}^\eta(\b, \l; \a)$. Assume that $q(\Re(\bfs))>\Re(\l)> \s >1$ and $1<\s<\ul/2$ and set
{\allowdisplaybreaks
\begin{align*}
U_{0, \eta}^{\pm}(\l; \bfs)&= \frac{1}{2\pi i}\int_{L_{\mp \s}}\frac{\beta(z)}{z+\lambda} \int_{\AA^{\times}}\Psi_{l}^{(0)}\left(\gn|\bfs; [\begin{smallmatrix}1&t^{-1}\\ 0&1\end{smallmatrix}][\begin{smallmatrix}1&x_{\eta}\\ 0&1\end{smallmatrix}]\right)\eta(tx_{\eta}^{*})|t|_\A^{\pm z}d^{\times}tdz, \\
U_{1, \eta}^{\pm}(\l; \bfs)&=\frac{1}{2\pi i}\int_{L_{\mp \s}}\frac{\beta(z)}{z+\lambda} \int_{\AA^{\times}}\Psi_{l}^{(0)}\left(\gn | \bfs; [\begin{smallmatrix}1&0\\ t^{-1}&1\end{smallmatrix}][\begin{smallmatrix}1&0\\ -x_{\eta}&1\end{smallmatrix}]w_{0}\right)\eta(tx_{\eta}^{*})|t|_\A^{\pm z}d^{\times}tdz
\end{align*}
} with $w_0=\left[\begin{smallmatrix} 0 & -1 \\ 1 & 0 \end{smallmatrix}\right]$, and
\begin{align*}
\Upsilon_{S}^{\eta}(z;\bfs)&=\prod_{v \in S}(1-\eta_{v}(\varpi_{v})q_{v}^{-(z+(s_{v}+1)/2)})^{-1}(1-q_{v}^{(s_{v}+1)/2})^{-1}, \\
\Upsilon_{S, l}^{\eta}(z; \bfs)= & D_{F}^{-1/2}\{\# (\go/\gf)^{\times}\}^{-1}
\{\prod_{v \in \Sigma_{\infty}}
\frac{2\Gamma(-z) \Gamma(l_{v}/2+z)}{\Gamma_{\RR}(-z + \e_{v})\Gamma(l_{v}/2)}i^{\e_{v}}\cos \left(\frac{\pi}{2}(-z+\e_{v})\right)\}\, 
\Upsilon_{S}^{\eta}(z;\bfs).
\end{align*}
Here $\e_{v} \in \{0, 1\}$ is the sign of $\eta_{v}$ for $v \in \Sigma_{\infty}$ (see \S \ref{add char and gauss sum}).

\begin{lem} \label{U^pm}
The double integrals $U_{0,\eta}^{\pm}(\l; \bfs)$ and $U_{1,\eta}^{\pm}(\l; \bfs)$ converge absolutely and 
\begin{align*}
U_{0, \eta}^{\pm}(\l; \bfs)&=\frac{1}{2\pi i}\int_{L_{\mp\s}}\frac{\b(z)}{z+\l}{\rm N}(\gf)^{\mp z}L(\mp z, \eta)(-1)^{\e(\eta)}\Upsilon_{S, l}^{\eta}(\pm z; \bfs)dz, \\
U_{1, \eta}^{\pm}(\l; \bfs)& = \frac{1}{2 \pi i}\int_{L_{\mp \s}}\frac{\b(z)}{z+\l}
{\rm N}(\gf)^{\mp z}{\rm N}(\gn)^{\pm z}\tilde{\eta}(\gn) \delta(\gn=\go)
L(\mp z, \eta)
i^{\tilde{l}}\Upsilon_{S, l}^{\eta}(\pm z;\bfs)dz,
\end{align*}
where $\tilde l=\sum_{v\in \Sigma_\infty} l_v$ and $\epsilon(\eta)=\sum_{v\in \Sigma_\infty}\e_v$. 
\end{lem}
\begin{proof}
 This is proved in the same way as \cite[Lemma 12.3]{Tsud}; to compute the archimedean integral, we use Lemma~\ref{lem:orbital integral of Shintani}. 
\end{proof}

By $1 <\s <\ul/2$, the possible poles of the integrand of $U_{0, \eta}^{+}(\l; \bfs)$ in the region
$-\s< \Re(z) <\s$ are $z=0, -1$.
In fact, we observe that the integrand is holomorphic at $z=-1$. We shift the contour $L_{-\s}$ to $L_{\s}$; by the residue theorem, 
\begin{align*}
U_{0, \eta}^{+}(\l;\bfs) = & \frac{1}{2\pi i}\int_{L_{\s}}\frac{\b(z)}{z+\l}
{\rm N}(\gf)^{-z}L(-z, \eta)(-1)^{\e(\eta)}\Upsilon_{S, l}^{\eta}(z; \bfs)dz
- \frac{\b(0)}{\l}
\delta_{\eta, \bf1}R_{F} \Upsilon_{S}^{\bf 1}(\bfs),
\end{align*}
where $R_{F}$ is the residue of $\zeta_{F}(s)$ at $s=1$.
In a similar manner,
{\small \begin{align*}
U_{1, \eta}^{+}(\l;\bfs) = & \frac{1}{2\pi i}\int_{L_{\s}}\frac{\b(z)}{z+\l}
{\rm N}(\gf)^{-z}\,L(-z, \eta)\,\delta(\gn=\cO)\,i^{\tilde{l}}\,\Upsilon_{S, l}^{\eta}(z; \bfs)dz
- \frac{\b(0)}{\l}
\delta_{\eta, \bf1}\,R_{F}\,\delta(\gn=\cO)\,i^{\tilde{l}}\,\Upsilon_{S}^{\bf 1}(\bfs).
\end{align*}}

Define $C_{0}(\eta)$ and $R(\eta)$ by
$$L(s, \eta)={R(\eta)}(s-1)^{-1} + C_{0}(\eta) + \Ocal(s-1), \ (s \rightarrow 1).$$
We remark that $R_{F}=R(\eta)$ if $\eta$ is trivial.

\begin{lem}\label{UNIPOTEMT-TERM1(crude)}
The function $\lambda \mapsto \JJ_{\rm u}^{\eta}(\b,\l;\a)$ on $\Re(\l)>1$ has a meromorphic continuation to the region $\Re(\l)>-\ul/2$. The constant term of $\JJ_{\rm u}^{\eta}(\b,\l;\a)$ at $\l=0$ equals $\JJ_{\rm u}^{\eta}(l, \gn|\a)\b(0)$.
Here we put
\begin{align*}
\JJ_{\rm u}^{\eta}(l, \gn |\a) =
(-1)^{\e(\eta)}\Gcal(\eta)D_{F}^{1/2}(1+(-1)^{\e(\eta)}\tilde{\eta}(\gn)i^{\tilde{l}}\delta(\gn=\go))
\left(\frac{1}{2\pi i}\right)^{\#S}\int_{\LL_{S}(\bfc)}
\Upsilon_{S}^{\eta}(\bfs)\mathfrak{C}_{S, {\rm u}}^{\eta}(\bfs)\a(\bfs)d \mu_{S}(\bfs)
\end{align*}
with
\begin{align*}
\mathfrak{C}_{S, {\rm u}}^{\eta}(\bfs)
=\pi^{\e(\eta)}C_{0}(\eta)+R(\eta)
\bigg\{ -\frac{d_{F}}{2}(C_{\rm Euler}+\log \pi )
+ \sum_{v\in\Sigma_{\infty}}\sum_{k=1}^{l_{v}/2-1}\frac{1}{k} + \sum_{v \in S}\frac{\log q_{v}}{1-q_{v}^{(s_{v}+1)/2}} + \log D_{F} \bigg\}.
\end{align*}
In particular, we have $\mathfrak{C}_{S, {\rm u}}^{\eta}(\bfs)=L_{\fin}(1, \eta)$
if $\eta$ is non-trivial.
\end{lem}

\begin{proof}
By definition,
$$
\JJ_{\rm u}^{\eta}(\b, \l;\a)
= \left(\frac{1}{2\pi i}\right)^{\# S} \int_{\LL_{S}(\bfc)}(U_{0, \eta}^{+}(\l;\bfs)+U_{0, \eta}^{-}(\l; \bfs)
+U_{1, \eta}^{+}(\l;\bfs)+ U_{1, \eta}^{-}(\l;\bfs))\a(\bfs)d\mu_{S}(\bfs).
$$
From Lemma~\ref{U^pm} and the computation after it, 
{\small \begin{align*}
\JJ_{\rm u}^{\eta}(\b, \l;\a)
= & \left(\frac{1}{2\pi i}\right)^{\# S} \int_{\LL_{S}(\bfc)}\frac{1}{2\pi i}
\int_{L_{\s}}\frac{\b(z)}{z+\l}
((-1)^{\e(\eta)}+i^{\tilde{l}}\delta(\gn=\go))
\{ {\rm N}(\gf)^{-z}L(-z, \eta)\Upsilon_{S, l}^{\eta}(z; \bfs) \\
& +{\rm N}(\gf)^{z}L(z, \eta)\Upsilon_{S, l}^{\eta}(-z; \bfs) \}dz
\a(\bfs)d\mu_{S}(\bfs )
-\frac{(-1)^{\e(\eta)}+i^{\tilde{l}}\delta(\gn=\go)}{2}\JJ_{\rm id}^{\eta}(\b, \l;\a),
\end{align*}
with $1<\s<\ul/2$ and $\Re(\l)>-\s$. Since $\sigma$ is arbitrary, this gives a meromorphic continuation of $\JJ_{\rm u}^{\eta}(\b, \l;\a)$ to $\Re(\l)>-\ul/2$.
By the above expression, 
\begin{align*}
& {\rm CT}_{\l=0}\JJ_{\rm u}^{\eta}(\b, \l;\a) \\
= & \left(\frac{1}{2\pi i}\right)^{\# S} \int_{\LL_{S}(\bfc)}
\frac{1}{2\pi i}
\int_{L_{\s}}\frac{\b(z)}{z}
((-1)^{\e(\eta)}+i^{\tilde{l}}\delta(\gn=\go))
\{ {\rm N}(\gf)^{-z}L(-z, \eta)\Upsilon_{S, l}^{\eta}(z; \bfs) \\
& +{\rm N}(\gf)^{z}L(z, \eta)\Upsilon_{S, l}^{\eta}(-z; \bfs) \}dz
\a(\bfs)d\mu_{S}(\bfs) \\
=& ((-1)^{\e(\eta)}+i^{\tilde{l}}\delta(\gn=\go))\frac{1}{2\pi i}\int_{L_{\s}}\frac{\b(z)}{z}\{f_{\rm u}(z)+f_{\rm u}(-z)\}dz\\
= & ((-1)^{\e(\eta)}+i^{\tilde{l}}\delta(\gn=\go)) \Res_{z=0}\left(\frac{\b(z)}{z}f_{\rm u}(z)\right) 
= ((-1)^{\e(\eta)}+i^{\tilde{l}}\delta(\gn=\go))\b(0){\rm CT}_{z=0}f_{\rm u}(z).
\end{align*}
}Here we put
$f_{\rm u}(z)={\rm N}(\gf)^{-z}L(-z, \eta)\Upsilon_{S, l}^{\eta}(z; \bfs)$.
By setting $\tilde{\Upsilon}_{S, l}^{\eta}(z;\bfs)=D_F^{1/2}\{\# (\cO/\ff)^\times \}{\Upsilon}_{S, l}^{\eta}(z;\bfs)$, the constant term is computed as follows:
{\small \allowdisplaybreaks
\begin{align*}
& {\rm CT}_{z=0}f_{\rm u}(z)
= \frac{d}{dz}{\rm N}(\gf)^{-z}zL(-z, \eta)\Upsilon_{S, l}^{\eta}(z; \bfs)\bigg|_{z=0} \\
= & \frac{d}{dz} \bigg\{ {\rm N}(\gf)^{-z} \times z i^{\e(\eta)}D_{F}^{1/2}{\rm N}(\gf)^{-1/2}\{\#(\go/\gf)^{\times}\}\Gcal(\eta)D_{F}^{1/2+z}{\rm N}(\gf)^{1/2+z} L(z+1, \eta) \\
& \times D_{F}^{-1/2}\{\#(\go/\gf)^{\times}\}^{-1}\tilde{\Upsilon}_{S, l}^{\eta}(z;\bfs) \bigg\} \bigg|_{z=0} \\
=& i^{\e(\eta)}\Gcal(\eta)D_{F}^{1/2}\times \frac{d}{dz} \bigg\{ D_{F}^{z}
zL(z+1, \eta) \times \tilde{\Upsilon}_{S,l}^{\eta}(z;\bfs) \bigg\} \bigg|_{z=0}\\
=& \Gcal(\eta) D_{F}^{1/2} \pi^{\e(\eta)} \tilde{\Upsilon}_{S}^{\eta}(\bfs)
\bigg\{ (\log D_{F})R(\eta) + C_{0}(\eta) + R(\eta) \frac{\frac{d}{dz}\tilde{\Upsilon}_{S,l}^{\bf 1}(z;\bfs)|_{z=0}}{\tilde{\Upsilon}_{S,l}^{\bf 1}(0;
\bfs)} \bigg\}.
\end{align*}
}Here $\e(\eta)=\sum_{v \in \Sigma_{\infty}}\e_{v}$.
We note that
$\tilde{\Upsilon}_{S,l}^{\eta}(0; \bfs) = (-i\pi)^{\e(\eta)} \Upsilon_{S}^{\eta}(0;\bfs)$
holds by
$$\frac{2\Gamma(-z)\Gamma(l_{v}/2+z)}{\Gamma_{\RR}(-z+\e_{v})\Gamma(l_{v}/2)}
i^{\e_{v}} \cos\left(\frac{\pi}{2}(-z+\e_{v})\right)\bigg|_{z=0}
=(-i\pi)^{\e_{v}}$$
for $v \in \Sigma_{\infty}$. The logarithmic derivative of $\tilde{\Upsilon}_{S,l}^{\bf 1}(z; \bfs)$ at $z=0$ is computed as
{\small \allowdisplaybreaks
\begin{align*}
& \sum_{v \in \Sigma_{\infty}}\frac{d}{dz}\log \bigg\{\frac{2\Gamma(-z)\Gamma(l_{v}/2+z)}{\Gamma_{\RR}(-z) \Gamma(l_{v}/2)}\cos\left(\frac{\pi z}{2}\right)\bigg\}\bigg|_{z=0}
 +\ \sum_{v \in S}\frac{d}{dz}\log (1-
q_{v}^{-(z+(s_{v}+1)/2)})^{-1}(1-q_{v}^{(s_{v}+1)/2})^{-1}\bigg|_{z=0} \\
=& \sum_{v \in \Sigma_{\infty}}
\bigg\{\psi(l_{v}/2)
-\frac{1}{2}\log \pi +\bigg(\frac{1}{2}\psi\left(\frac{-z}{2}\right)-\psi(-z) \bigg)\bigg|_{z=0}\bigg\}
+\sum_{v \in S}
\frac{\log q_{v}}{1-q_{v}^{(s_{v}+1)/2}}.
\end{align*}
}By the formulas
{\small \begin{align*}
\psi(l_{v}/2)=-C_{\rm Euler}+ \sum_{k=1}^{l_{v}/2-1}\frac{1}{k}, \qquad
\bigg(\frac{1}{2}\psi\left(\frac{-z}{2}\right)-\psi(-z) \bigg)\bigg|_{z=0}
=\frac{1}{2}C_{\rm Euler},
\end{align*}
}we are done.
\end{proof}

Assume that $q(\Re(\bfs))>\Re(\l)> \s$ and $1<\s<\ul/2$.
Analyzing the integrals 
{\allowdisplaybreaks
\begin{align*}
&\frac{1}{2\pi i}\int_{L_{\pm \s}}\int_{\AA^{\times}}\Psi_{l}^{(0)}\left(\gn|\bfs; [\begin{smallmatrix}1&0 \\ t &1\end{smallmatrix}][\begin{smallmatrix}1&x_{\eta}\\ 0&1\end{smallmatrix}]\right)\eta(tx_{\eta}^{*})|t|_\A^{\pm z}d^{\times}tdz, \\
&\frac{1}{2\pi i}\int_{L_{\pm \s}}\int_{\AA^{\times}}\Psi_{l}^{(0)}\left(\gn|\bfs; [\begin{smallmatrix}1&t \\ 0 &1\end{smallmatrix}][\begin{smallmatrix}1&0 \\ - x_{\eta}&1\end{smallmatrix}]w_{0}\right)\eta(tx_{\eta}^{*})|t|_\A^{\pm z}d^{\times}tdz
\end{align*}
}in the same way as $U^{\pm}_{\epsilon,\eta}(\lambda;{\bf s})$, we obtain the following lemma.

\begin{lem} \label{UNIPOTEMT-TERM2(crude)}
The function $\lambda \mapsto \JJ_{\bar{\rm u}}^{\eta}(\b,\l;\a)$ on $\Re(\l)>1$ has a meromorphic continuation to the region $\Re(\l)>-\ul/2$. The constant term of $\JJ_{\bar{\rm u}}^{\eta}(\b,\l;\a)$ at $\l=0$ equals $\JJ_{\bar{\rm u}}^{\eta}(l, \gn|\a)\b(0)$.
Here we put
\begin{align*}
\JJ_{\bar{\rm u}}^{\eta}(l, \gn |\a) =
(-1)^{\e(\eta)}\Gcal(\eta)D_{F}^{1/2}((-1)^{\e(\eta)}\tilde{\eta}(\gn)+i^{\tilde{l}}\delta(\gn=\go))
\left(\frac{1}{2\pi i}\right)^{\#S}\int_{\LL_{S}(\bfc)}
\Upsilon_{S}^{\eta}(\bfs)\mathfrak{C}_{S, \bar{\rm u}}^{\eta}(\bfs)\a(\bfs)d \mu_{S}(\bfs)
\end{align*}
with $\mathfrak{C}_{S, \bar{\rm u}}^{\eta}(\bfs)
= \mathfrak{C}_{S, {\rm u}}^{\eta}(\bfs)+ R(\eta)\log{\rm N}(\gn).$
\end{lem}

Let us consider the term $\JJ^\eta_{\rm{hyp}}(\beta,\lambda;\alpha)$, which is, by definition, equal to 
$$\int_{ \A^\times}\sum_{b\in F^\times-\{-1\}} \hat{\Psi}_l^{(0)}\left(\fn|{\a};\delta_b\,\left[\begin{smallmatrix} t & 0 \\ 0 &1 \end{smallmatrix}\right]\,\left[\begin{smallmatrix} 1 & x_\eta \\ 0 & 1 \end{smallmatrix}\right]\right)\,\{\hat\beta_\lambda(|t|_\A)+\hat\beta_{\lambda}(|t|_\A^{-1})\}\,\eta(tx_\eta^{*})\,\d^\times t.
$$
\begin{lem} \label{HYPERBOLIC-TERM(crude)}
The integral $\JJ^{\eta}_{\rm hyp}(\b, \l; \a)$ converges absolutely and has an analytic continuation to the region $\Re(\l)> -\e$ for some $\e>0$.
Moreover, we have
${\rm CT}_{\l=0}\JJ^{\eta}_{\rm hyp}(\b, \l; \a) = \JJ_{\rm hyp}^{\eta}(l, \gn|\a)\b(0)$.
Here $\JJ_{\rm hyp}^{\eta}(l, \gn|\a)$ is defined by 
\begin{align*}
\JJ_{\rm{hyp}}^{\eta}(l, \fn|\alpha)=\left(\frac{1}{2\pi i}\right)^{\# S}\int_{\LL_S(\bfc)} \fK_{S}^{\eta}(l,\gn|\bs)\,\alpha(\bs)\,\d \mu_S(\bs)
\end{align*}
with 
\begin{align*}
\fK_{S}^{\eta}(l,\gn|\bs)=\sum_{b\in  F^{\times}-\{-1\}} \int_{\A^\times} \Psi_{l}^{(0)}\left(\fn|\bs; \delta_b\,\left[\begin{smallmatrix} t & 0 \\ 0 &1 \end{smallmatrix}\right]\,\left[\begin{smallmatrix} 1 & x_\eta \\ 0 & 1 \end{smallmatrix}\right]\right)\,\eta(tx_\eta^{*})\,\d^\times t.
\end{align*}

\end{lem}
\begin{proof}
We take $c\in \RR$ such that $ \ul/2-1>(c+1)/4>1$. Then, from Lemma~\ref{JJ-hyperbolic}, there exists $\epsilon>0$ such that, for $0<|\rho|<\epsilon$ the integral
{\small \begin{align}
& \int_{\LL_{S}({\underline c})}|\a(\bfs)||d\mu_{S}(\bfs)|\int_{L_{\rho}}
\frac{|\b(z)|}{|z+\l|}
\sum_{b \in F^{\times}-\{-1\}}\int_{\AA^{\times}}\big|\Psi_{l}^{(0)}\left(\gn|\bfs; \delta_{b}[\begin{smallmatrix}t&0\\0&1\end{smallmatrix}]
[\begin{smallmatrix}1&x_{\eta}\\0&1\end{smallmatrix}]
\right)\big|\{|t|_{\AA}^{\rho}+|t|_\A^{-\rho}\}\,d^{\times}t|dz|, 
\label{++++}
\end{align}
}which is majorized by 
{\small \begin{align*}
&\int_{\LL_{S}({\underline c})}|\a(\bfs)||d\mu_{S}(\bfs)|\int_{L_{\rho}}\frac{|\b(z)|e^{d_{F}\pi|\Im(z)|/2}}
{|z+\l|}|dz|\sum_{b \in F^{\times}-\{-1\}}
\{ M_{\e}(\gn|0, \rho, l, \underline{c};b)+M_{\e}(\gn|0,-\rho, l, \underline{c};b)\},
\end{align*}
}is convergent. By $|t|_\A^{\rho}+|t|_\A^{-\rho}\geq 2$ $(t\in \A^\times)$, the integral \eqref{++++} is finite even for $\rho=0$.
Hence, we obtain an analytic continuation of the function
{\small
\begin{align*}
& \JJ_{\rm hyp}^{\eta}(\b, \l; \a) \\ = & \left(\tfrac{1}{2\pi i}\right)^{\#S}{\textstyle{\int_{\LL_{S}(\underline{c})}}} \bigl\{ \tfrac{1}{2\pi i}{\textstyle{\int_{L_{\rho}}}}\tfrac{\b(z)}{z+\l}
\bigl(\sum_{b \in F^{\times}-\{-1\}}{\textstyle{\int_{\AA^{\times}}}}\Psi_{l}^{(0)}\left(\gn |\bfs; \delta_{b}[\begin{smallmatrix}t&0\\0&1\end{smallmatrix}][\begin{smallmatrix}1&x_{\eta}\\0&1\end{smallmatrix}]\right) 
 (|t|_{\AA}^{z}+|t|_{\AA}^{-z})\,\eta(tx_\eta^*)\,d^{\times}t \bigl)\, dz \bigl\} \a(\bfs) d\mu_{S}(\bfs)
\end{align*}}
in the variable $\l$ to the region $\Re(\l)>-\e$.
\end{proof}

\section{The relative trace formula}
Let $\fn$ be an integral ideal of $F$, $l=(l_v)_{v\in \Sigma_\infty}$ an even weight with $l_v\geq 6$ for all $v\in \Sigma_\infty$, and $\eta$ a real valued idele class character of $F^\times$ unramified at all $v\in S(\fn)$. Let $\ff$ denote the conductor of $\eta$. We assume $(-1)^{\e(\eta)}\tilde{\eta}(\gn)=1$
(for the definition of $\e(\eta)$, see \S \ref{add char and gauss sum}). Put $\tilde l=\sum_{v\in \Sigma_\infty} l_v$. Let $S$ be a finite subset of $\Sigma_{\fin}$ disjoint from $S(\fn)\cup S(\ff)$. For $v\in S$, let $\ccA_v$ be the space of all holomorphic functions $\alpha_v(s_v)$ in $s_v\in \C$ satisfying $\alpha_v(s_v)=\alpha_v(-s_v)$ and $\alpha_v(s_v+\frac{4\pi i }{\log q_v})=\alpha_v(s_v)$. Set $\ccA_S=\bigotimes_{v\in S} \ccA_v$.

\begin{thm} \label{RELATIVETRACEFORMULA}
For any function $\alpha\in \ccA_S$, we have the identity 
\begin{align}
C(l,\fn,S) \sum_{\pi \in \Pi_{\rm{cus}}(l, \fn)}\II_{\rm cus}^\eta(\pi; l, \fn) \,\alpha(\nu_S(\pi))=
\tilde{\JJ}_{\rm u}^{\eta}(l, \gn|\a) + \JJ_{\rm hyp}^{\eta}(l, \gn|\a)
 \label{RELATIVE-TF}
\end{align}
Here $\nu_S(\pi)=\{\nu_v(\pi)\}_{v\in S}$ is the spectral parameter of $\pi$ at $S$ (see \S 6.5.2), 
\begin{align*}
C(l,\fn,S)&=(-1)^{\# S}\{ \prod_{v\in \Sigma_\infty}\frac{2 \pi \, \Gamma(l_v-1)}{\Gamma(l_v/2)^{2}}\} \frac{D_F^{-1}}{2} [\bK_\fin:\bK_0(\fn)]^{-1}
,\\
\II_{\rm cus}^\eta(\pi; l, \fn)&= (-1)^{\epsilon(\eta)}\cG(\eta)\,
\,w_\fn^\eta(\pi)\,\frac{L(1/2,\pi)L(1/2,\pi\otimes \eta)}{\nr(\ff_\pi)\,[\bK_\fin:\bK_0(\ff_\pi)]^{-1}\,L^{S_\pi}(1,\pi;{\rm{Ad}})}
\end{align*}
with $w_\fn^{\eta}(\pi)$ given in Lemma~\ref{value of PP}, and
\begin{align*}
\tilde{\JJ}_{\rm u}^{\eta}(l, \gn|\a)&=
2(-1)^{\e(\eta)}\Gcal(\eta)D_{F}^{1/2}(1+i^{\tilde{l}}\delta(\gn=\go))
\left(\frac{1}{2\pi i}\right)^{\#S}\int_{\LL_{S}(\bfc)}
\gU_{S}^{\eta}(l, \gn|\bfs)\a(\bfs)d\mu_{S}(\bfs), \\
\JJ_{\rm hyp}^\eta(l, \gn|\a)&=
\left(\frac{1}{2\pi i}\right)^{\# S}\int_{L_S(\bfc)} \fK_{S}^{\eta}(l,\gn|\bs)\,\alpha(\bs)\,\d \mu_S(\bs)
\end{align*}
with
\begin{align*}
\gU_{S}^{\eta}(l, \gn|\bfs) = & \prod_{v \in S}(1-\eta_{v}(\varpi_{v})q_{v}^{-(s_{v}+1)/2})^{-1}(1-q_{v}^{(s_{v}+1)/2})^{-1} \bigg\{{\bf C}_{F}^{\eta}(l,\gn)+R(\eta)\sum_{v \in S}\frac{\log q_{v}}{1-q_{v}^{(s_{v}+1)/2}}\bigg\}, \\
\fK_{S}^{\eta}(l,\gn|\bs)= & \sum_{b\in  F^{\times}-\{-1\}} \int_{\A^\times} \Psi_{l}^{(0)}\left(\fn|\bs;\delta_b\,\left[\begin{smallmatrix} t & 0 \\ 0 &1 \end{smallmatrix}\right]\,\left[\begin{smallmatrix} 1 & x_\eta \\ 0 & 1 \end{smallmatrix}\right]\right)\,\eta(tx_\eta^{*})\,\d^\times t
\end{align*}
and
\begin{align*}
{\bf C}_{F}^{\eta}(l, \gn)=\pi^{\e(\eta)}C_{0}(\eta) + R(\eta)
\bigg\{ -\frac{d_{F}}{2}(C_{\rm Euler}+\log \pi )
+ \log (D_{F}{\rm N}(\gn)^{1/2}) +\sum_{v\in\Sigma_{\infty}}\sum_{k=1}^{l_{v}/2-1}\frac{1}{k} \bigg\}.
\end{align*}
We remark ${\bf C}_{F}^{\eta}(l, \gn) = L_{\fin}(1, \eta)$ if $\eta$ is non-trivial.\end{thm}
\begin{proof}
From Lemmas~\ref{IDEMPOTENT TERM}, \ref{UNIPOTEMT-TERM1(crude)}, \ref{UNIPOTEMT-TERM2(crude)} and \ref{HYPERBOLIC-TERM(crude)}, $P_{\rm reg}^{\eta}(\hat{\bf\Psi}_{\rm reg}^{l}(\fn|\alpha))$ is given by the right-hand side of \eqref{RELATIVE-TF}; the left-hand side is provided by Proposition~\ref{Reg-per}. 
\end{proof}

We restrict our attention to the test functions of the form $\alpha(\bs)=\prod_{v\in S} \alpha_v^{(m_v)}(s_v)$ with 
\begin{align}
\alpha_v^{(m)}(s_v)=q_v^{ms_v/2}+q_v^{-ms_v/2}, \qquad v\in S ,\, m\in \N_{0}.
\label{HeckeFunctions}
\end{align}
As is well known, these functions form a $\C$-basis of the image of the spherical Hecke algebra $\cH(G_v,\bK_v)$ by the spherical Fourier transform. Thus, by restricting our consideration to these functions, no generality is lost practically. The following two theorems are proved in \S 10 and \S 11.
 
\begin{thm} \label{HYPERBOLIC(fine)}
For $\alpha=\otimes_{v\in S}\, \alpha_v$, we have 
\begin{align*}
\JJ_{\rm{hyp}}^{\eta}(l, \fn|\alpha)&=\sum_{b\in F^{\times}-\{-1\}} \{\prod_{v\in S} J_v^{\eta_v}(b;\alpha_v)\}\{\prod_{v\in \Sigma_\infty} J_v^{\eta_v}(l_v;b)\}\,\{\prod_{v\in \Sigma_\fin-S} J_v^{\eta_v}(b)\}.
\end{align*} 
Here $J_v^{\eta_v}(b;\alpha_v)$ is given by Lemma~\ref{MS2}, $J_v^{\eta_v}(b)$ is given by Lemmas \ref{MS3}, \ref{MS4} and \ref{ramifiedJ-exact}, and $J_v^{\eta_v}(l_v;b)$ is given by Lemma~\ref{arcInt}. 
\end{thm} 

\begin{thm}\label{UNIPOTENT(fine)}
For $\alpha=\otimes_{v\in S}\,\alpha_v$, we have 
\begin{align*}
\tilde \JJ_{\rm{u}}^\eta(l, \fn|\alpha)&=2(-1)^{\epsilon(\eta)}\cG(\eta)\,D_F^{1/2}(1+i^{\tilde l}\delta(\fn=\cO))\,
\\
&\quad\times \left\{\bC_F^\eta(l, \fn)\,\prod_{v\in S} U_v^{\eta_v}(\alpha_v)+R(\eta)\sum_{v\in S} U'_v(\alpha_v)\prod_{w\in S-\{v\}}U_w^{\eta_w}(\alpha_w)
\right\}. 
\end{align*}
Here $U_v^{\eta_v}(\alpha_v)$ and $U_v'(\alpha_v)$ are explicitly given in Proposition~\ref{MS5}. 
\end{thm}

\hspace{-4mm}{\bf The proof of Theorem~\ref{ED-THM}.}
By the same procedure done in \cite[\S7.1]{Sugiyama2}, the estimation is reduced to that for a similar average over $\Pi_{\rm{cus}}(l,\gn)$ (in place of $\Pi_{\rm{cus}}^*(l,\fn)$).
From \cite[Lemma 13.15]{Tsud}, we have 
{\small
$$\tilde{\JJ}_{\rm u}^\eta(l, \gn|\a)=2(-1)^{\e(\eta)+\# S}(1+i^{\tilde l}\delta(\fn=\cO))
D_F^{1/2}\Gcal(\eta) L_\fin(1,\eta) \int_{\fX_S^0} \a(\bs)\,\d\mu_S^\eta(\bs).$$
}By Theorem~\ref{RELATIVETRACEFORMULA}, it suffices to show $\JJ_{\rm hyp}^\eta(l, \gn|\a)={\mathcal O}_{\e, l, \eta, \a}(\nr(\gn)^{-\inf_{v \in \Sigma_{\infty}}l_{v}/2 +1+\e})$ for any sufficiently small $\e>0$. This follows from the proof of Lemma~\ref{HYPERBOLIC-TERM(crude)} and Lemmas~\ref{hyper(rev)4} and \ref{hyper(rev)6}
by taking $c \in \RR$ and $\rho\neq 0$ such that
$\inf_{v \in \Sigma_{\infty}}l_{v}/2-1
>(c+1)/4>(c+1)/4 -|\rho|>\inf_{v \in \Sigma_{\infty}}l_{v}/2-1-\e>1$
and $|\rho|$ is sufficiently small. When $\gn$ is square-free, by noting $w_\fn^\eta(\pi)=\delta(\gn = \gf_{\pi})$, we need no procedure as in \cite[\S7.1]{Sugiyama2}. Thus the exponent of the error term is not spoiled and remain $-\inf_{v \in \Sigma_{\infty}}l_{v}/2+1+\e$ in the final result.
\qed

\begin{cor}\label{stability}
 Suppose $\eta_v(-1)=-1$ for all $v\in \Sigma_\infty$. Let $\fa$ be an integral ideal. Then, $\JJ_{\rm{hyp}}^{\eta}(l,\fn|\alpha)=0$ for any ideal $\fn$ such that $\nr(\fn)> \nr(\ff\,\fa)$ and for any $\alpha\in \ccA_S$ of the form $\otimes_{v\in S} \alpha_v$ with $\alpha_v(s_v)$ being a linear combination of $\a_v^{(m)}(s_v)$ $(0\leq m\leq {\rm{ord}}_v\fa)$.
\end{cor}
\begin{proof} 
 From Theorem~\ref{HYPERBOLIC(fine)}, the condition on $\a$ implies that the hyperbolic term is a sum of certain terms indexed by $b\in F^\times-\{-1\}$ such that $b\in \fn\ff^{-1}\fa^{-1}$ (Lemmas~\ref{MS2.5}, \ref{MS3}, \ref{MS4} and \ref{ramifiedJ-exact}) and $0<|\nr(b)|<1$ (Lemma~\ref{arcInt}). Thus, the summation becomes empty if $\nr(\fn)>\nr(\ff\,\fa)$. 
\end{proof}

\medskip
\noindent
{\bf Remark} :
\begin{itemize}
\item[(1)] In our forthcoming paper \cite{SugiyamaTsuzuki}, we show that the error term in Theorem~\ref{ED-THM} is improved to $\Ocal(\nr(\gn)^{-1+\e})$. 
\item[(2)] The vanishing of the term $\JJ_{\rm{hyp}}^{\eta}(l,\fn|\alpha)$ for $(\fn,\a)$ with both $\nr(\fn)$ and ${\rm{deg}}(\a)$ large is called the {\it stability} and was already observed in \cite{MichelRamakrishinan}, \cite{FeigonWhitehouse} and \cite{Nelson} at least when $\fn$ is square-free. Actually, even when $\eta$ admits a place $v\in \Sigma_\infty$ such that $\eta_v(-1)=+1$, our relative trace formula (Theorems~\ref{RELATIVETRACEFORMULA}, \ref{HYPERBOLIC(fine)}, and \ref{UNIPOTENT(fine)} combined) gives an exact formula for the spectral average although the expression involves an infinite sum. 
\end{itemize}

\section{Explicit formula of the hyperbolic term}
In this section, we compute $\JJ_{\rm{hyp}}^\eta(l,\fn|\alpha)$ further for particular test functions $\a=\otimes_{v \in S} \a_{v}$. 
By changing the order of integrals, we have
\begin{align*}
\JJ_{\rm{hyp}}^{\eta}(l, \fn|\alpha)&=\sum_{b\in F^{\times}-\{-1 \}} \{\prod_{v\in S} J_v^{\eta_v}(b;\alpha_v)\}\,
\{\prod_{v\in \Sigma_\infty} J_v^{\eta_v}(l_v;b)\}\,\{\prod_{v\in \Sigma_\fin-S} J_v^{\eta_v}(b)\},
\end{align*}
where 
\begin{align*}
J_v^{\eta_v}(b;\alpha_v)&=\frac{1}{2\pi i } \int_{L_v(c)} \left\{ \int_{F_v^\times} \Psi_v^{(0)}\left(s_v;\,\delta_b\,\left[\begin{smallmatrix} t & 0 \\ 0 &1 \end{smallmatrix}\right]\right)\,\eta_v(t)\,\d^\times t\right\} \alpha_v(s_v)\,\d\mu_v(s_v)
\end{align*}
for $v\in S$ with $\Psi_v^{(0)}(s_v;-)$ being the Green function on $G_v$ (Lemma~\ref{v-Greenftn}), 
\begin{align*}
J_v^{\eta_v}(b)&=\int_{F_v^\times} \Phi_{v,\fn}^{(0)}\left(\delta_b\,\left[\begin{smallmatrix} t & 0 \\ 0 &1 \end{smallmatrix}\right]\,\left[\begin{smallmatrix} 1 & \varpi_v^{-f(\eta_v)} \\ 0 &1 \end{smallmatrix}\right] \right)\,\eta_{v}(t\varpi_v^{-f(\eta_v)})\,\d^\times t, \quad {\text{if $v\in \Sigma_{\fin}-S$}}, \\
J_v^{\eta_v}(l_v;b)&=\int_{\R^\times} \Psi_v^{(0)}\left(l_v;\delta_b\,\left[\begin{smallmatrix} t & 0 \\ 0 &1 \end{smallmatrix}\right]\right)\,\eta_v(t)\,\d^\times t, \quad {\text{if $v\in \Sigma_\infty$}}
\end{align*}
with $\Psi_v^{(0)}(l_v;-)$ being the Shintani function (Proposition~\ref{Sh-ftn}), and $L_v(c)$ denotes the vertical contour directed from $c-\frac{2\pi i}{\log q_v}$ to $c+\frac{2\pi i}{\log q_v}$.

\subsection{An evaluation of non-archimedean integrals (for unramified $\eta_v$)}
In this paragraph, we explicitly compute the integrals $J_v^{\eta_v}(b;\alpha_v^{(m)})$ at $v\in S$ and the integrals $J_{v}^{\eta_v}(b)$ at $v\in \Sigma_\fin-S\cup S(\ff)$. 

\begin{lem}
\label{MS1} 
Let $v\in S$. Let $\alpha_v^{(m)}(s_v)=q_{v}^{ms_{v}/2}+q_{v}^{-ms_{v}/2}$ with $m\in \N_{0}$. Set
\begin{align*}
\widehat{\Phi_v}_{m}(g_v)=\frac{1}{2\pi i } \int_{L_v(c)} \Psi_v^{(0)}(s_v;\,g_v)
\alpha_v^{(m)}(s_v)\,\d \mu_v(s_v).
\end{align*}
If $m>0$, then, for any $x\in F_v$ with $\sup(|x|_v,1)=q_v^{l}$ with $l\in \N_{0}$, we have 
\begin{align*}
\widehat{\Phi_{v}}_{m}\left(\left[\begin{smallmatrix} 1 & x \\ 0 &1 \end{smallmatrix}\right]\right)&=
\begin{cases}
 0 \qquad & (l\geq m+1), \\
 -q_v^{-m/2} \qquad &( l=m), \\
 (m-l-1)q_v^{1-m/2}-(m-l+1)q_v^{-m/2} \qquad &(0\leq l <m).
\end{cases}
\end{align*}
If $m=0$, then for any $x\in F_v$ with $\sup(|x|_v,1)=q_v^{l}$ with $l\in \N_{0}$, we have 
\begin{align*}
\widehat{\Phi_v}_{0}\left(\left[\begin{smallmatrix} 1 & x \\ 0 &1 \end{smallmatrix}\right]\right)&=-2\,\delta(l=0).
\end{align*}
\end{lem}
\begin{proof}
From Lemma~\ref{v-Greenftn} and the formula $\d\mu_v(s)=2^{-1}\log q_v\,(q_v^{(1+s)/2}-q_v^{(1-s)/2})\,\d s$, we have  
\begin{align*}
\widehat{\Phi_v}_{m}(\left[\begin{smallmatrix} 1 & x \\ 0 &1 \end{smallmatrix}\right])&=\frac{1}{2\pi i } \int_{L_v(c)} q_v^{-(s+1)l/2}\,(1-q_v^{-(s+1)/2})^{-1}(1-q_v^{(s+1)/2})^{-1}\\
&\qquad \times (q_v^{-ms/2}+q_v^{ms/2})\,2^{-1}\log q_v\,(q_v^{(1+s)/2}-q_v^{(1-s)/2})\,\d s.
\end{align*}
By the variable change $z=q_v^{s/2}$, this becomes
\begin{align*}
\frac{q_v^{(1-l)/2}}{2\pi i} \oint_{|z|=R} z^{-l}(1-q_v^{-1/2}z^{-1})^{-1}(1-q_v^{1/2}z)^{-1}(z^m+z^{-m})\,(z-z^{-1})\,\frac{\d z}{z}
\end{align*}
with $R=q_v^{c/2}\,(>1)$. Thus, by the residue theorem, we have the equality
\begin{align}
\widehat{\Phi_v}_{m}(\left[\begin{smallmatrix} 1 & x \\ 0 &1 \end{smallmatrix}\right])&
=q_v^{(1-l)/2}\,\left({\rm{Res}}_{z=q_{v}^{-1/2}} \phi(z)+{\rm{Res}}_{z=0}\phi(z)\right)
 \label{MS1-1}
\end{align}
with $\phi(z)=\frac{(1-z^2)(z^m+z^{-m})}{(1-q_v^{1/2}z)^2}\frac{q_v^{1/2}}{z^{1+l}}$. By a direct computation, we have
{\small \begin{align}
{\rm{Res}}_{z=q_{v}^{-1/2}} \phi(z)=&
-q_v^{(1+l)/2}\,\biggl(\{(m+l+1)(1-q_v^{-1})+2q_v^{-1}\}q_v^{m/2} 
+\{(-m+l+1)(1-q_v^{-1})+2q_v^{-1}\}q_v^{-m/2}\biggr), 
\label{MS1-2}
\\
{\rm{Res}}_{z=0}\phi(z)=\, &
\delta(l\geq m+1)\,\{(l-m+1)q_v^{(l-m+1)/2}-(l-m-1)q_v^{(l-m-1)/2}\} 
 \label{MS1-3}
\\
&+\delta(l\geq 1-m)\,\{(l+m+1)q_v^{(l+m+1)/2}-(l+m-1)q_v^{(l+m-1)/2}\} 
 \notag
\\
&+\{\delta(m=l)+\delta(m=-l)\}\,q_v^{1/2}.
 \notag
\end{align}
}From \eqref{MS1-1}, \eqref{MS1-2} and \eqref{MS1-3}, we obtain the desired formula easily.
\end{proof}

\begin{lem} \label{MS2}
Let $v\in S$. Let $\alpha_v^{(m)}(s_v)=q_{v}^{ms_{v}/2}+q_{v}^{-ms_{v}/2}$ with $m\in \N_{0}$. Then, for any $b\in F_v^{\times}-\{-1\}$, 
$$
J_v^{\eta_v}(b;\alpha_v^{(m)})=I^{+}_v(m;\,b)+ \eta_{v}(\varpi_{v})I^{+}_v(m;\,\varpi_v^{-1}(b+1))
$$
with 
{\small \begin{align*}
I_v^{+}(m;\,b)&=\vol(\cO_v^\times)\,2^{\delta(m=0)}\,\biggl(
-q_v^{-m/2}\,\delta_m^{\eta_v}(b)
+\sum_{l=\sup(0,-{\rm{ord}}_v(b))}^{m-1}\{(m-l-1)q_v^{1-m/2}-(m-l+1)q_v^{-m/2}\}\delta^{\eta_v}_l(b)\biggr),
\end{align*}
}where for $n\in \N_0$, 
$$
\delta^{\eta_v}_n(b)=\delta(|b|_v\leq q_v^{n})\, \eta_{v}(\varpi_{v}^{n})
\begin{cases} ({\rm{ord}}_v(b)+1)^{\delta(n=0)}, \quad (\eta_v(\varpi_v)=1), \\  (2^{-1}(\eta_v(b)+1))^{\delta(n=0)}, \quad (\eta_v(\varpi_v)=-1).
\end{cases}
$$
\end{lem}
\begin{proof} Let $m>0$. By definition, $J_v^{\eta_v}(b;\alpha_v)=I_v^{+}(m;b)+I_{v}^{-}(m;b)$ with $I_{v}^{+}(m;b)$ and $I_v^{-}(m;b)$ being the integrals of $\widehat{\Phi_v}_{m}\left(\delta_b \left[\begin{smallmatrix} t & 0 \\ 0 &1 \end{smallmatrix}\right]\right)\eta_v(t)$ with respect to the measure $\d ^\times t$ over $|t|_v\leq 1$ and over $|t|_v>1$, respectively. From \cite[Lemma 11.4]{Tsud}, 
\begin{align}
\delta_b \left[\begin{smallmatrix} t & 0 \\ 0 &1 \end{smallmatrix}\right]
&\in H_v \left[\begin{smallmatrix} 1 & x \\ 0 &1 \end{smallmatrix}\right]\,\bK_v\quad \text{with} \quad  
|x|_v=\begin{cases} |t|_v^{-1}|b|_v\qquad& (|t|_v\leq 1), \\
 |t|_v|b+1|_v \qquad& (|t|_v>1).
\end{cases}
 \label{MS2-0}
\end{align}
Hence, by Lemma~\ref{MS1}, $I_v^{+}(m;b)$ becomes the sum of the integral 
\begin{align}
\int_{\substack{|t|_v\leq 1 \\ \sup(1,|t|_v^{-1}|b|_v)=q_v^{m}}}(-q_v^{-m/2})\,\eta_v(t)\,\d^\times t \label{MS2-1}
\end{align}
and 
\begin{align}
\sum_{l=0}^{m-1}\int_{\substack{|t|_v\leq 1 \\ \sup(1,|t|_v^{-1}|b|_v)=q_v^{l}}} \{(m-l-1)q_v^{1-m/2}-(m-l+1)q_v^{-m/2}\}\,\eta_v(t)\,\d^\times t. 
 \label{MS2-2}
\end{align}
The condition $|t|_v\leq 1$, $\sup(1,|t|_v^{-1}|b|_v)=q_v^{l}$ is equivalent to $|b|_v\leq q_v^{l}$, $|t|_v=q_v^{-l}|b|_v$ if $l>0$ and to $|b|\leq |t|\leq 1$ if $l=0$.  Hence, \eqref{MS2-1} is equal to $-q_v^{-m/2}\,\vol(\cO_v^\times)\,\delta_{m}^{\eta_v}(b)$, and \eqref{MS2-2} is equal to the following expression
\begin{align*}
\vol(\cO_v^\times)\,\sum_{l=\sup(0,-{\rm{ord}}_v(b))}^{m-1} \{(m-l-1)q_v^{1-m/2}-(m-l+1)q_v^{-m/2}\}\,\delta_l^{\eta_v}(b).
\end{align*}
This completes the evaluation of the integral $I_v^{+}(m;b)$. In the same way as above, the other integral $I_v^{-}(m;b)$ is calculated in a similar form; from the resulting expression, $I_v^{-}(m;b)= \eta_{v}(\varpi_{v})I_v^{+}(m;\varpi_v^{-1}(b+1))$ is observed. This settles our consideration when $m>0$. The other case $m=0$ is similar. 
\end{proof}

From Lemma~\ref{MS2}, we have a useful estimate for the function $J_v(b,\alpha_v)$ in $b$. 

\begin{lem}\label{MS2.5}
Let $\alpha_v^{(m)}(s_{v})=q_{v}^{ms_{v}/2}+q_{v}^{-ms_{v}/2}$ with $m\in \N_{0}$. If $m>0$, then
\begin{align*}
|J_v^{\eta_v}(b,\alpha_v^{(m)})|\ll (m+1)^{2}\{ \delta(|b|_v\leq q_v^{m-1})\,q_v^{1-m/2}+\delta(|b|_v=q^{m}_v)\,q_v^{-m/2}\}, \qquad b\in F_v^{\times}-\{-1\}
\end{align*}
with the implied constant independent of $v$ and $m$. If $m=0$, then,
$$
J_v^{\eta_v}(b,\alpha_v^{(0)})=-2\vol(\cO_v^\times)\,\Lambda_v^{\eta_v}(b),
$$
where $\Lambda_v^{\eta_v}$ is the function on $F_v^{\times}-\{-1\}$ defined by 
\begin{align}
\Lambda_v^{\eta_v}(b)= \delta(b \in \go_{v})\,\delta_{0}^{\eta_{v}}(b(b+1)).
 \label{MS2.5-0}
\end{align}
\end{lem}
\begin{proof}
To infer the estimate from Lemma~\ref{MS2} in the case when $m>0$, it suffices to note that $I^{+}_m(\varpi_v^{-1}(b+1))=0$ if $|b|_v \ge q_v^{m}$, or equivalently if $|\varpi_v^{-1}(b+1)|_v \ge q_{v}^{m+1}$. The formula of $J_{v}^{\eta_{v}}(b,\alpha_v^{(0)})$ is obtained by noting the relation $\Lambda_v^{\eta_v}(b)=\delta_0^{\eta_v}(b)+ \eta_{v}(\varpi_{v})\, \delta_0^{\eta_v}(\varpi_v^{-1}(b+1))$. 
\end{proof}

\begin{lem} \label{MS3}
Let $v\in \Sigma_\fin-S\cup S(\fn\ff)$. Then 
\begin{align*}
J_v^{\eta_v}(b)&=\vol(\cO_v^\times)\,\Lambda_v^{\eta_v}(b)
\end{align*}
with $\Lambda_v^{\eta_v}(b)$ being defined by \eqref{MS2.5-0}. 
\end{lem}
\begin{proof}
This is proved in the same way as the case $m=0$ in Lemma \ref{MS2.5}.
\end{proof}

\begin{lem} \label{MS4}
Let $v\in S(\fn)-S(\ff)$. If $\eta_v(\varpi_v)=1$, then 
\begin{align*}
J_v^{\eta_v}(b)&=\vol(\cO_v^\times)\,\delta(b\in \fn\cO_v)\,
\{{\rm{ord}}_v(b)-{\rm{ord}}_v(\fn)+1\}.
\end{align*}
If $\eta_v(\varpi_v)=-1$, then
\begin{align*}
J_v^{\eta_v}(b)&=\vol(\cO_v^\times)\,\delta(b\in \fn\cO_v)\,2^{-1}(\eta_v(b)+(-1)^{{\rm{ord}}_v(\fn)}).
\end{align*}
\end{lem}
\begin{proof}
This is proved in the same way as Lemma~\ref{MS3}. We only have to remark that the assertion in the last sentence of \cite[Lemma 11.4]{Tsud} is relevant here.   
\end{proof}

\subsection{An evaluation of non-archimedean integrals (for ramified $\eta_v$)}
We shall calculate the integral $J_v^{\eta_v}(b)$ at finitely many places $v\in S(\ff)$. In what follows in this paragraph, we fix $v\in S(\ff)$ and set $f=f(\eta_v)$; thus $f$ is a positive integer. For $l\in \Z$, consider the following subsets of $F_v^\times$ depending on $b\in F_v^{\times}-\{-1\}$. 
\begin{align*}
D_l(b)&=\{t\in F_v^\times|\,|t|_v=q_v^{-l}, \,|1+t\varpi_v^{-f}|_v\,|b+t\varpi_v^{-f}(b+1)|_v\leq q_v^{-l}\,\}, \quad (l\in\Z-\{f\}),\\
D_f(b)&=\{t\in F_v^\times|\,-t\in \varpi_v^{f}(\cO_v^\times-U_v(f)), \, |1+t\varpi_v^{-f}|_v\,|b+t\varpi_v^{-f}(b+1)|_v\leq q_v^{-f}\,\}, 
\end{align*}
where $U_v(m)=1+\fp_v^{m}$ for any positive integer $m$.

\begin{lem}\label{ramifiedJ-L1}
Let $l>f$. Then, $D_l(b)=\emp$ unless $l=f-{\rm{ord}}_v(b+1)+{\rm{ord}}_v(b)$, in which case, we have ${\rm{ord}}_v(b)>0$, ${\rm{ord}}_v(b+1)=0$ and 
\begin{align*}
\int_{t\in D_l(b)}\eta_v(t)\,\d^\times t=
\eta_v\left(-\varpi_v^{f}\,b(b+1)^{-1}\right)\,(1-q_v^{-1})^{-1}\,q_{v}^{-f-d_{v}/2}.
\end{align*}
\end{lem}
\begin{proof}
By the variable change $t=\varpi_v^{l}t'$, we have
$$\int_{t\in D_l(b)} \eta_v(t)\,\d^\times t=\eta_v(\varpi_v^{l})\,\int_{t'\in D'} \eta_{v}(t')\,\d^\times t'$$
with $D'=\{t'\in \cO_v^\times|\,|1+t'\varpi_v^{l-f}|_v\,|b+t'\varpi_v^{l-f}(b+1)|_v\leq q_v^{-l}\}$.
Let $t'\in \cO_v^\times$. Then, the condition 
$$|1+t'\varpi_v^{l-f}|_v\,|b+t'\varpi_v^{l-f}(b+1)|_v\leq q_v^{-l}$$
is equivalent to 
\begin{align}
t'\in \varpi_v^{f-l}\tfrac{-b}{b+1}(1+\varpi_v^{l}b^{-1}\cO_v). 
 \label{ramifiedJ-L1-1}
\end{align}
If $|\varpi_v^{l}b^{-1}|_v>1$, then $1+\varpi_v^{l}b^{-1}\cO_v=\varpi_v^{l}b^{-1}\cO_v$. Hence, from \eqref{ramifiedJ-L1-1}, 
$$
1=|t'|_v\leq \left|\varpi_v^{f-l}\tfrac{-b}{b+1}\cdot \varpi_v^{l}b^{-1}\right|_v=\left|\tfrac{\varpi_v^{f}}{b+1}\right|_v,
$$
and $b+1\in \fp_v^{f}$ follows. Since $f>0$, we obtain $|b|_v=1$, which, combined with $|\varpi_v^{l}b^{-1}|_v>1$, implies the inequality $|\varpi_v^{l}|_v>1$ contradicting to $l>f>0$. 

If $|\varpi_v^{l}b^{-1}|_v=1$, then $b\in \varpi_v^{l}\cO_v^\times$; thus, $|b+1|_v=1$ by $l>f>0$. Hence, from \eqref{ramifiedJ-L1-1}, we have the inequality 
\begin{align*}
1=|t'|_v\leq \left|\varpi_v^{f-l}\tfrac{-b}{b+1}\right|_v=|\varpi_v^{f}|_v=q_v^{-f},
\end{align*}
which is impossible due to $f>0$. From the considerations so far, we have the inequality $|\varpi_v^{l}b^{-1}|_v<1$, which yields $1+\varpi_v^{l}b^{-1}\cO_v\subset \cO_v^\times$. Hence, from \eqref{ramifiedJ-L1-1}, we have the second equality of  
$$
1=|t'|_v=\left|\varpi_v^{f-l}\tfrac{-b}{b+1}\right|_v,
$$
which implies $l=f-{\rm{ord}}_v(b+1)+{\rm{ord}}_v(b)$. From this and $l>f$, we have ${\rm{ord}}_v(b+1)<{\rm{ord}}_v(b)$, which holds if and only if ${\rm{ord}}_v(b)>0$ and ${\rm{ord}}_v(b+1)=0$.

If we set $t'=\varpi_v^{f-l}\frac{-b}{b+1}r$, then \eqref{ramifiedJ-L1-1} becomes $r\in 1+\varpi_v^{l}b^{-1}\cO_v=1+\varpi_v^{f}\cO_v$; thus
\begin{align*}
\int_{t\in D_l(b)}\eta_v(t)\,\d^\times t&=
\eta_v(\varpi_v^{l}) \,\eta_v\left(\varpi_v^{f-l}\tfrac{-b}{b+1}\right)
\int_{r\in 1+\varpi_v^{f}\cO_v}\eta_v(r)\,\d^\times r\\
&=\eta_v\left( \varpi_v^{f}\tfrac{-b}{b+1}\right)\,q_v^{-f-d_{v}/2}(1-q_v^{-1})^{-1}.
\end{align*} 
\end{proof}

\begin{lem}\label{ramifiedJ-L2}
Let $l<f$. Then, $D_l(b)=\emp$ unless $l=f-{\rm{ord}}_v(b+1)+{\rm{ord}}_{v}(b)$, in which case, we have ${\rm{ord}}_v(b+1)>0$, ${\rm{ord}}_v(b)=0$ and 
\begin{align*}
\int_{t\in D_l(b)}\eta_v(t)\,\d^\times t=
\eta_v\left(-\varpi_v^{f}\,b(b+1)^{-1}\right)\,(1-q_v^{-1})^{-1}q_v^{-f-d_{v}/2}. 
\end{align*}
\end{lem}
\begin{proof} 
 This is proved in the same way as the previous lemma. 
\end{proof}

\begin{lem} \label{ramifiedJ-L3}
The set $D_f(b)$ is empty unless ${\rm{ord}}_v(b)={\rm{ord}}_v(b+1)\leq 0$, in which case 
\begin{align*}
\int_{t\in D_f(b)}\eta_v(t) \,\d^\times t&=\delta\left(b(b+1)^{-1}\in \cO_v^\times\right) \, \eta_v\left(-\varpi_v^{f}\,b(b+1)^{-1}\right)\,(1-q_v^{-1})^{-1}\,q_v^{-f+\ord_{v}(b)-d_{v}/2}.
\end{align*}
\end{lem}
\begin{proof}
By $t=-\varpi_v^{f}t'$, the set $D_f(b)$ is mapped bijectively onto the set of $t'$ such that 
\begin{align}
&t'\in \cO_v^{\times}-U_v(f), 
 \label{ramifiedJ-L3-1}
\\
&|1-t'|_v\,|b-t'(b+1)|_v\leq q_v^{-f}. 
 \label{ramifiedJ-L3-2}
\end{align}
We shall show that \eqref{ramifiedJ-L3-1} and \eqref{ramifiedJ-L3-2} are equivalent to the following conditions:
\begin{align}
&t'\in \tfrac{b}{b+1}(1+\varpi_v^{f}b^{-1}\cO_v), 
 \label{ramifiedJ-L3-3}
\\
&\tfrac{b}{b+1}\in \cO_v^{\times}, \qquad b\not\in \fp_v.
 \label{ramifiedJ-L3-4}
\end{align}
Noting that, under the condition \eqref{ramifiedJ-L3-4}, the sets $U_v(1)$ and $\frac{b}{b+1}(1+\varpi_v^{f}b^{-1}\cO_v)$ are disjoint, we see easily that \eqref{ramifiedJ-L3-3} and \eqref{ramifiedJ-L3-4} imply \eqref{ramifiedJ-L3-1} and \eqref{ramifiedJ-L3-2}. To have the converse, we first observe that \eqref{ramifiedJ-L3-1} is equivalent to $t'\in \cO_v^\times$ and $|t'-1|_v>q_v^{-f}$. Hence by \eqref{ramifiedJ-L3-2}, \begin{align*}
|b-t'(b+1)|_v\leq q_v^{-f}|t'-1|_v^{-1} <1,
\end{align*}
or equivalently 
\begin{align}
b-t'(b+1)\in \fp_v.
\label{ramifiedJ-L3-5}
\end{align}
If $b\in \fp_v$, then $b+1\in\cO_v^\times$. From these and \eqref{ramifiedJ-L3-5}, $t'\in \frac{b}{b+1}+\fp_v=\fp_v$; this contradicts $t'\in \cO_v^\times$. Thus $b\not\in \fp_v$ is obtained. From \eqref{ramifiedJ-L3-5}, we have $t'\frac{b+1}{b}\in 1+ \fp_v \subset \cO_v^\times$. Since $t'\in \cO_v^\times$ by \eqref{ramifiedJ-L3-1}, $\frac{b}{b+1}\in \cO_v^\times$ is obtained. From \eqref{ramifiedJ-L3-5}, 
\begin{align*}
t'\in \tfrac{b}{b+1}+\tfrac{1}{b+1}\fp_v=\tfrac{b}{b+1}(1+b^{-1}\fp_v).
\end{align*}
Since $b^{-1}\in \cO_v$, we have $t' \in \frac{b}{b+1}\,U_v(1)$, which yields $t'\in \cO_v^\times -U_v(1)$ because $\frac{b}{b+1}U_v(1)\cap U_v(1)=\emp$. Thus $|t'-1|_v=1$. Combining this with \eqref{ramifiedJ-L3-2}, we obtain \eqref{ramifiedJ-L3-3}. This settles the desired converse implication. 

Consequently, we have
\begin{align*}
\int_{t\in D_f(b)} \eta_v(t)\, \d^\times t&=
\eta_v(-\varpi_v^{f})\,\int \eta_v(t')\,\d^\times t'
\\
&=\delta\left(\tfrac{b}{b+1}\in \cO_v^\times, \,b\not\in \fp_v\right)\,
\eta_v(-\varpi_v^{f})\,\eta_v\left(\tfrac{b}{b+1}\right) 
\int_{r\in 1+\varpi_v^{f}b^{-1}\cO_v} \eta_v(r)\,\d^\times r
\\
&=\delta\left(\tfrac{b}{b+1}\in \cO_v^\times\right)\,
\eta_v\left(\varpi_v^{f}\tfrac{-b}{b+1}\right)\,\delta(b\in \cO_v^\times)\,q_v^{-f+\ord_{v}(b)-d_{v}/2}(1-q_v^{-1})^{-1}. 
\end{align*}
\end{proof}

\begin{lem} \label{ramifiedJ-exact}
Let $\eta_v$ be a character of $F_v^{\times}$ of order $2$ and of conductor $f>0$. Then, for $b\in F_v^{\times}-\{-1\}$, we have
\begin{align*}
J_v^{\eta_v}(b)
&=\delta(b\in \fp_v^{-f})\,\big\{\eta_v(-1)
+( \delta(b \in \go_{v})+\delta(b \notin \go_{v})q_{v}^{\ord_{v}(b)})
\eta_v(-b(b+1))\}\,q_v^{-f-d_{v}/2}(1-q_v^{-1})^{-1}. 
\end{align*}
\end{lem}
\begin{proof} From \cite[Lemmas 11.4 and 11.5]{Tsud}, 
\begin{align*}
J_v^{\eta_v}(b)&= \delta(b\in \fp_v^{-f})(J^{\eta_v}_{v,1}(b)+J^{\eta_v}_{v,2}(b))
\end{align*} 
with 
\begin{align*}
J^{\eta_v}_{v,1}(b)=\int_{\substack{-t\in \varpi_v^{f}\,U_v(f) \\ |t|_v|b+1|_v \leq 1}} \eta_v(-1)\,\d^\times t ,\qquad 
J^{\eta_v}_{v,2}(b)=\int_{\substack{ -t\in F_v^\times-\varpi_v^{f}\,U_v(f) \\ |1+t\varpi_v^{-f}|_v\,|b+t\varpi_v^{-f}(b+1)|_v \leq |t|_v}} \eta_v(t\varpi_v^{-f})\,\d^\times t. 
\end{align*}
If $b\in \fp_{v}^{-f}$, then $t\in -\varpi_v^{f}\,U_v(f)$ implies $|b+1|_v\leq q_v^{f}=|t|_v^{-1}$; thus, 
\begin{align*}
J_{v,1}^{\eta_v}(b)&=\eta_v(-1)\,{\rm{vol}}(-\varpi_v^{f}U_v(f);\d^\times t)=\eta_v(-1){\rm{vol}}(U_v(f);\d^\times t)=\eta_v(-1)\,q_{v}^{-f-d_{v}/2}(1-q_v^{-1})^{-1}.
\end{align*}
The integral domain of $J_{v,2}^{\eta_v}(b)$ is a disjoint union of the sets $D_l(b)\,(l\in \Z)$. From Lemmas~\ref{ramifiedJ-L1}, \ref{ramifiedJ-L2} and \ref{ramifiedJ-L3}, we have 
$$
J_{v,2}^{\eta_v}(b)=(\delta(b\in \go_{v})+\delta(b \notin \go_{v})q_{v}^{\ord_{v}(b)})\eta_v\left(\tfrac{-b}{b+1}\right)\,q_v^{-f-d_{v}/2}(1-q_v^{-1})^{-1}.
$$
\end{proof}

\begin{lem} \label{ramifiedJ}
Let $\eta$ be an idele class character of $F^\times$ with conductor $\ff$ such that $\eta^2=\1$. There exists a constant $C>1$ independent of $\eta$ such that
\begin{align*}
|J_v^{\eta_v}(b)|\leq C\,\delta(|b|_v\leq q_v^{f(\eta_v)})\,q_v^{-f(\eta_v)}
\end{align*}
for any $v\in S(\ff)$ and for any $b\in F_v^{\times}-\{-1\}$. 
\end{lem}
\begin{proof}
 This is obvious from the previous lemma. Indeed, $C=4$ is sufficient. 
\end{proof}

\begin{cor} \label{ramifiedJ-cor}
For any $\epsilon>0$, we have
\begin{align*}
|\prod_{v\in S(\ff)} J_v^{\eta_v}(b)|\ll_{\epsilon} \{\prod_{v\in S(\ff)}\delta(b\in \fp_v^{-f(\eta_v)})\}\,\nr(\ff)^{-1+\epsilon}, \quad b\in F^{\times}-\{-1\}
\end{align*}
with the implied constant independent of $\eta$ and $b\in F^{\times}-\{-1\}$. 
\end{cor}
\begin{proof}
Given $\epsilon>0$, let $P(\epsilon)$ be the set of $v\in \Sigma_\fin$ such that $q_v\leq C^{1/\epsilon}$, where $C>1$ is the constant in the previous lemma. Then, from the lemma, 
$$|J_v^{\eta_v}(b)|\leq C\,\delta(|b|_v\leq q_v^{f(\eta_v)})\,q_v^{-f(\eta_v)+\epsilon} \quad \text{if $v\in S(\ff)\cap P(\epsilon)$}
$$
and 
$$|J_v^{\eta_v}(b)|\leq \delta(|b|_v\leq q_v^{f(\eta_v)})\,q_v^{-f(\eta_v)+\epsilon} \quad \text{if $v\in S(\ff)-P(\epsilon)$.}
$$
Taking the product of these inequalities, we have 
{\small \begin{align*}
|\prod_{v\in S(\ff)} J_v^{\eta_v}(b)|
&=\{\prod_{v\in S(\ff)\cap P(\epsilon)} |J_v^{\eta_v}(b)|\}\,\{\prod_{v\in S(\ff)-P(\epsilon)} |J_v^{\eta_v}(b)|\}
\leq C^{\#S(\ff)}\,\{\prod_{v\in S(\ff)}\delta(b\in \fp_v^{-f(\eta_v)})\}\,\nr(\ff)^{-1+\epsilon}.
\end{align*}}
\end{proof}

\subsection{An evaluation of archimedean integrals}

In this subsection, we evaluate the integral 
\begin{align}
J^{\eta}(l;b)=\int_{\R^\times} \Psi^{(0)}\left(l;\delta_b\,\left[\begin{smallmatrix} t & 0 \\ 0 & 1 \end{smallmatrix}\right]\right)\,\eta(t)\,\d^\times t
, \quad b\in \R^{\times}-\{-1\}
 \label{JIntegral}
\end{align}
explicitly, where $\eta:\R^\times \rightarrow \{\pm 1\}$ is a character, and $\Psi^{(0)}(l;-)$ is the holomorphic Shintani function of weight $l \,(\geq 4)$.  

\begin{lem}\label{MMS6.1} We have 
$$
J^\eta(l;b)=\int_{\R^\times}(1-it)^{-l/2}(1+b+t^{-1}bi)^{-l/2}\,\eta(t)\,\d^\times t, \quad b\in \R^{\times}-\{-1\}.
$$
\end{lem}
\begin{proof}
From Lemma~\ref{Shintani-val-unipotent}, 
\begin{align*}
\Psi^{(0)}\left(l ;\delta_b \left[\begin{smallmatrix} t & 0 \\ 0 & 1 \end{smallmatrix}\right]\right)=e^{il\theta}\,(1+ix)^{-l/2}\quad \text{if} \quad \delta_b \left[\begin{smallmatrix} t & 0 \\ 0 & 1 \end{smallmatrix}\right]\in T \left[\begin{smallmatrix} 1 & x \\ 0 & 1 \end{smallmatrix}\right]\,k_\theta.
\end{align*}
A direct computation yields $e^{i\theta}=\tfrac{1+it}{\sqrt{t^2+1}}$ and $x=bt^{-1}+t(b+1)$. Thus,
\begin{align*}
J^{\eta}(l;b)&=\int_{\R^\times} \left(\tfrac{1+it}{\sqrt{t^2+1}}\right)^{l}\,\{1+i(bt^{-1}+t(b+1))\}^{-l/2}\,\eta(t)\,\d^\times t
\\
&=\int_{\R^\times}(1-it)^{-l/2}(1+b+t^{-1}bi)^{-l/2}\,\eta(t)\,\d^\times t.
\end{align*}
\end{proof}

\begin{lem} \label{Jplus}
Define 
$$ J_{+}(l;b)=i^{l/2}\,(1+b)^{-l/2} \int_{0}^{\infty} (t+i)^{-l/2}\left(t+\tfrac{bi}{b+1}\right)^{-l/2}t^{l/2-1}\,\d t. 
$$
Then 
$$
J^{\1}(l;b)=J_{+}(l;b)+\overline{J_{+}(l;b)}, \qquad J^{{\rm{sgn}}}(l;b)=J_{+}(l;b)-\overline{J_{+}(l;b)}. 
$$
\end{lem}
\begin{proof}
By dividing the integral $J^\eta(l;b)$ to two parts according to $t>0$ and $t<0$, we obtain the assertions immediately. 
\end{proof}

\begin{lem} \label{EV-Jplus}
 Suppose $b(b+1)>0$. Then 
\begin{align*}
 J_{+}(l;b)&=(1+b)^{-l/2}
\int_{0}^{1} u^{l/2-1}(1-u)^{l/2-1}\left(\tfrac{-1}{b+1}u+1\right)^{-l/2}\,\d u
\\
&=(1+b)^{-l/2}\,{\Gamma(l/2)^2}{\Gamma(l)}^{-1}\,{}_2F_1\left(l/2,l/2;l;(b+1)^{-1}\right)=2\,Q_{l/2-1}(2b+1),
\end{align*}
where $Q_{n}(x)$ is the Legendre function of the 2nd kind. 
\end{lem}
\begin{proof}
If we set $f(z)=i^{l/2}(1+b)^{-l/2}\,(z+i)^{-l/2}\{z+bi/(b+1)\}^{-l/2}z^{l/2-1}$, then $f(z)$ is a meromorphic function on $\C$ with poles only at $z=-i$ and $\frac{-bi}{1+b}$, both of which are in the lower half plane $\Im(z)<0$. For $R>0$, let $Q_R$ denote the rectangle $0\leq \Im(z)\leq R$, $0\leq \Re(z)\leq R$. Regarding $\partial Q_R$ as a contour with counterclockwise orientation, by Cauchy's theorem, we have $\int_{\partial Q_R} f(z)\,\d z=0.$ From this, 
\begin{align*} 
J_+(l;b)&=\int_{0i}^{i\infty} f(z)\,\d z-\lim_{R\rightarrow \infty} \int_{\partial Q_R-[0,R]\cup i[0,R]} f(z)\,\d z
=\int_{0i}^{i\infty} f(z)\,\d z\\
&=(1+b)^{-l/2}\int_{0}^{+\infty} (t+1)^{-l/2}\left(t+\tfrac{b}{b+1}\right)^{-l/2}t^{l/2-1}\,\d t.
\end{align*}
By the variable change $t+1=u^{-1}$, this becomes 
$$(1+b)^{-l/2}
\int_{0}^{1} u^{l/2-1}(1-u)^{l/2-1}\left(\tfrac{-1}{b+1}u+1\right)^{-l/2}\,\d u.$$
By using the integral representation of ${}_2F_1(a,b;c,z)$ in \cite[p.54]{MOS} here, we obtain
$$
J_+(l;b)=(1+b)^{-l/2}\,{\Gamma(l/2)^2}{\Gamma(l)}^{-1}\,{}_2F_1\left(l/2,l/2;l;(b+1)^{-1}\right). 
$$
If we further apply the formula
$$
{2^{-n}(2n+1)!}(n!)^{-2}\,Q_n(x)=(1+x)^{-(n+1)}{}_2F_1\left(n+1,n+1;2n+2;\tfrac{2}{x+1}\right)
$$
(\cite[p.233]{MOS}) with $n=l/2-1$ and $x=2b+1$, then $J_{+}(l;b)=2\,Q_{l/2-1}(2b+1)$ as desired. 
\end{proof}

\begin{lem}\label{arcInt} 
\begin{itemize}
\item[(1)] If $b(b+1)>0$, then 
$$J^{\1}(l;b)=(1+b)^{-l/2}\,{2\,\Gamma(l/2)^2}{\Gamma(l)}^{-1}\,{}_2F_1\left(l/2,l/2;l;(b+1)^{-1}\right),\qquad J^{{\rm{sgn}}}(l;b)=0.$$ 
\item[(2)] If $b(b+1)<0$, then 
\begin{align}
J^\1(l;b)&=2\log\left|\frac{b+1}{b}\right|\,P_{l/2-1}(2b+1)-\sum_{m=1}^{[l/4]}\frac{8(l-4m+1)}{(2m-1)(l-2m)}\,P_{l/2-2m}(2b+1), 
 \label{arcInt2}
\\
J^{{\rm{sgn}}}(l;b)&=2\pi i\, P_{l/2-1}(2b+1),
 \label{arcInt3}
\end{align}
where $P_n(z)$ denotes the Legendre polynomial of degree $n$. 
\end{itemize}
\end{lem}
\begin{proof}
First suppose $b\in \R$ and $b(b+1)>0$. Then from the previous lemma, $J_{+}(l;b)$ is a real number. Thus, $J^{\1}(l;b)=2J^{+}(l;b)$ and $J^{{\rm{sgn}}}(l;b)=0$ by Lemma~\ref{Jplus}. 

From $J_+(l;b)=2Q_{l/2-1}(2b+1)$, applying the formula in \cite[p.234]{MOS}, we obtain  
\begin{align}
J_{+}(l;b)=\log\left(\frac{b+1}{b}\right)\,P_{l/2-1}(2b+1)-\sum_{m=1}^{[l/4]}\frac{4(l-4m+1)}{(2m-1)(l-2m)}\,P_{l/2-2m}(2b+1) 
\label{acrInt1}
\end{align}
for $b\in\R$ such that $b(b+1)>0$. From the defining formula of $J_{+}(l;b)$, the function $b\mapsto J_{+}(l;b)$ on $\R^{\times}-\{-1\}$ has a holomorphic continuation to the whole complex $b$-plane away from the set
$
S=\{b\in \C|\,\tfrac{bi}{b+1}\in (-\infty,0)\}\cup\{0,-1\},
$
which is the upper half of the circle centered at $-1/2$ of radius $1/2$ with the edge points included. Thus, if we choose the branch of $\log(\frac{1+b}{b})$ on the domain $\C-S$ so that it is real for $b>0$, then the formula \eqref{acrInt1} remains valid on $\C-S$ by analytic continuation. Let $b\in \R$ such that $b(b+1)<0$. Such $b$ is contained in $\C-S$. Hence, by taking the sum of \eqref{acrInt1} and its complex conjugate, we obtain the formula for $J^{\1}(l;b)$. As for $J^{{\rm{sgn}}}(l;b)$, we have
\begin{align*}
J^{{\rm{sgn}}}(l;b)&=J_{+}(l;b)-\overline{J_{+}(l;b)}
= \{ \, \log(\tfrac{b+1}{b}) - \overline{\log(\tfrac{b+1}{b})}\, \} P_{l/2-1}(2b+1)
= 2\pi i P_{l/2-1}(2b+1).
\end{align*}
\end{proof}

\section{Explicit formula of the unipotent term}

Let $v\in S$. The aim of this section is to evaluate the integrals 
\begin{align}
U_v^{\eta_v}(\alpha_v)&=\frac{1}{2\pi i } \int_{L_v(c)}
(1-\eta_{v}(\varpi_{v})q_v^{-(s+1)/2})^{-1}(1-q_v^{(s+1)/2})^{-1}\,\alpha_v(s)\,\d\mu_v(s),
 \label{UNIPOTENT(fine)-1}
 \\
U'_v(\alpha_v)&=\frac{\log q_v}{2\pi i } \int_{L_v(c)} (1-q_v^{(s+1)/2})^{-2}(1-q_v^{-(s+1)/2})^{-1}\,\alpha_v(s)\,\d\mu_v(s). 
 \label{UNIPOTENT(fine)-2}
\end{align}
for the test functions $\a$ given by \eqref{HeckeFunctions}, where $L_v(c)=c+i[\frac{-2\pi}{\log q_v}, \frac{2\pi}{\log q_v}]$.

\begin{prop} \label{MS5}
Let $\alpha_v(s)=q_{v}^{ms/2}+q_{v}^{-ms/2}$ with $m\in\N_{0}$. We have
\begin{align*}
U_v^{\eta_v}(\alpha_v)&=
\begin{cases}
\delta(m>0)\,q_v^{1-m/2}\,\{(m-1)-(m+1)\,q_v^{-1}\}-2\,\delta(m=0)
, &\quad (\eta_v(\varpi_v)=1),\\
\delta(m \in 2\NN )\ q_{v}^{1-m/2}(1-q_{v}^{-1}) -2\ \delta(m=0), &\quad (\eta_v(\varpi_v)=-1),
\end{cases}
\\
U'_v(\alpha_{v})&=-2^{-1}(\log q_v)\,q_v^{1-m/2}\,\delta(m>0)\,\{(m-1)(m-2)-m(m
+1)\,q_v^{-1}\}.
\end{align*}
\end{prop}
\begin{proof} 
We give an indication of the proof for \eqref{UNIPOTENT(fine)-1} when $\eta_v(\varpi_v)=-1$; the remaining cases are similar. By a variable change, 
\begin{align*}
U_{v}^{\eta_v}(\a_{v})
= & \frac{1}{2\pi i}\oint_{|z|=q_{v}^{c/2}}(1+q_{v}^{-1/2}z^{-1})^{-1}(1-q_{v}^{1/2}z)^{-1}(z^{m}+z^{-m})
q_{v}^{1/2}(z-z^{-1})\frac{dz}{z} \\
= & \{\Res_{z=q_{v}^{-1/2}} + \Res_{z=-q_{v}^{-1/2}} +\Res_{z=0}\} \phi(z),
\end{align*}
where $\phi(z) = \frac{(z^{2}-1)(z^{m}+z^{-m})}{1-q_{v}z^{2}}\frac{q_{v}}{z}$. By evaluating the residues, we are done. 
\end{proof}

\section{Subconvexity estimates in the weight aspect}

In this section we prove Theorem~\ref{T4} by using the relative trace formula (Theorem~\ref{RELATIVETRACEFORMULA}); we take a particular test function $\alpha^\pi_S\in \ccA_S$ depending on a fixed cuspidal representation $\pi$ with varying $S$. To have a good control of the term $\JJ_{\rm{hyp}}^\eta(\fn|\alpha_S^\pi)$ explicating the dependence on $S$, our formula of local orbital integrals (Lemma~\ref{MS2.5}) is indispensable. In this section, $\theta\in [0,1]$ denotes a real number such that the spectral radius of the Satake parameter $A_v(\pi)$ of $\pi \in \Pi_{\rm{cus}}(l, \fn)$ at $v\in \Sigma_\fin-S(\ff_\pi)$ is no greater than $q_v^{\theta/2}$ for any $v\in \Sigma_\fin-S(\ff_\pi)$. Since the Ramanujan conjecture for the holomorphic Hilbert cusp forms is known (\cite{Blasius}), we can actually take $\theta=0$; however, we let $\theta$ unspecified until the very end to be able to keep track of the dependence on the Ramanujan exponent $\theta$ in various estimations.

In this section, we abuse the symbol $\fp_v$ to designate the global ideal $\gp_v\cap \cO$. 

\subsection{An auxiliary estimate of semilocal terms}
Let $S$ be a finite set of finite places $v$ such that $\eta_v(\varpi_v)=-1$.
For a decomposable function $\alpha_S(\bs)=\prod_{v\in S} \alpha_v(s_v)$ in $\ccA_S=\bigotimes_{v\in S} \ccA_v$, we set 
$$
J_S(b;\alpha_S)=\prod_{v\in S}J_v(b;\alpha_v), \quad b\in F^{\times}-\{-1\},
$$
where we simply write $J_v(b;\alpha_v)$ in place of $J_v^{\eta_v}(b;\alpha_v)$. Extending this linearly, we have a linear functional $\alpha_S\mapsto J_S(b;\alpha_S)$ on the space $\ccA_S$. Given $\pi\in \Pi_{\rm{cus}}(l, \fn)$, set
$$
\lambda_v(\pi)=\tr\,A_v(\pi), \qquad v\in \Sigma_\fin-S(\ff_\pi)
$$
with $A_v(\pi)\in {\rm{GL}}_2(\C)$ the Satake parameter of $\pi_v$. Then, we define an element of $\ccA_S$ depending on the automorphic representation $\pi$ as follows:$$
\alpha_S^{\pi} (\bs)=\biggl(\sum_{v\in S}\{\lambda_v(\pi)\,(z_v+z_v^{-1})-(z_v^2+z_v^{-2}+1)\} \biggr)^{2},
$$
where $z_v=q_v^{s_v/2}$ for each $v\in S$. We need an estimate of $J_S(b;\,\alpha_S^{\pi})$ with varying $b$. 
For an $\cO$-ideal $\fa$ such that $S(\fa)\subset S$, let us define a function $D_{S}(\fa;-)$ on $F^{\times}-\{-1\}$ by
\begin{align*}
D_{S}(\fa;\,b)=\{\prod_{w\in S-S(\fa)}\Lambda_w(b)\}\,\{\prod_{v\in S(\fa)}\delta(|b|_v\leq q_v^{{\rm{ord}}_v(\fa)})\}, 
\end{align*}
where $\Lambda_{w}(b) = \delta(b \in \go_{w})(\ord_{w}(b(b+1))+1)$.

\begin{prop} \label{MS6} 
Set $\Sd=\{(v_1,v_2)\in S^2|\,v_1\not=v_2\,\}$. We have the estimate
{\allowdisplaybreaks
\begin{align*}
|J_S(b;\alpha_S^\pi)|\ll&\sum_{v\in S}\biggl\{D_{S}(\cO;b)\,q_v^{(\theta+1)/2}
+D_{S}(\fp_v;b)\,q_v^{\theta} \\
&\qquad +D_{S}(\fp_v^2;b)\,q_v^{\theta-1}+D_{S}(\fp_v^3;b)\,q_v^{-1}+D_{S}
(\fp_v^4;b)\,q_v^{-2}\biggr\} \\
&+\sum_{(v_1,v_2)\in \Sd}\biggl\{
D_{S}(\cO;b)\,q_{v_1}^{(\beta+1)/2}\,q_{v_2}^{(\theta+1)/2}
+D_{S}(\fp_{v_1};b)\,q_{v_2}^{(\theta+1)/2}
 +D_{S}(\fp_{v_1}\fp_{v_2};b)\\
&\qquad 
+D_{S}(\fp_{v_1}^2\fp_{v_2};b)q_{v_1}^{-1}+
D_{S}(\fp_{v_1}^2;b)\,q_{v_1}^{-1}\,q_{v_2}^{(\theta+1)/2}
+D_{S}(\fp_{v_1}^{2}\fp_{v_2}^{2};b)q_{v_1}^{-1}q_{v_2}^{-1}
\biggr\}
\end{align*}
}
for $b\in F^{\times}-\{-1\}$, where the implied constant is absolute.
\end{prop}
\begin{proof}
Set $Z_v=\lambda_v(\pi)\,(z_v+z_v^{-1})-(z_v^2+z_v^{-2}+1)$ for any $v\in S$. By expanding the square, we have $\alpha_S^{\pi}(\bs)=\sum_{v\in S} Z_v^2+\sum_{(v_1,v_2)\in \Sd} Z_{v_1}\,Z_{v_2},$ which, together with Lemma~\ref{MS2.5}, gives us
{\small \begin{align}
& J_S(b;\alpha_S^{\pi})\label{MS6-1}\\
&=\sum_{v\in S} \{\prod_{w\in S-\{v\}} J_{w}(b;1)\}\,J_v(b;Z_v^2)+\sum_{(v_1,v_2)\in \Sd} \{\prod_{w\in S-\{v_1,v_2\}}J_{w}(b;1)\}\,J_{v_1}(b;Z_{v_1})\,J_{v_2}(b;Z_{v_2}) 
\notag
\\
&=\sum_{v\in S} \{\prod_{w\in S-\{v\}} -\vol(\cO_w^\times)\,\Lambda_w(b)\}
\,J_v(b;Z_v^2) 
 \notag
+\sum_{(v_1,v_2)\in \Sd} \{\prod_{w\in S-\{v_1,v_2\}}-\vol(\cO_w^\times)\,\Lambda_w(b)\}\,J_{v_1}(b;Z_{v_1})\,J_{v_2}(b;Z_{v_2}).  \notag
\end{align}
}Let us estimate the integral $J_v(b;Z_v^2)$. By expanding the square, 
{\small \begin{align*}
Z_v^2&=\lambda_v(\pi)^2\,(z_v^2+z_v^{-2}+2)+(z_v^{4}+z_v^{-4}+2)+2(z_v^2+z_v^{-2})+1 -2\lambda_v(\pi)(z_v^3+z_v^{-3})-4\lambda_v(\pi)(z_v+z_v^{-1})
\\
&=\lambda_v(\pi)^2(\alpha_v^{(2)}+\alpha_v^{(0)})+\alpha_v^{(4)}+2\alpha_v^{(2)}+\frac{3}{2}\alpha_v^{(0)}-2\lambda_v(\pi)\alpha_v^{(3)}-4\lambda_v(\pi)\alpha_v^{(1)}. 
\end{align*}
}By this expression and by the estimates in Lemma~\ref{MS2.5}, we obtain
{\small \allowdisplaybreaks
\begin{align}
&|J_v(b;Z_v^2)|\vol(\cO_v^\times)^{-1}  
 \label{MS6-2}
\\
& \ll
\delta(|b|_v\leq 1)\,\{|\lambda_v(\pi)|q_v^{1/2}+\Lambda_v(b)(1+|\lambda_v(\pi)|^2)\}
+\delta(|b|_v\leq q_v)\,\{|\lambda_v(\pi)|q_v^{-1/2}+1+|\lambda_v(\pi)|^2\}
\notag
 \\
& \quad +\delta(|b|_v\leq q_v^2)\,\{|\lambda_v(\pi)|q_v^{-1/2}+q_v^{-1}+|\lambda_v(\pi)|^2q_v^{-1}\}
+\delta(|b|_v\leq q_v^{3})\,\{|\lambda_v(\pi)|q_v^{-3/2}+q_v^{-1}\}
 \notag
\\
&\quad +\delta(|b|_v\leq q_v^{4})\,q_v^{-2}
 \notag
\\
&\ll \delta(|b|_v\leq 1)\,\{q_v^{(\theta+1)/2}+\Lambda_v(b)\,q_v^{\theta} \}
 \notag
\\
&\quad 
+\delta(|b|_v\leq q_v)\,q_v^{\theta}+\delta(|b|_v\leq q_v^{2})\,q_v^{\theta-1}
+\delta(|b|_v\leq q_v^{3})\,q_v^{-1}+\delta(|b|_v\leq q_v^4)\,q_v^{-2},
 \notag
\end{align}
}where to show the second inequality we use the estimate $|\lambda_v(\pi)|\leq 2q_v^{\theta/2}$ as well as the inequalities $-1\leq (\theta-1)/2\leq \theta \leq (\theta+1)/2$, $(\theta-3)/2\leq -1$. For $J_v(b;Z_v)$, directly from Lemma~\ref{MS2.5}, we have
{\small \begin{align}
|J_v(b;Z_v)|\,\vol(\cO_v^\times)^{-1}\ll \delta(|b|_v\leq 1)\,\{q_v^{(\theta+1)/2}+\Lambda_v(b)\}+\delta(|b|_v\leq q_v)+\delta(|b|_v\leq q_v^{2})\,q_v^{-1}. 
 \label{MS6-3}
\end{align}
}From \eqref{MS6-1}, \eqref{MS6-2} and \eqref{MS6-3}, we have the desired estimate immediately.
\end{proof}

\subsection{A basic majorant for the hyperbolic term (odd case)}
For $b\in F^{\times}-\{-1\}$, viewing $b$ as a real number, say $b_v$, by the mapping $F\hookrightarrow F_v\cong \R$ for each $v\in \Sigma_\infty$, we define
$$
\fm_\infty(l;b)=\prod_{v\in \Sigma_\infty}|J^{{\rm{sgn}}}(l_v;b_v)|,
$$
where $J^{{\rm{sgn}}}(l_v;b_v)$ is the integral \eqref{JIntegral}. For relatively prime integral ideals $\fn$ and $\fa$ and for $l=(l_v)_{v\in \Sigma_\infty}\in (2\Z_{>2})^{\Sigma_\infty}$, we set
\begin{align*}
\fI(l,\fn,\fa):=\sum_{b\in \fn\cO(S)-\{0,-1\}} \{\prod_{v\in \Sigma_\fin-S}\Lambda_v(b)\}\,D_S(\fa;b)\,\fm_\infty(l;b),
\end{align*}
where $S$ is a finite set of places such that $S(\fa)\subset S \subset \Sigma_\fin-S(\fn)$, and $\cO(S)$ is the $S$-integer ring. We note that the right-hand side is indepdnent of $S$. Indeed, if we set
$$
\tau^{S(\fa)}(b)=\prod_{v\in \Sigma_\fin-S(\fa)}\Lambda_v(b)\prod_{v \in S(\ga)}\delta(|b|_{v}\le q_{v}^{\ord_{v}(\ga)}),$$ 
then 
\begin{align}
\fI(l,\fn,\fa)=\sum_{b \in \fn\fa^{-1}-\{0, -1\}} \tau^{S(\fa)}(b)\,\fm_\infty(l;b). 
\label{MS6.5-1}
\end{align}

\begin{lem} \label{MS6.5}
Let $\fa$ and $\fn$ be relatively prime ideals. Then, for any $\epsilon>0$, the estimate 
\begin{align*}
\fI(l,\fn,\fa) \ll_{\epsilon}  \{\prod_{v\in \Sigma_\infty}l_v\}^{-1/2}\, \nr(\fa)^{5/4
+ \epsilon} 
\end{align*}
holds with the implied constant depending on $\epsilon$ while independent of the data $(l,\fn, \fa)$. 
If $\fa$ is trivial, for any $\fn$ and $l$, we have $\fI(l,\fn,\cO)=0.$
\end{lem}

\subsection{The proof of Lemma~\ref{MS6.5}}

\begin{lem} \label{MS6.5-L1}
For any $\epsilon>0$, we have
\begin{align*}
\tau^{S(\fa)}(b)\ll_{\epsilon}(\nr(\fa)^2\,|\nr(b(b+1))|)^{\epsilon}, \qquad b\in \fa^{-1}-\{0, -1\}
\end{align*}
with the implied constant independent of $b$. 
\end{lem}
\begin{proof}
Let $b\in \fa^{-1}-\{0\}$; then 
$(b(b+1))\fa^{2}=\fb\,\prod_{j=1}^{r}\fp_j^{e_j}$
for some $e_j \in \NN$, where $\fp_j$ are prime ideals of $\cO$ relatively prime to $\fa$ and $\fb$ is an ideal of $\cO$ dividing $\fa$. For each $j$, there exist a prime number $p_j$ and $d_j\in \N$ such that $\nr(\fp_j)=p_j^{d_j}$. By taking norms, we have 
\begin{align*}
\nr(\fa)^2\,|\nr(b(b+1))|=\nr(\fb)\prod_{j=1}^{r}\nr(\fp_j)^{e_j} =\nr(\fb)\,\prod_{j=1}^{r}p_j^{d_j e_j}.
\end{align*}
Hence
$$
d(\nr(\fa)^2 |\nr(b(b+1))|)=d(\nr(\fb))\,\prod_{j=1}^{r}(e_jd_j+1)\geq \prod_{j=1}^{r}(e_j+1)=\tau^{S(\fa)}(b),
$$
where, for a natural number $m$, $d(m)$ denotes the number of positive divisors of $m$. Invoking the well known bound $d(m)\ll_{\epsilon} m^{\epsilon}$, we obtain the desired estimate. 
\end{proof}
From Lemmas~\ref{MS6.5-L1} and \ref{arcInt}, 
\begin{align*}
\fI(l,\fn,\fa)\ll_{\epsilon}\nr(\fa)^{2\epsilon}\,\sum_{b\in \fn\fa^{-1}\cap \cQ_\infty} |\nr(b(b+1))|^{\epsilon} \prod_{v\in \Sigma_\infty}|2\pi\,P_{l_v/2-1}(2b_v+1)|,
\end{align*}
where $\cQ_\infty$ denotes the cube $(-1,0)^{\Sigma_\infty}$ in $\prod_{v\in \Sigma_\infty}\R$. Invoking the inequality $|P_{n}(x)|\leq (1-x^2)^{-1/2}\,n^{-1/2}$ for $|x|< 1,\, n\in \N$ (\cite[p.237]{MOS}), we have
\begin{align}
\fI(l,\fn,\fa)\ll_{\epsilon} \pi^{d_F}\,\nr(\fa)^{2\epsilon}\,\sum_{b\in \fn\fa^{-1}\cap \cQ_\infty} |\nr(b(b+1))|^{-1/4+\epsilon} \prod_{v\in \Sigma_\infty}(l_v/2-1)^{-1/2}.
 \label{Separation of l}
\end{align}
To estimate the sum $\sum_{b\in \fn\fa^{-1}\cap \cQ_\infty}|\nr(b(b+1))|^{-1/4+\epsilon}$, we need several lemmas.  

\begin{lem} \label{IdealEst}
For a positive integer $c$, let $\nu(c)$ be the number of $\cO$-ideals $\fc$ such that ${\rm N}(\fc)=c$. Then, for any $\epsilon>0$, $\nu(c) \ll_{\epsilon} c^{\epsilon}
$ with the implied constant independent of $c$. 
\end{lem}
\begin{proof} Suppose $c$ is a prime power $p^{t}$. Then an ideal $\fc$ such that $\nr(\fc)=p^{t}$ must be a power of a prime ideal $\fp$ lying above $p$. The number of choices for such $\fp$ is at most $d_F=[F:\Q]$. If $\fc=\fp^{e}$, then $\nr(\fc)=p^{t}$ is equivalent to $p^{me}=p^{t}$, where $\nr(\fp)=p^{m}$. Hence $e=\tfrac{t}{m}\leq t \leq t {\log_2 p} \leq {\log_2 p^{t}}. $ From this, we have the inequality $\nu(p^{t})\leq d_F\,\log_2 p^{t}$. Given $\epsilon>0$, let $x(\epsilon)>1$ be a number such that $d_F\,\log_2 x\leq x^{\epsilon}$ for any $x\geq x(\epsilon)$. Let $Q(\epsilon)$ be the set of prime powers $p^{t}$ such that $p^{t} \leq x(\epsilon)$. Noting that $Q(\epsilon)$ is a finite set, we set $C(\epsilon)=\prod_{q\in Q(\epsilon)}\nu(q)$, which is a constant depending only on $\epsilon$. Let $c'$ (resp. $c''$) be the product of the prime powers $p_i^{t_i}$ such that $p_i^{t_i}\in Q(\epsilon)$ (resp. $p_{i}^{t_i}\not\in Q(\epsilon)$) in the prime factorization $c=\prod_{i}p_i^{t_i}$ of $c$. Since $\nu$ is multiplicative, we have
{\small$$\nu(c)=\nu(c')\,\nu(c'')\leq \prod_{q\in Q(\epsilon)} \nu(q) \,\prod_{i;\,p_i|c''} d_F\,\log_2 p_i^{t_i} \leq C(\epsilon)\,\prod_{i;\,p_i|c''} p_i^{t_i\epsilon} \leq C(\epsilon)\,(\prod_{i} p_i^{t_i}) ^\epsilon\leq C(\epsilon)\,c^{\epsilon}. $$
}This completes the proof.
\end{proof}

\begin{lem} \label{UnitEstimate}
Let $C=\{C_v\}_{v\in \Sigma_\infty}$ be a family of positive real numbers. For any $\epsilon>0$, we have
$$
\# \{u\in \cO^\times|\,|u_v|<C_v\,(\forall\,v\in \Sigma_\infty)\,\} \ll_{\epsilon} \biggl(\prod_{v\in \Sigma_\infty}C_v\biggr)^{\epsilon} 
$$
with the implied constant independent of $C$. 
\end{lem}
\begin{proof} For simplicity, we set $d=d_F$. By the Dirichlet unit theorem, there exist fundamental units $\varepsilon_j\,(1\leq j\leq d-1)$ such that any $\gamma\in \cO^\times$ is written uniquely in the form $\gamma=\pm \varepsilon_1^{n_1}\cdots \varepsilon_{d-1}^{n_{d-1}}$ with integers $n_j\in \Z$. By this, the inequality $|\gamma_v|<C_v$ is written as
\begin{align}
\sum_{j=1}^{d-1} n_j\,\log|(\varepsilon_j)_{v}|<\log C_v, \qquad (v\in \Sigma_\infty). 
 \label{Unitineq}
\end{align}
Let $\fU(C)$ be the set of $u\in \cO^\times$ such that $|u_v|<C_v$ for all $v\in \Sigma_\infty$. Thus, the number $\# \fU(C)$ is bounded from above by the number of integer points $(n_j)\in \Z^{d-1}$ lying on the Euclidean domain $D(C)$ in $\R^{d-1}$ defined by the system of linear inequalities \eqref{Unitineq}. Fix an enumeration $\Sigma_\infty=\{v_1,\dots,v_{d}\}$ and let $E_i=(\log|(\varepsilon_j)_{v_i}|)_{1\leq j\leq d-1}\in \R^{d-1}$ for $1\leq i\leq d$. First $d-1$ vectors $E_i\,(1\leq i\leq d-1)$ form a basis of $\R^{d-1}$; let $E_{j}^{*}$ $(1\leq j\leq d-1)$ be its dual basis. From the relation $|\nr(\varepsilon_j)|=1$, we have $\sum_{i=1}^{d}E_i=0$. Hence, if we write a general point $y\in \R^{d-1}$ by $y=\sum_{i=1}^{d-1}(\log C_{v_i}-y_i)\,E_i^{*}$, then $y\in D(C)$ if and only if
$$ y_i>0 \qquad (1\leq i\leq d-1), \quad \sum_{i=1}^{d-1} y_i<\sum_{j=1}^{d}\log C_{v_j}.
$$
The volume of this region in the $y$-space with respect to the Euclidean measure is $\frac{1}{r_{F}(d-1)!}\,(\sum_{j=1}^{d}\log C_{v_j})^{d}$, where $r_{F}$ is the regulator of $F$. Thus $\vol(D(C))\ll (\log \prod_{v}C_v)^{d}\ll_{\epsilon} (\prod_{v}C_v)^{\epsilon}$, and we are done.  
\end{proof}

\begin{lem} \label{IntegerEST}
Let $\fa$ be an integral ideal and $c$ a positive integer. 
For any $\epsilon, \e'>0$, 
$$
\#\{b \in \fa^{-1}\cap \cQ_\infty|\,\nr((b)\fa)=c\,\}\ll_{\epsilon, \e'} c^{\e' -\e}\,\nr(\fa)^{\epsilon}
$$with the implied constant independent of $\fa$ and $c$. 
\end{lem}
\begin{proof}
Let $\fc$ be an integral ideal such that $\nr(\fc)=c$. From Lemma~\ref{IdealEst}, the number of such $\fc$ is bounded by $c^{\epsilon'}$ for any $\epsilon'>0$. If $\fc\fa^{-1}$ is a principal ideal, say $(\xi)$, then, using Lemma~\ref{UnitEstimate}, we have
\begin{align*}
\#\{b \in \fa^{-1}\cap \cQ_\infty|\,\fc=(b)\fa\,\}& =\#\{u \in \cO^\times|\,|u_v|<|\xi_v|^{-1}\,(\forall\,v\in \Sigma_\infty)\,\}
 \\
&\ll_{\epsilon} (\prod_{v\in \Sigma_\infty}|\xi_v|^{-1})^{\epsilon}=(|\nr(\xi)|^{-1})^{\epsilon}=(c^{-1}\,\nr(\fa))^{\epsilon}.  
\end{align*} 
\end{proof}

\subsubsection{The completion of the proof of Lemma \ref{MS6.5}}
 From \eqref{Separation of l}, we have 
\begin{align}
\fI(l,\fn, \ga)\ll_{\epsilon} \nr(\fa)^{2\epsilon}\,\{\prod_{v\in \Sigma_\infty}l_v\}^{-1/2}\,\sum_{b\in \fa^{-1} \cap \cQ_\infty}|\nr(b(b+1))|^{-1/4+\e}
 \label{MS6.5-1}
\end{align}
with the implied constant independent of $(l,\fn,\fa)$. Setting $\nr((b)\fa)=c$, we rewrite the last summation in the following way. 
\begin{align*}
\sum_{b\in \fa^{-1} \cap \cQ_\infty}|\nr(b(b+1))|^{-1/4+\e}
&=\nr(\fa)^{1/4-\e}\,\sum_{c=1}^{\infty} c^{-1/4+\e}\,\sum_{\substack{b \in \fa^{-1} \cap \cQ_\infty \\ |\nr((b)\fa)|=c}}|\nr(b+1)|^{-1/4+\e}.
\end{align*}
The range of $c$ is reduced to $1\leq c\leq \nr(\fa)$ by the condition $b\in \cQ_\infty$. Since $(0)\not=(b+1)\fa\subset \cO$, we have $\nr((b+1)\fa) \geq 1$, by which the last summation in $b$ is trivially bounded by $\nr(\fa)^{1/4-\e}\,\#\{b\in \fa^{-1}\cap \cQ_\infty|\,|\nr((b)\fa)|=c\,\}$ for any $\e \in (0, 1/4)$. Combining these considerations and by Lemma~\ref{IntegerEST}, we obtain the bound
\begin{align}
\sum_{b\in \fa^{-1}\cap \cQ_\infty}|\nr(b(b+1))|^{-1/4+\e}
&\ll_{\e, \delta, \delta'}
\nr(\fa)^{1/2-2\e} \sum_{c=1}^{\nr(\fa)}c^{-1/4+\e}c^{\delta'- \delta} \,\nr(\fa)^{\delta}
 \notag
\\
&\ll_{\e, \delta, \delta'} \nr(\fa)^{1/2-2\e} \nr(\fa)^{\delta}\times \nr(\fa)^{3/4+\e+\delta'-\delta} \log \nr(\ga) \notag \\
& = \nr(\fa)^{5/4 -\e +\delta'} \log \nr(\ga)
 \label{MS6.5-2}
\end{align}
for any sufficiently small $\delta, \delta'>0$. Consequently, we have the desired estimate from \eqref{MS6.5-1} and \eqref{MS6.5-2}. It remains to show the second assertion of Lemma~\ref{MS6.5}. To argue, suppose $b\in \gn \cap \cQ_\infty$. The integrality of $b$ yields $\nr(b)\in \Z$. From the condition $b\in \cQ_\infty$, we have $0<|b_v|<1$ for all $v\in \Sigma_\infty$, from which $0<|\nr(b)|<1$ is obtained. Thus, if $\fa=\cO$, then the summation in the right-hand side of \eqref{Separation of l} is empty. This completes the proof.

\subsection{An estimate of the hyperbolic term }
 Let $\eta$ be a quadratic idele class character of $F^\times$ with conductor $\ff$ such that $\eta_v(-1)=-1$ for all $v\in \Sigma_\infty$. 
 Given an integral ideal $\fn$, for a large number $K\ge 2$, let $S=S_K^{\fn,\eta}=\{v\in \Sigma_\fin-S(\fn\ff)|\,\eta_v(\varpi_v)=-1,\,K\leq q_v\leq 2K\,\}$, and consider the test function $\alpha^{\pi}_{S}(\bs)$ depending on a cuspidal representation $\pi \in \Pi_{\rm{cus}}(l,\fn)$. 

\begin{lem} \label{Principalidealdensity}
There exists a constant $C>1$ independent of $\fn$ and $\eta$ such that \\
$C^{-1}{K}({\log K})^{-1}<\# S <C\,K(\log K)^{-1}$ for all $K\ge 2$.   
\end{lem}
\begin{proof} This follows from an analogue of Dirichlet's theorem on arithmetic progression for number fields.
\end{proof}

For $S = S_{K}^{\gn, \eta}$ and for a given $\pi \in \Pi_{\rm{cus}}(l,\fn)$, let $\alpha_S^{\pi}(\bs)$ be the function defined in \S 12.1.

\begin{prop} \label{MS10}
For any $\epsilon>0$, we have
\begin{align*}
|\JJ_{\rm{hyp}}^{\eta}(l,\fn|\alpha_{S}^{\pi})|\ll _{\epsilon} \{\prod_{v\in \Sigma_\infty}l_v\}^{-1/2}\,\nr(\ff)^{1/4+\epsilon}\,K^{5+\epsilon}\end{align*}
with the implied constant independent of $l$, $\fn$, $\pi$, $\eta$ and $K$. 
\end{prop}
\begin{proof} Set $\Sd=\{(v_1,v_2)\in S^2|\,v_1\not=v_2\,\}$. From Lemmas~\ref{MS3} and \ref{MS4}, we have the bound
\begin{align*}
|\JJ_{\rm{hyp}}^{\eta}(l,\fn|\alpha_{S}^{\pi})|\ll \sum_{b\in F^{\times}-\{-1\}} 
|J_{S}(b;\alpha_{S}^\pi)|\,\{\prod_{v\in \Sigma_\fin-S\cup S(\ff)}\Lambda_v(b)\}\,\{\prod_{v\in S(\ff)} |J_v^{\eta_v}(b)|\}\,\fm_\infty(l;b). 
\end{align*}
Combining this with Corollary~\ref{ramifiedJ-cor} and Proposition~\ref{MS6}, we have that this is majorized by the $\nr(\ff)^{-1+\epsilon}$ times the following expression
{\small \allowdisplaybreaks
\begin{align}
&\{\sum_{v\in S} q_v^{(\theta+1)/2}\}\, \fI(l,\fn,\ff)
+\sum_{v\in S} q_v^{\theta}\,\fI(l,\fn,\fp_v\ff)
+\sum_{v\in S} q_v^{\theta-1}\,\fI(l,\fn,\fp_v^2\ff)
 \notag
+\sum_{v\in S} q_v^{-1}\,\fI(l,\fn,\fp_v^3\ff)
+\sum_{v\in S} q_v^{-2}\,\fI(l,\fn,\fp_v^4\ff)
 \notag \\
& +\{\sum_{(v_1,v_2)\in \Sd} q_{v_1}^{(\theta+1)/2}\,q_{v_2}^{(\theta+1)/2}\}\,\fI(l,\fn,\ff) 
\notag
+\sum_{(v_1,v_2)\in \Sd}q_{v_1}^{(\theta+1)/2}\,\fI(l,\fn,\fp_{v_2}\ff)
+\sum_{(v_1,v_2)\in \Sd} \fI(l,\fn,\fp_{v_1}\fp_{v_2}\ff)
\notag 
\\
& +\sum_{(v_1,v_2)\in \Sd} q_{v_1}^{-1}\,\fI(l,\fn,\fp_{v_1}^2\fp_{v_2}\ff)
+\sum_{(v_1,v_2)\in \Sd} q_{v_1}^{-1}\,q_{v_2}^{(\theta+1)/2}\,\fI(l,\fn,\fp_{v_1}^{2}\ff) \notag
+\sum_{(v_1,v_2)\in \Sd} q_{v_1}^{-1}q_{v_2}^{-1}\,\fI(l,\fn,\fp_{v_1}^2\fp_{v_2}^2\ff).
\notag
\end{align}
}Invoking the bound $\sharp S \prec K$ obtained from Lemma \ref{Principalidealdensity} and applying Lemma~\ref{MS6.5}, we estimate each term occurring above. Thus, after a power saving, we obtain \\ $|\JJ_{\rm{hyp}}^{\eta}(l, \fn|\alpha_{S}^{\pi})|\ll_{\epsilon} \nr(\ff)^{-1+\epsilon}\,\varphi(l,K)$, where $L=\prod_{v \in \Sigma_{\infty}}l_v$ and $\varphi(l,K)$ is
\begin{align*}
&\nr(\ff)^{5/4+\epsilon}\,L^{-1/2}\,(K^{(\theta+3)/2}+K^{\theta+9/4+\e}+K^{\theta+5/2+2\e}+K^{15/4+3\e}+K^{4+4\e} +K^{3+\theta} +K^{(2\theta+15)/4+\e} \\
&+K^{9/2+2\e} +K^{19/4+3\e} +K^{(\theta+8)/2+2\e}+K^{5+4\e}). 
\end{align*}
Since $\theta \in [0,1]$, this is bounded by $\nr(\ff)^{1/4+2\epsilon}\,L^{-1/2}\,K^{5+4\epsilon}$. This completes the proof. 
\end{proof}

\subsection{An estimate of the unipotent term}
Set $S=S_{K}^{\gn, \eta}$ with $K\ge 2$.

\begin{prop} \label{MS11}
Let $\pi \in \Pi_{\rm{cus}}(l,\fn)$. For any $\e>0$, we have
\begin{align*}
|\tilde{\JJ}_{\rm{u}}^{\eta}(l, \fn|\alpha_{S}^{\pi})|\ll_{\e}
|\Gcal(\eta)|\nr(\gf)^{\epsilon}\,K^{1+\theta},
\end{align*}
with the implied constant independent of $l$, $\fn$, $\pi$, $\eta$ and $K$. 
\end{prop}
\begin{proof} 
We use the same notation as in the proof of Proposition~\ref{MS10}. By substituting the expression $\alpha_{S}^{\pi}(\bs)=\sum_{v\in S} Z_v^2+\sum_{(v_1,v_2)\in \Sd}Z_{v_1}Z_{v_2}$, we obtain
{\small {\allowdisplaybreaks
\begin{align*}
&|\tilde{\JJ}_{\rm{u}}^{\eta} (l,\fn|\alpha_{S}^{\pi})|
 \\
\ll &
\bC_F^{\eta}(l, \fn )\, 
\bigg( \sum_{v\in S}\{\prod_{w\in S-\{v\}}|U_{w}^{\eta_{w}}(1)|\}\,|U_{v}^{\eta_{v}}(Z_v^2)| +\sum_{(v_1,v_2)\in \Sd}\{\prod_{w\in S-\{v_1,v_2\}}|U_w^{\eta_{w}}(1)|\}\,|U_{v_1}^{\eta_{v_{1}}}(Z_{v_1})||U_{v_2}^{\eta_{v_{2}}}(Z_{v_2})|\bigg)
\\
\ll & \, L_{\fin}(1, \eta)
\biggl(\sum_{v\in S} |U_{v}^{\eta_{v}}(Z_v^2)|+\sum_{(v_1,v_2)\in \Sd}|U_{v_1}^{\eta_{v_{1}}}(Z_{v_1})||U_{v_2}^{\eta_{v_{2}}}(Z_{v_2})|\biggr),
\end{align*}
}}where to simplify the terms, we use $U_w(1)=-1$ from Proposition~\ref{MS5}.
As in the proof of Proposition~\ref{MS6}, using Proposition~\ref{MS5}, we compute each term and estimate it as follows. 
\begin{align*}
U_{v}^{\eta_{v}}(Z_v^2)&=
\lambda_v(\pi)^2\,\{U_{v}^{\eta_{v}}(\alpha_v^{(2)})+U_{v}^{\eta_{v}}(\alpha_v^{(0)})\}
+U_{v}^{\eta_{v}}(\alpha_v^{(4)})+2U_{v}^{\eta_{v}}(\alpha_v^{(2)}) +\frac{3}{2}U_{v}^{\eta_{v}}(\alpha_v^{(0)})
\\
&=\lambda_v(\pi)^2\{(1-q_v^{-1})-2\}+q_v^{-1}(1-q_v^{-1})+2(1-q_v^{-1})-3.
\end{align*}
By $|\lambda_v(\pi)|\ll q_v^{\theta/2}$ with $\theta\in [0,1]$, from this,
\begin{align*}
|U_v^{\eta_{v}}(Z_v^2)|\ll q_v^{\theta}(1+q_v^{-1})+q_{v}^{-2}+q_{v}^{-1}+1
\ll q_{v}^{\theta}.
\end{align*} 
In a similar way, $U_{v}^{\eta_{v}}(Z_v) =q_{v}^{-1}$. Applying these, we continue the estimate of $\tilde{\JJ}_{\rm{u}}^{\eta}(l,\fn|\alpha_{S}^{\pi})$ as follows.
{\allowdisplaybreaks
\begin{align*}
L_{\fin}(1, \eta) \,\big(\sum_{v\in S} q_{v}^\theta+\sum_{(v_1,v_2)\in \Sd
}q_{v}^{-1}q_{v}^{-1} \big)
\ll_{\e}  \,\nr(\gf)^{\e}\, \left\{ \frac{K} {\log K}\, K^{\theta}+\left(\frac{K}{\log K}\right)^{2}K^{-2} \right\}
\ll \, \nr(\gf)^{\e} \,K^{\theta+1}.
\end{align*}
}We remark that $L_\fin(1,\eta)\ll_\epsilon \nr(\ff)^{\epsilon}$ (\cite[Theorem 2]{Carletti}). This completes the proof.
\end{proof}

\subsection{A subconvexity bound (odd case)} 

Let $\fn$ be an ideal of $\cO$. For a family of positive even integers $l=(l_v)_{v\in \Sigma_\infty}$, let $\Pi_{\rm{cus}}^*(l,\fn)$ denote the set of all cuspidal automorphic representations $\pi\cong \bigotimes_{v}\pi_{v}$ of $\PGL_2(\A)$ such that $\gf_{\pi}=\fn$ and such that $\pi_{v}$ is isomorphic to the discrete series representation $D_{l_v}$ of minimal $\bK_{v}^0$-type $l_v$ for each $v\in \Sigma_\infty$.

\begin{thm} \label{Subconvex-thm} 
Let $\eta$ be a quadratic idele class character of $F^\times$ with conductor $\ff$ such that $\eta_v(-1)=-1$ for all $v\in \Sigma_\infty$. Let $\fn$ be an integral ideal relatively prime to $\ff$.
Assume that $l_{v} \ge 6$ for all $v \in \Sigma_{\infty}$.
Then, for any $\epsilon>0$
\begin{align*}
|L_{\fin}(1/2,\pi)\,L_\fin(1/2,\pi\otimes\eta) |\ll_{\epsilon} (\nr(\fn\ff)KL)^{\epsilon}\,\nr(\fn)\,\,(LK^{\theta-1}+\nr(\ff)^{3/4}L^{1/2}K^{3}),
\end{align*}
where $L=\prod_{v\in \Sigma_\infty}l_v$ and with the implied constant independent of $l$, $\fn$, $\eta$, $K \ge 2$ and $\pi\in \Pi_{\rm{cus}}^{*}(l,\fn)$. 
\end{thm}
\begin{proof} Let $\pi\in \Pi_{\rm{cus}}^{*}(l, \fn)$ and let $S=S_{K}^{\gn, \eta}$. By applying Theorem~\ref{RELATIVETRACEFORMULA} for the test function $\alpha_{S}^{\pi}(\bs)$, we have 
\begin{align*}
|C(l,\fn,S)|\,|\sum_{\pi'\in \Pi_{\rm{cus}}(l, \fn)}\II_{\rm cus}^\eta(\pi';l, \fn)\, \alpha_{S}^{\pi}(\nu_{S}(\pi'))|\leq |\tilde{\JJ}_{\rm{u}}^\eta(l, \fn|\alpha_{S}^{\pi} )|+|\JJ_{\rm{hyp}}^\eta(l,\fn|\alpha_{S}^{\pi})|.
\end{align*}
with $C(l,\fn,S)=(-1)^{\# S} 2^{-1}D_F^{-1}\,[\bK_\fin:\bK_0(\fn)]^{-1}\prod_{v\in \Sigma_\infty}2 \pi \Gamma(l_v-1)/\Gamma(l_v/2)^{2}$. From Proposition~\ref{Reg-per} and the non-negativity of $\II_{\rm cus}^\eta(\pi'; l, \fn)/(-1)^{\epsilon(\eta)}\cG(\eta)$ by Lemma \ref{value of PP}, the left-hand side becomes
\begin{align*}
&|C(l,\fn,S)|\,|\cG(\eta)|\,\sum_{\pi'\in \Pi_{\rm{cus}}(l, \fn)} \frac{[\bK_\fin:\bK_0(\ff_{\pi'})]}{\nr(\ff_{\pi'})}
w_{\fn}^{\eta}(\pi')\,\frac{L(1/2,\pi')\,L(1/2,\pi'\otimes \eta)}{L^{S_{\pi'}}(1,\pi';{\rm{Ad}})}\, \alpha_{S}^{\pi}(\nu_{S}(\pi')),
\end{align*} 
which is greater than the summand corresponding to $\pi$ by the non-negativity again. Let us examine the $\pi$-term closely. First, from the explicit formula, $w_{\fn}^\eta(\pi)=1$ for $\ff_\pi=\fn$. Let $A_v(\pi)={\rm{diag}}(z_{v},z_v^{-1})$ be the Satake parameter of our $\pi$. Then, using Lemma
\ref{Principalidealdensity}, we obtain 
\begin{align*}
\alpha_{S}^{\pi}(\nu_{S}(\pi))
=\biggl(\sum_{v\in S} \{(z_v+z_v^{-1})^2-(z_v^{2}+z_{v}^{-2}+1)\}\biggr)^2=(\# S)^2 \gg_{\epsilon} K^{2-\epsilon}. 
\end{align*}
Separating the gamma factors from the $L$-functions, we have
\begin{align*}
|C(l,\fn,S)|\, \frac{L_\infty(1/2,\pi)\,L_\infty(1/2,\pi\otimes \eta)}{L_\infty(1,\pi;{\rm{Ad}})}
\asymp & [\bfK_{\fin} : \bfK_{0}(\gn)]^{-1}\prod_{v \in \Sigma_{\infty}}\frac{2 \pi\, \Gamma(l_{v}-1)}{\Gamma(l_{v}/2)^{2}}
\, \prod_{v\in \Sigma_\infty}\frac{\Gamma_{\C}(l_v/2)^2}{\Gamma_{\C}(l_v)}
\\
\asymp & [\bfK_{\fin} : \bfK_{0}(\gn)]^{-1}\,\prod_{v\in \Sigma_\infty}(l_v-1)^{-1}, 
\end{align*}
where all the implied constants are only dependent on $F$. The remaining factors in the $\pi$-term are easily seen to be bounded from below by a constant independent of $(l,\gn,\pi,\eta)$. Combining the considerations so far, we obtain the estimate
\begin{align}
|\Gcal(\eta)| K^{2-\epsilon}\,\nr(\fn)^{-1}\,L^{-1}\,
\frac{L_{\fin}(1/2,\pi)\,L_\fin(1/2,\pi\otimes \eta)}{L_{\fin}^{S_{\pi}}(1,\pi;{\rm{Ad}})}
\ll_{\epsilon} |\tilde{\JJ}_{\rm{u}}^\eta(l, \fn|\alpha_{S}^{\pi} )|+|\JJ_{\rm{hyp}}^\eta(l,\fn|\alpha_{S}^{\pi})|.
 \label{Subconvex-thm-1}
\end{align}
From Propositions~\ref{MS10} and \ref{MS11}, the right-hand side is estimated by  \begin{align*}
&\ll_{\epsilon} |\Gcal(\eta)| \nr(\ff)^\epsilon\,K^{1+\theta} +\nr(\ff)^{1/4+\epsilon}\,L^{-1/2}K^{5 +\epsilon}.
\end{align*}
To complete the proof, we invoke the bound 
$L_{\fin}^{S_{\pi}}(1,\pi;{\rm{Ad}})\ll_{\epsilon} (\nr(\fn)L)^{\epsilon}$
which is known to hold for a general class of $L$-series (\cite[Theorem 2]{Carletti}).
We remark that $|\cG(\eta)|=D_{F}^{-1/2}\nr(\ff)^{-1/2} \prod_{v\in S(\gf)}(1-q_{v}^{-1})^{-1} \ge D_{F}^{-1/2}\nr(\gf)^{-1/2}$. 
\end{proof}

\begin{thm} \label{Subconvex-Thm}
Let $\eta$ be a quadratic idele class character of $F^\times$ such that $\eta_v(-1)=-1$ for all $v\in \Sigma_\infty$. Let $\fn$ be an integral ideal relatively prime to $\ff$. Assume that $l_{v} \ge 6$ for all $v \in \Sigma_{\infty}$.
Then, for any $\e>0$, 
\begin{align*}
|L_{\fin}(1/2,\pi)\,L_\fin(1/2,\pi\otimes \eta)|\ll_{\epsilon} \nr(\ff)^{3/4+
\epsilon}\,\nr(\fn)^{1+\epsilon}\,\{\prod_{v\in \Sigma_\infty}l_v\}^{(7-\theta)/(8-2\theta)+\epsilon}
\end{align*}
with the implied constant independent of $l$, $\fn$, $\eta$ and $\pi\in \Pi_{\rm{cus}}^*(l,\fn)$.
\end{thm}
\begin{proof}
We apply the estimate in Theorem~\ref{Subconvex-thm} with taking $K$ so that $
LK^{\theta-1}\asymp L^{1/2}K^{3}$, or equivalently $K\asymp L^{1/(8-2\theta)}$. Then, we obtain the desired estimate. 
\end{proof}
If $\theta\in [0,1)$, the estimate in Theorem~\ref{Subconvex-Thm} breaks the convex bound $L_\fin(1/2,\pi)\,L_\fin(1/2,\pi\otimes \eta) \ll_{\epsilon} \{C(\pi)\,C(\pi\otimes \eta)\}^{1/4+\epsilon} \ll (\prod_{v\in \Sigma_\infty}l_v)^{1+\epsilon}$ in the weight aspect with a fixed level $\fn$ and a fixed character $\eta$. To have Theorem~\ref{T4}, we only have to invoke the Ramanujan bound $\theta=0$ (\cite{Blasius}) in Theorem~\ref{Subconvex-Thm}. 

\section*{Acknowledgements}
The first author was supported by Grant-in-Aid for JSPS Fellows (25$\cdot$668).
The second author was supported by Grant-in-Aid for Scientific Research (C) 22540033. 


\medskip
\noindent
{Shingo SUGIYAMA\\
Institute of Mathematics for Industry, Kyushu University, 744, Motooka, Nishi-ku, Fukuoka 819-0395, Japan} \\
	{\it E-mail} : {\tt s-sugiyama@imi.kyushu-u.ac.jp}

\medskip
\noindent
{Masao TSUZUKI \\ Department of Science and Technology, Sophia University, Kioi-cho 7-1 Chiyoda-ku Tokyo, 102-8554, Japan} \\
{\it E-mail} : {\tt m-tsuduk@sophia.ac.jp}

\end{document}